\documentclass[11pt]{amsart}

\usepackage{amsmath,amssymb,amscd,amsthm,graphicx,enumerate}

\usepackage{hyperref}
\usepackage{xypic}
\usepackage{xcolor}
\usepackage{color}
\hypersetup{breaklinks=true}
\hypersetup{
    colorlinks=true,
    linkcolor=red,
    filecolor=magenta,      
    urlcolor=cyan,
    citecolor=green,
}

\numberwithin{equation}{section}
\theoremstyle{plain}

\newcommand{\gra }{ \eta }

\newcommand{\cE}{\mathcal{E}}

\newcommand{\bY}{{\bar Y}}

\newcommand{\bry}{{\bf y}}

  \newcommand{\hD}{\mathfrak{D}}

\newcommand{\vp}{\varphi}

\newcommand{\mm}{\, e^\mu  d\vp   dv}

\newcommand{\collar}{\mathcal{K}}  
\newcommand{\pd}{\partial}

\newcommand{\eps}{\varepsilon}

\newcommand{\be}{\begin{equation}}
\newcommand{\bee}{\begin{equation*}}
\newcommand{\bea}{\begin{equation}\begin{aligned}}
\newcommand{\eea}{\end{aligned}\end{equation}}
\newcommand{\ee}{\end{equation}}
\newcommand{\eee}{\end{equation*}}
\newcommand{\bsp}{\begin{split}}
\newcommand{\esp}{\end{split}}

\newcommand{\R}{{\mathbb R}}

\newtheorem{theorem}{Theorem}[section]
\newtheorem{proposition}[theorem]{Proposition}
\newtheorem{lemma}[theorem]{Lemma}
\newtheorem{corollary}[theorem]{Corollary}

\newtheorem{claim}[theorem]{Claim}

\theoremstyle{remark}

\theoremstyle{definition}
\newtheorem{definition}[theorem]{Definition}

\newcommand{\cC}{\mathcal{C}}
\renewcommand{\cD}{\mathcal{D}}
\newcommand{\cT}{\mathcal{T}}
\renewcommand{\cH}{\mathcal{H}}

\newcommand{\fp}{\mathfrak{p}}

\begin{document}

\title{Classification of bubble-sheet ovals in $\mathbb{R}^{4}$}

\author[B. Choi, P. Daskalopoulos, W. Du, R. Haslhofer, N. Sesum]{Beomjun Choi, Panagiota Daskalopoulos, Wenkui Du, Robert Haslhofer, Natasa Sesum}

\begin{abstract}
In this paper, we prove that any bubble-sheet oval for the mean curvature flow in $\mathbb{R}^4$, up to scaling and rigid motion, either is the $\textrm{O}(2)\times \textrm{O}(2)$-symmetric ancient oval constructed by Hershkovits and the fourth author, or belongs to the one-parameter family of $\mathbb{Z}_2^2\times \textrm{O}(2)$-symmetric ancient ovals constructed by the third and fourth author. In particular, this seems to be the first instance of a classification result for geometric flows that are neither cohomogeneity-one nor selfsimilar.
\end{abstract}

\maketitle

\tableofcontents

\section{Introduction}\label{sec-intro}

In the study of geometric flows it is crucial to understand ancient solutions, i.e. solutions that are defined for all sufficiently negative times. In particular, to capture the formation of singularities one always magnifies the original flow by rescaling by a sequence of factors going to infinity and passes to a limit, and any such blowup limit is an ancient solution.

\subsection{Ancient solutions and the dimension barrier}
We recall that a mean curvature flow $M_t$ is called ancient if it is defined for all $t\ll 0$, and noncollapsed if it is mean-convex and there is an $\alpha>0$ such that every point $p\in M_t$ admits interior and exterior balls of radius at least $\alpha/H(p)$, c.f. \cite{ShengWang,Andrews_noncollapsing,HaslhoferKleiner_meanconvex} (in fact, by \cite{Brendle_inscribed,HK_inscribed} one can always take $\alpha=1$). It is known that all blowup limits of mean-convex mean curvature flow are ancient noncollapsed flows, see \cite{White_size,White_nature,White_subsequent,HH_subsequent}.  More generally, by Ilmanen's mean-convex neighborhood conjecture \cite{Ilmanen_problems}, which has been proved recently in the case of neck-singularities in \cite{CHH,CHHW}, it is expected that for mean curvature flow starting at any closed embedded hypersurface any blow up limit near any cylindrical singularity is in fact an ancient noncollapsed flow. In particular, by \cite{CM_generic,CCMS} it is expected that any blowup limit near any generic singularity is an ancient noncollapsed flow.

For ancient noncollapsed flows in $\mathbb{R}^3$, or more generally in $\mathbb{R}^{n+1}$ under the additional assumption that the flow is uniformly two-convex, a complete classification has been obtained in significant works by Brendle-Choi \cite{BC1,BC2} and Angenent and the second and the fifth author \cite{ADS1,ADS2}.  Specifically, any such flow is, up to parabolic rescaling and space-time rigid motion, either the flat plane, the round shrinking sphere, the round shrinking neck, the rotationally symmetric translating bowl soliton from \cite{AltschulerWu}, or the rotationally symmetric ancient oval from \cite{White_nature,HaslhoferHershkovits_ancient}. Ultimately, what made this classification possible is that at the end of the day all solutions turned out to be rotationally symmetric.

In stark contrast, in higher dimensions without two-convexity assumption there are multi-parameter families of examples of ancient noncollapsed flows that are not rotationally symmetric, and not even cohomogeneity-one, see \cite{Wang_convex,HIMW,DH_ovals}. For this reason, the classification of ancient noncollapsed flows in higher  dimensions without two-convexity assumption until recently seemed out of reach. As a parallel story, ancient $\kappa$-noncollapsed 3d Ricci flows have been classified in \cite{Brendle_Bryant,ABDS,BDS}, with an extension to higher dimensions under the additional PIC2 assumption in \cite{LiZhang,BN,BDNS}, but in light of examples from \cite{Lai} the classification of general ancient $\kappa$-noncollapsed Ricci flows in higher dimensions has remained widely open.

\subsection{The classification program in $\mathbb{R}^4$}\label{subsec_class_prog}

In a recent paper \cite{DH_hearing_shape}, with the aim of overcoming the dimension barrier discussed above, the third and fourth author introduced a classification program for ancient noncollapsed  flows in $\mathbb{R}^4$. To describe this, recall first that if $M_t$ is an ancient noncollapsed flow in $\mathbb{R}^4$, then its tangent flow at $-\infty$ is always either a round shrinking sphere, a round shrinking neck, a round shrinking bubble-sheet, or a static plane. The first and last scenario are of course trivial, and ancient noncollapsed flows whose tangent flow at $-\infty$ is a neck have been classified in \cite{ADS1,ADS2,BC1,BC2}, as discussed above. We can thus assume from now on that the tangent flow at $-\infty$ 
 is a bubble-sheet, specifically
\begin{equation}\label{bubble-sheet_tangent_intro}
\lim_{\lambda \rightarrow 0} \lambda M_{\lambda^{-2}t}=\mathbb{R}^{2}\times S^{1}(\sqrt{2|t|}).
\end{equation}
Writing the renormalized flow $\bar M_\tau = e^{\frac{\tau}{2}}  M_{-e^{-\tau}}$ as a graph of a function $u(\cdot,\tau)$ over increasing domains exhausting $\Gamma=\mathbb{R}^2\times S^{1}(\sqrt{2})$, namely
\begin{equation}\label{graph_over_cylinder}
\left\{ q+ u(q,\tau)\nu(q) \, : \, q\in \Gamma\cap B_{\rho(\tau)} \right\} \subset \bar M_\tau\, ,
\end{equation}
where $\nu$ is the outward unit normal of $\Gamma$, and $\lim_{\tau\to -\infty}\rho(\tau)= \infty$, we recall:
\begin{theorem}[{bubble-sheet quantization, \cite{DH_hearing_shape,DH_no_rotation}}]\label{thm_intro_bubblesheetquant}
For any ancient noncollapsed mean curvature flow in $\mathbb{R}^4$, whose tangent flow at $-\infty$ is given by \eqref{bubble-sheet_tangent_intro}, the bubble-sheet function $u$ satisfies
 \begin{equation}\label{main_thm_ancient}
\lim_{\tau\to -\infty} \Big\|\,
|\tau| u({\bf y},\vartheta,\tau)- {\bf y}^\top Q{\bf y} +2\mathrm{tr}(Q)\, \Big\|_{C^{k}(B_R)} = 0
\end{equation}
for all $R<\infty$ and all integers $k$, where $Q$ is a symmetric $2\times 2$-matrix whose eigenvalues are quantized to be either 0 or $-1/\sqrt{8}$. 
\end{theorem}

In particular, the theorem uniquely associates to the flow a symmetric $2\times 2$-matrix $Q$, whose eigenvalues are quantized to be either $0$ or $-1/\sqrt{8}$, so that in the region with bounded ${\bf y}=(y_1,y_2)$ one has the expansion
 \begin{equation}
u({\bf y},\vartheta,\tau)= \frac{{\bf y}^\top Q{\bf y} -2\textrm{tr}(Q)}{|\tau|} + o(|\tau|^{-1}).
\end{equation}
Intuitively, directions in the range of $Q$ are short directions with inwards quadratic bending, while directions in the kernel of $Q$ are long directions.

Thanks to Theorem \ref{thm_intro_bubblesheetquant} (bubble-sheet quantization) the problem of classifying general ancient noncollapsed flows in $\mathbb{R}^4$ can be naturally divided into three cases according to the rank of the bubble-sheet matrix $Q$.

In the case $\mathrm{rk}(Q)=0$ it has been shown in \cite[Theorem 1.2]{DH_hearing_shape}, as a consequence of the no-ancient-wings theorem from \cite{CHH_wing}, that the flow must be either a round shrinking $\mathbb{R}^2\times S^1$ or a translating $\mathbb{R}\times$2d-bowl.

In the case $\mathrm{rk}(Q)=1$ it has been shown in \cite[Theorem 1.3]{DH_hearing_shape}, as a consequence of \cite{ADS1,ADS2,BC1,CHH_translator}, that under the additional assumption that the flow either splits off a line or is selfsimilar, it is $\mathbb{R}\times$2d-oval or belongs to the one-parameter family of noncollapsed translators constructed by Hoffman-Ilmanen-Martin-White \cite{HIMW}, respectively. The general  $\mathrm{rk}(Q)=1$ case without additional assumptions will be addressed in forthcoming work  by Choi, Hershkovits and the fourth author \cite{CHH_ancient}.

In the present paper, we are concerned with the case  $\mathrm{rk}(Q)=2$, also known as the bubble-sheet oval case:

\begin{definition}[bubble-sheet oval]\label{def_bubble_sheet_oval}
A \emph{bubble-sheet oval} in $\mathbb{R}^{4}$ is an ancient noncollapsed mean curvature flow in $\mathbb{R}^{4}$, whose tangent flow at $-\infty$ is given by \eqref{bubble-sheet_tangent_intro} and whose bubble-sheet matrix $Q$ has $\textrm{rk}(Q)=2$.
\end{definition}

We recall that by \cite[Theorem 1.4]{DH_hearing_shape} all bubble-sheet ovals are compact. Also, since they become extinct as a round point, but the tangent flow at $-\infty$ is a bubble-sheet, they are obviously not selfsimilar. In addition to this lack of selfsimilarity, a second major difficulty in classifying bubble-sheet ovals is, as we will review momentarily, the presence of a whole one-parameter family of examples that are only $\mathbb{Z}_2^2\times \textrm{O}(2)$-symmetric.

\subsection{Ancient ovals}

We recall that an ancient noncollapsed flow is called an \emph{ancient oval} if its time-slices are compact but not round.

Ancient ovals play an important role as potential \emph{compact} singularity models in mean-convex flows \cite{White_size,White_nature,HaslhoferKleiner_meanconvex}. Moreover, the two-convex ancient ovals, whose uniqueness up to rigid motion and dilation has been established in the recent work by Angenent, the second and fifth author \cite{ADS1,ADS2}, appeared as potential blowup limits in the recent proof of the mean-convex neighborhood conjecture \cite{CHH,CHHW}. Furthermore, ancient ovals are tightly related to the fine structure of singularities, including questions about accumulation of neckpinch singularities and finiteness of singular times \cite{CM_arrival,CHH_ovals}. Finally, ancient ovals are in fact also of key importance for the analysis of \emph{noncompact} singularities, where they describe the asymptotic shape of the level sets of translators \cite{CHH_translator}.

The existence of ancient ovals has been proved first by White \cite{White_nature}. Later, Hershkovits and the second author carried out White's construction in more detail \cite{HaslhoferHershkovits_ancient}, which in particular yielded $\mathrm{O}(k)\times \mathrm{O}(n+1-k)$-symmetric ancient ovals in $\mathbb{R}^{n+1}$ for every $1\leq k\leq n$. 
Recently, in \cite{DH_ovals} the third and fourth author proved uniqueness, up to rigid motion and dilation, among $\mathrm{SO}(k)\times \mathrm{SO}(n+1-k)$-symmetric solutions. Furthermore, for $t\to -\infty$ these cohomogeneity-one solutions are asymptotic to small perturbations of ellipsoids with $k$ long axes of length $\sqrt{2|t|\log|t|}$ and $n-k$ short axes of length $\sqrt{2(n-k)|t|}$. In particular, for $(n,k)=(3,2)$ this gives an $\textrm{O}(2)\times \textrm{O}(2)$-symmetric ancient oval in $\mathbb{R}^4$, unique up to rigid motion and dilation, which is the simplest example of a bubble-sheet oval in $\mathbb{R}^4$.

On the other hand, in the same paper \cite{DH_ovals} the third and fourth author constructed a whole one-parameter family $\mathcal{A}^{\circ}$ of ancient ovals in $\mathbb{R}^4$ that are only $\mathbb{Z}_2^2\times \textrm{O}(2)$-symmetric.\footnote{More generally, for every $k \geq 2$ the construction gives a $(k-1)$-parameter family of uniformly $(k+1)$-convex ancient ovals in $\mathbb{R}^{n+1}$ that are only $\mathbb{Z}_2^k\times\mathrm{O}(n + 1 -k)$-symmetric.} Intuitively, this family interpolates between $\mathbb{R}\times$2d-oval and 2d-oval$\times\mathbb{R}$, and the $\textrm{O}(2)\times \textrm{O}(2)$-symmetric ancient oval from above sits in the middle of the family. Specifically, to construct this family, for any  $a\in (0, 1)$ and any $\ell<\infty$ one considers the ellipsoid
\begin{equation}
E^{\ell,a}:=\left\{ x\in \mathbb{R}^{4} \, : \, \frac{a^2}{\ell^2} x_1^2 +\frac{(1-a)^2}{\ell^2} x_2^2 + x_3^2+x_4^2 = 2 \right\}\, .
\end{equation}
One then chooses time-shifts $t_{\ell,a}$ and dilation factors $\lambda_{\ell,a}$ so that the flow
\begin{equation}\label{m_ell_a}
M^{\ell,a}_t := \lambda_{\ell,a} \cdot E^{\ell,a}_{\lambda^{-2}_{\ell,a} t+t_{\ell,a}}
\end{equation}
becomes extinct at time $0$ and satisfies 
\begin{equation}
\int_{M_{-1}^{\ell,a}}\frac{1}{(4\pi)^{3/2}}e^{-\frac{|x|^2}{4}}=\tfrac{1}{2}\left(\tfrac{4}{e}+\sqrt{\tfrac{2\pi}{e}}\right).
\end{equation}
Considering sequences $a_i\in (0, 1)$  and $\ell_i\to \infty$, the class of examples is then defined by
\begin{equation}\label{Ao}
\mathcal{A}^{\circ}:=\left\{ \lim_{i\to \infty} M^{\ell_i,a_i}_t : \textrm{the limit along $a_i,\ell_i$ exists and is compact}\right\}\, .
\end{equation}
Of course, this class contains as a special case the $\textrm{O}(2)\times \textrm{O}(2)$-symmetric oval from \cite{HaslhoferHershkovits_ancient}.
It has been shown in \cite[Theorem 1.9]{DH_ovals} that all elements of the class $\mathcal{A}^{\circ}$ are $\mathbb{Z}_2^2\times \textrm{O}(2)$-symmetric, ancient, noncollapsed, and with  tangent flow at $-\infty$ given by \eqref{bubble-sheet_tangent_intro}, and that for any $\mu\in(0,1)$ there exists an $M_t\in\mathcal{A}^{\circ}$ whose reciprocal width ratio satisfies
\begin{equation}
\frac{(\max_{x\in M_{-1}}|x_1|)^{-1}}{(\max_{x\in M_{-1}}|x_1|)^{-1}+(\max_{x\in M_{-1}}|x_2|)^{-1}}=\mu.
\end{equation}
Let us also point out that under the $\mathbb{Z}_2$-symmetry that swaps the $x_1$ and $x_2$ coordinate, $\mu$ gets of course mapped to $1-\mu$, and correspondingly one could consider the quotient class $\mathcal{A}^\circ/\mathbb{Z}_2$. However, for our purpose it is most convenient to work with the slightly redundant description as class $\mathcal{A}^\circ$.

\subsection{Main result and consequences}

Our main result classifies all bubble-sheet ovals for the mean curvature flow in $\mathbb{R}^4$ (see Definition \ref{def_bubble_sheet_oval}):

\begin{theorem}[classification of bubble-sheet ovals]\label{classification_theorem}
Any bubble-sheet oval in $\mathbb{R}^{4}$  belongs, up to space-time rigid motion and parabolic dilation, to the class  $\mathcal{A}^{\circ}$.
\end{theorem}

The most important feature of Theorem \ref{classification_theorem} (classification of bubble-sheet ovals), in contrast to all prior classification results for geometric flows in the literature, is that it provides a classification result for ancient flows that are neither cohomogeneity-one nor selfsimilar. In particular, recall that all ancient flows that arise as blowup limits near neck-singularities at the end of the day turned out to be rotational symmetric  \cite{CHH,CHHW}, and thus their time slices can be described by a single spatial variable. In contrast, the bubble-sheet ovals from our main theorem genuinely depend on 3 variables, namely the two spatial variables $y_1,y_2$ and the time variable $\tau$.

Moreover, Theorem \ref{classification_theorem}  (classification of bubble-sheet ovals) completes the classification program for ancient noncollapsed flows in $\mathbb{R}^4$, as introduced in \cite{DH_hearing_shape} and reviewed in Section \ref{subsec_class_prog}, in the case $\textrm{rk}(Q)=2$. In particular, as a corollary, together with the  prior results reviewed above, we obtain:

\begin{corollary}[blowup limits]\label{cor_blowup_limits}
For mean curvature flow of closed embedded mean-convex hypersurfaces in $\mathbb{R}^4$ (or more generally in any 4-manifold) any blowup limit, up to parabolic rescaling and space-time rigid motion, is
\begin{itemize}
\item either one of the shrinkers $S^3$, $\mathbb{R}\times S^2$, $\mathbb{R}^2\times S^1$ or $\mathbb{R}^3$,
\item or the 3d-bowl, or $\mathbb{R}\times$2d-bowl, or is a strictly convex ancient noncollpased flow whose tangent flow at $-\infty$ is a bubble-sheet and whose bubble-sheet matrix $Q$ has $\mathrm{rk}(Q)=1$ (such as the one-parameter family of $\mathbb{Z}_2\times \mathrm{O}(2)$-symmetric translators from \cite{HIMW}).
\item or the rotationally symmetric 3d-oval from \cite{White_nature}, or $\mathbb{R}\times$2d-oval, or the $\mathrm{O}(2)\times \mathrm{O}(2)$-symmetric 3d-oval from \cite{HaslhoferHershkovits_ancient},
or belongs to the one-parameter family of $\mathbb{Z}_2^2\times \mathrm{O}(2)$-symmetric 3d-ovals from \cite{DH_ovals}.
\end{itemize}
\end{corollary}

Corollary \ref{cor_blowup_limits} (blowup limits) provides the first general classification result (with the caveat that the remaining case $\textrm{rk}(Q)=1$ will be addressed in forthcoming work \cite{CHH_ancient}) of blowup limits in higher dimensions without two-convexity assumption. Moreover, Corollary \ref{cor_blowup_limits} (blowup limits) also suggests a corresponding conjectural picture for $\kappa$-solutions in 4d Ricci flow. This will be discussed in forthcoming work of the fourth author \cite{Haslhofer_4dRicci}.

Finally, one naturally wonders how all these ancient solutions fit together in a global picture. Specifically, we can consider the moduli space
\begin{multline}
\mathcal{X}=\big\{ \mathcal{M} \textrm{ is a compact ancient noncollapsed flow in $\mathbb{R}^4$,}\\
\textrm{ whose tangent flow at $-\infty$ is a bubble-sheet} \big\}/\sim,
\end{multline}
where the topology is the one induced by locally smooth convergence, and we mod out by space-time rigid motions and parabolic dilation. Equivalently, in terms of the renormalized mean curvature flow, $\mathcal{X}$ can be described as the space of all connecting orbits between $\mathbb{R}^2\times S^1(\sqrt{2})$ and $S^3(\sqrt{6})$. 

As a consequence of Theorem \ref{classification_theorem} (classification of bubble-sheet ovals) in combination with the forthcoming work \cite{CHH_ancient}, we obtain:
\begin{corollary}[moduli space]\label{cor_moduli}
The moduli space $\mathcal{X}$ is homeomorphic to a half-open interval. More precisely, we have the homeomorphisms
\begin{equation}
\mathcal{X}\cong \mathcal{A}^\circ / \mathbb{Z}_2\cong [0,1).
\end{equation}
\end{corollary}

In Corollary \ref{cor_moduli} (moduli space) the endpoint $0$ of course corresponds to the $\mathrm{O}(2)\times \mathrm{O}(2)$-symmetric 3d oval from \cite{HaslhoferHershkovits_ancient}, while the endpoint $1$ corresponds to the product of the 2d oval with a line. Equivalently, in terms of the renormalized mean curvature flow the endpoint $0$ corresponds to the $\mathrm{O}(2) \times \mathrm{O}(2)$-symmetric connecting orbit from $\mathbb{R}^2\times S^1(\sqrt{2})$ to $S^3(\sqrt{6})$, while as one degenerates to the endpoint $1$ the renormalized flow lingers around for longer and longer times in the vicinity of the fixed point $\mathbb{R}\times S^2(\sqrt{4})$. For a more detailed discussion of this we refer to forthcoming work by Angenent, the second and the fifth author  \cite{ADS_space}, where the space of ancient noncollapsed mean curvature flows will be discussed from a more general perspective.

\subsection{Related prior work}\label{sec_prior_work} In this subsection, as a motivation for our approach, let us review some ideas from the papers \cite{ADS1,ADS2,CHH_translator}, which are most closely related to our current problem.\footnote{As an historical aside, let us mention that the classification of closed ancient solutions was initiated by Hamilton, the second and fifth author \cite{DHS}. Also, there has been some interesting related work in the collapsed case, see e.g. \cite{BLT1,BLT2}.}

As already mentioned several times, ancient ovals have been classified in the uniformly two-convex case:

\begin{theorem}[{two-convex ancient ovals, \cite{ADS1,ADS2}}]\label{thm_intro_2convex_ovals}
Any uniformly two-convex ancient oval in $\mathbb{R}^{n+1}$ agrees, up to rigid motion and scaling, with the rotationally symmetric ancient oval constructed by White. 
\end{theorem}

We recall that the first paper \cite{ADS1} establishes sharp asymptotics, and the second paper \cite{ADS2} upgrades the sharp asymptotics to uniqueness.

To relate with the terminology of Section \ref{subsec_class_prog}, note that by \cite{HaslhoferKleiner_meanconvex} for ancient ovals the assumption of being uniformly two-convex is equivalent to the assumption that the tangent flow at $-\infty$ is a neck, namely
\begin{equation}\label{neck_tangent_intro}
\lim_{\lambda \rightarrow 0} \lambda M_{\lambda^{-2}t}=\mathbb{R}\times S^{n-1}(\sqrt{2(n-1)|t|}).
\end{equation}
Hence, uniformly two-convex ancient ovals could simply be called neck-ovals. Also, for ancient noncollapsed flows whose tangent flow at $-\infty$ is a neck, the matrix $Q$ is just a $1\times 1$-matrix, i.e. a number, that is either $0$ or $-\sqrt{(n-1)/8}$. The neck-oval case is the case where $Q=-\sqrt{(n-1)/8}$ (in case $Q=0$, the solution is noncompact, and thus either a round neck or the rotationally symmetric bowl soliton, as shown in \cite{BC1,BC2}).

Neck-ovals, thanks to the $\mathrm{SO}(n)$-symmetry, which has been established using the neck-improvement theorem from \cite{BC1,BC2}, can be described by a renormalized profile function $v(y,\tau)$ depending only on a single spatial variable $y=y_1$, namely one can express the renormalized ovals as
\begin{equation}
\bar{M}_\tau=\left\{ (y,y_2,\ldots,y_{n+1})\in\mathbb{R}^{n+1} : |(y_2,\ldots,y_{n+1})|=v(y, \tau) \right\}.
\end{equation}
The uniqueness proof in \cite{ADS2} is based on energy estimates. To this end, it is useful to split up the ovals into a cylindrical region, where $v$ is larger than some small fixed number, and a tip region. In the cylindrical region, the evolution of $v$ is governed by the 1d Ornstein-Uhlenbeck operator,
\begin{equation}
\mathcal{L}_{1}=\partial_y^2-\tfrac{y}{2}\partial_y+1.
\end{equation}
In the tip region, one instead works with the inverse profile function $Y=Y(v,\tau)$. The tip region in turn can be split into a soliton region, where a zoomed-in version of $Y$ is close to the translating bowl soliton, and a collar region, where one transitions between soliton and neck behaviour. The bulk of  \cite{ADS2} then derives energy estimates in carefully weighed function spaces for the profile functions in the respective regions. Moreover, it is shown via the maximum principle that the function $v^2$ is concave, which is the key a priori estimate for dealing with the collar region. Finally, given two neck-ovals one has to arrange that the difference of their profile functions is orthogonal to the unstable and neutral eigenfunctions of $\mathcal{L}_1$, but this can be accomplished easily by suitable rigid motion and parabolic rescaling.\footnote{The uniqueness of $\mathrm{O}(k)\times \mathrm{O}(n+1-k)$-symmetric ovals in \cite{DH_ovals} has been established in a similar spirit as above, thanks to the cohomogeneity-one property.}

The other most directly related prior work is the recent classification of noncollapsed translators in $\mathbb{R}^4$ by Choi, Hershkovits and the fourth author:
\begin{theorem}[{translators, \cite{CHH_translator}}]\label{thm_intro_transl}
Every noncollapsed translator $M\subset\mathbb{R}^4$ is either $\mathbb{R}\times$2d-bowl,  or the 3d round bowl,  or belongs to the one-parameter family of 3d oval-bowls constructed by Hoffman-Ilmanen-Martin-White.
\end{theorem}

Let us recall the construction of the translators from \cite{HIMW}. For any ellipsoidal parameter $a\in [0,\frac{1}{3}]$ and any height $h<\infty$, let $M^{a,h}$ be the $\mathrm{SO}(2)$-symmetric translator-with-boundary with tip at the origin  and whose boundary lies at height $x_1=h$ and is an ellipse of the form
$a^2 x_2^2 + (\tfrac{1-a}{2})^2 x_3^2 + (\tfrac{1-a}{2})^2 x_4^2= R^2$,
where $R=R(a,h)$. The Hoffman-Ilmanen-Martin-White class is then defined as the collection of all possible limits, namely
\begin{equation}
\mathcal{A}_{\mathrm{HIMW}}:= \left\{ \lim_{i\to \infty} M^{a_i,h_i} \, |\,  a_i \in [0,1/3] \textrm{ and } h_i\to \infty \right\}.
\end{equation}

To sketch the main steps of the proof from \cite{CHH_translator}, given a noncollapsed translator $M\subset\mathbb{R}^4$, that is neither $\mathbb{R}\times$2d-bowl nor 3d-bowl, we normalize without loss of generality such that ${\bf{H}}=e_4^\perp$. By the no-wings theorem from \cite{CHH_wing}, the spatial blowdown is a ray, more precisely
\begin{equation}
\lim_{\lambda\to 0} \lambda M= \{\mu e_4 | \mu\geq 0\},
\end{equation}
which in particular yields $\mathrm{SO}(2)$-symmetry \cite{Zhu}. Hence, the level sets,
\begin{equation}
\Sigma^h=M\cap \{x_4 = h\},
\end{equation}
can be described by a renormalized profile function $v(y,\tau)$, where $\tau=-\log h$, whose analysis is governed by the 1d Ornstein-Uhlenbeck operator $\mathcal{L}_1$ from above. Using this, it has been shown that $v$ satisifies the same sharp asymptotics as the 2d ancient ovals in $\mathbb{R}^3$. An important technical point is that these estimates are uniform for certain one-parameter families. Specifically, they only depend on a parameter $\kappa$, which captures the error in the inwards quadratic bending in the central region. The key analytic step is then to establish a spectral uniqueness theorem, which says that if for two (suitably normalized) translators the difference of the profile functions $v_1-v_2$ is perpendicular to the unstable and neutral eigenspace of $\mathcal{L}_1$, then the translators agree. Finally,  this spectral condition is arranged via a continuity argument in the ellipsoidal parameter $a$ from the HIMW-construction. This continuity argument in turn relies on a Rado-type argument, which relates the parameter $a$ with the smallest principle curvature $k$ at the tip.

\subsection{Major new challenges} In this subsection, we describe some major new challenges that we face in our classification of bubble-sheet ovals.

Generally speaking, recall that our bubble-sheet oval problem genuinely depends on 3 independent variables, namely the 2 spatial variables $(y_1,y_2)$ and the time variable $\tau$. This is obviously more complex than all prior problems. In particular, the classification from \cite{ADS1,ADS2, DH_ovals} crucially relied on the fact that thanks to the cohomogeneity-one property there is only 1 spatial variable, and the classification from \cite{CHH_translator} crucially relied on the fact that thanks to the selfsimilarity assumption there is no time-dependence. More specifically, our bubble-sheet ovals can be described by a renormalized profile function $v=v(y, \varphi, \tau)$ in polar coordinates, defined via
\begin{equation}\label{def_profile_fn}
\bar{M}_\tau=\Big\{ (y\cos\varphi, y\sin\varphi,y_3,y_4) : |(y_3,y_4)|=v(y, \varphi, \tau) \Big\}.
\end{equation}
As we will see, the function $v$ evolves by
{\begin{equation}
\label{eq-u-polar}
\begin{split}
v_{\tau} =& \frac{(y^2+v_{\varphi}^2)v_{yy} - 2v_y v_{\varphi} v_{\varphi y} + (1 + v_y^2)v_{\varphi\varphi}}{ y^2\, (1 + v_y^2)+v_{\varphi}^2} \,\\
&+ \left(\frac{2}{y^2}  - \frac{1+v_y^2}{  y^2(1+v_y^2)+v_{\varphi}^2}- \frac 12\right) y v_y+ \frac v2 - \frac 1v,
\end{split}
\end{equation}
and the inverse function $Y$, defined via
\begin{equation}\label{Y_inverse profile}
   y = Y(v(y,\varphi,\tau),\varphi,\tau) ,
\end{equation}
evolves by
\begin{equation}
\label{eq-f-polar}
\begin{split}
Y_{\tau} = &\frac{(Y^2 + Y_{\varphi}^2)Y_{vv}  - 2Y_{\varphi}Y_v Y_{v\varphi} + (1 + Y_v^2)Y_{\varphi\varphi}}{ Y^2\, (1 + Y_v^2)+Y_{\varphi}^2}   \\
&+  \left(\frac 1v - \frac v2\right)Y_v - \frac{Y_{\varphi}^2}{Y\,(  Y^2\, (1 + Y_v^2)+Y_{\varphi}^2)} + \frac{Y}{2} - \frac 1Y.
\end{split} 
\end{equation} 
The analysis of these equations is of course quite a bit more involved than the one of the 1-variable profile functions in Section \ref{sec_prior_work}.

Arguably the biggest new challenge is the quadratic concavity estimate. Recall that in \cite{ADS2} it has been shown that the function $y\mapsto v(y,\tau)^2$ is concave, which is the crucial a priori estimate for connecting the neck behaviour in the cylindrical region and the translator behaviour in the soliton region. The quadratic concavity estimate has been proved by applying the maximum principle to the 1-variable function $(v^2)_{yy}$. Similar arguments went trough in \cite{DH_ovals}, thanks to the symmetry, and in \cite{CHH_translator}, up to exponentially small errors, thanks to the selfsimilarity. In stark contrast, in our setting where the evolution depends on many angular terms, this 1-variable maximum principle argument completely breaks down.

Another major new challenge is dealing with the unstable and neutral eigenfunctions of the 2d Ornstein-Uhlenbeck operator
\begin{equation}
\mathcal{L}_{2}=\partial^2_{y}+\tfrac{1}{y}\partial_y+\tfrac{1}{y^2}\partial^2_\varphi-\tfrac{1}{2}y\partial_{y} + 1.
\end{equation}
In \cite{ADS2} the contribution from the unstable and neutral eigenfunctions of $\mathcal{L}_1$ could simply be killed via rigid motions and parabolic rescaling. This is not possible any more in our setting, where the space of unstable and neutral eigenfunctions of $\mathcal{L}_2$ is 6-dimensional, but there are only 5 degrees of freedom from rigid motions and rescaling. Of course this must be the case, since we are now classifying a genuine 1 parameter family of solutions. A similar phenomenon has already been encountered in \cite{CHH_translator}, and has been dealt with via a continuity argument in the ellipsoidal parameter $a$, as reviewed above. However, this continuity argument in turn built on a Rado-type argument and on a continuous recentering method, which both crucially relied on the fact that there were only 2 independent variables. In our current setting with 3 independent variables, the Rado-type argument completely breaks down. Moreover, the standard recentering argument from the literature, see e.g. \cite{ADS2,DH_ovals}, is not applicable either, since it is based on degree theory, which only gives existence and no uniqueness.

\subsection{Outline and intermediate results}
\label{sec-outline}
Motivated by the above discussion, we break up our argument into the following four steps:
\begin{itemize}
\item uniform sharp asymptotics,
\item quadratic almost concavity,
\item spectral uniqueness,
\item from spectral uniqueness to classification.
\end{itemize}

We will now first discuss a few generalities that are used throughout this paper, and then discuss these four steps in turn.

\bigskip

We recall from \cite[Theorem 1.4]{DH_hearing_shape} that all bubble-sheet ovals in $\mathbb{R}^4$ are $\textrm{SO}(2)$-symmetric. Remembering also the uniqueness of tangent flows from \cite{CM_uniqueness}, we can thus assume throughout the paper that the tangent flow at $-\infty$ is given by \eqref{bubble-sheet_tangent_intro}, and that the $\textrm{SO}(2)$-symmetry is in the $x_3x_4$-plane centered at  the origin. We also recall that by \cite{HaslhoferKleiner_meanconvex} every ancient noncollapsed flow is convex. Instead of working with the ancient flow $M_t$ itself, it will usually be more convenient to consider the renormalized flow $\bar{M}_\tau = e^{\tau/2}M_{-e^{-\tau}}$.
In particular, under this correspondence the round shrinking bubble-sheet $\mathbb{R}^2\times S^1(\sqrt{2|t|})$ simply becomes the static bubble-sheet
$\mathbb{R}^2\times S^1(\sqrt{2})$.

Using the above setting, any bubble-sheet oval in $\mathbb{R}^4$ can be described by a renormalized profile function $v=v(y,\varphi,\tau)$ as defined in \eqref{def_profile_fn}. Equivalently, this means that we can parametrize the renormalized flow $\bar{M}_\tau$ as
\begin{equation}\label{par_ren_mcf}
(y,\varphi,\theta)\mapsto \big(y\cos\varphi,y\sin\varphi,v(y,\varphi,\tau)\cos \vartheta,v(y,\varphi,\tau)\sin \vartheta\big).
\end{equation}
As mentioned, it will be useful to consider different regions. Specifically, fixing a small constant $\theta>0$ and a large constant $L<\infty$, we consider the \emph{cylindrical region}
\begin{equation}
\label{eq-cyl-region}
\cC= \big\{ v \ge \theta  \big\},
\end{equation}
and the \emph{tip region}
\begin{equation}
\label{eq-tip-region}
\mathcal{T}= \big\{v \le 2   \theta \big\},
\end{equation}
where we further subdivide the latter into  the \emph{collar region}
\begin{equation}
\label{eq-collar-region}
\collar = \bigl\{  L/\sqrt{|\tau|} \le v \le 2  \theta \bigr\},
\end{equation}
and the \emph{soliton region}
\begin{equation}
\label{eq-soliton-region}
\mathcal{S} = \bigl\{v \le L/\sqrt{|\tau|} \bigr\}.
\end{equation}
Linearizing the renormalized mean curvature flow \eqref{par_ren_mcf} around the static bubble-sheet
$\mathbb{R}^2\times S^1(\sqrt{2})$, i.e. around the steady state $v=\sqrt{2}$, one sees that the evolution is governed by the 2d Ornstein-Uhlenbeck operator
\begin{equation}
\mathcal{L}=\partial^2_{y}+\tfrac{1}{y}\partial_y+\tfrac{1}{y^2}\partial^2_\varphi-\tfrac{1}{2}y\partial_{y} + 1.
\end{equation}
This operator is self-adjoint on the Gaussian $L^2$-space
\begin{equation}
\mathcal{H}=L^2\big(\mathbb{R}^2,e^{-y^2/4} \, y d\varphi \, dy\big)= \mathcal{H}_+\oplus \mathcal{H}_0\oplus \mathcal{H}_-,
\end{equation}
where the unstable and neutral eigenspace are explicitly given by
\begin{equation}
    \mathcal{H}_+=\textrm{span}\big\{1, y\cos\varphi, y\sin\varphi\big\}, 
\end{equation}
and
\begin{equation}
\mathcal{H}_0=\textrm{span}\big\{y^2-4, y^2\cos(2\varphi), y^2\sin(2\varphi)\big\}.
\end{equation}
We denote the orthogonal projections to $\mathcal{H}_0$ and $\mathcal{H}_\pm$ by $\mathfrak{p}_{0}$ and $\mathfrak{p}_{\pm}$, respectively. Moreover, to localize in the cylindrical region we set
\begin{equation}\label{cut-off}
    v_{\cC}=v\chi_{\cC}(v),
\end{equation}
where $\chi_{\cC}:\mathbb{R}_{+}\to [0,1]$ is a cutoff function that satisfies $\chi_{\cC}(v)=1$ for $v\geq \tfrac 78  \theta$ and $\chi_{\cC}(v)=0$ for $v\leq  \tfrac 58  \theta$.

\bigskip

In Section \ref{Uniform sharp asymptotics}, we establish uniform sharp asymptotics. We note that in \cite[Theorem 1.4]{DH_hearing_shape} it has already been shown that all bubble-sheet ovals in $\mathbb{R}^4$ satisfy the same sharp asymptotics as the $\textrm{O}(2)\times \textrm{O}(2)$-symmetric ancient oval from \cite{HaslhoferHershkovits_ancient}. However, for the application in the continuity method it will be crucial to have uniform estimates that only depend on the precision of the inwards-quadratic bending at a single time. 
To capture this, motivated by \cite[Definition 1.4]{CHH_translator}, we consider the following notion:

\begin{definition}[$\kappa$-quadratic]\label{k_tau00}
A bubble-sheet oval in $\mathbb{R}^{4}$ 
is called \emph{$\kappa$-quadratic at time $\tau_{0}$} if its truncated renormalized profile function $v_{\cC}=v\chi_{\cC}(v)$ satisfies 
\begin{equation}\label{condition1}
    \left\|v_{\cC}(y,\varphi, \tau_{0})-\sqrt{2}+\frac{y^2-4}{\sqrt{8}|\tau_{0}|}  \right\|_{\mathcal{H}}\leq \frac{\kappa}{|\tau_{0}|},
\end{equation}
and moreover the following centering and graphical radius conditions hold:
\begin{equation}\label{condition2}
    \mathfrak{p}_{+}\big(v_{\cC}(\tau_{0})-\sqrt{2}\big)=0,\,    \sup_{\tau\in [2\tau_{0}, \tau_{0}]} |\tau|^{\frac{1}{50}} \|v(\cdot ,\tau)-\sqrt{2}\|_{C^{4}(B(0, 2|\tau|^{{1}/{100}}))}\leq 1.
\end{equation}
\end{definition}
The nonuniform sharp asymptotics from \cite[Theorem 1.4]{DH_hearing_shape} imply that, given any $\kappa>0$, every bubble-sheet oval $\mathcal{M}$ in $\mathbb{R}^4$ is, after suitable recentering,\footnote{For the detailed recentering argument, please see Section \ref{sec_conclusion}.} $\kappa$-quadratic at some $\tau_0=\tau_0(\mathcal{M},\kappa)\ll 0$. Here, we upgrade these prior asymptotics to uniform sharp asymptotics for $\kappa$-quadratic families:

\begin{theorem}[uniform sharp asymptotics]\label{strong_uniform0} For every $\varepsilon>0$, there exists $\kappa>0$ and $\tau_{*}>-\infty$ such that if a bubble-sheet oval $\mathcal{M}=\{M_t\}$ in $\mathbb{R}^4$ is $\kappa$-quadratic at time $\tau_{0}\leq \tau_{*}$, then for every $\tau\leq \tau_{0}$ we have:
\begin{enumerate}[(i)]
    \item Parabolic region: The renormalized profile function $v$ satisfies
     \begin{equation*}
    \sup_{y\leq \eps^{-1}}\left|v(y,\varphi, \tau)-\sqrt{2}+\frac{y^2-4}{\sqrt{8}|\tau|}\right|\leq \frac{\varepsilon}{|\tau|} .
    \end{equation*}
    \item Intermediate region: The renormalized profile function $v$ satisfies
\begin{equation*}
        \sup_{z\leq \sqrt{2}-\eps}\left|v(|\tau|^{1/2}z,\varphi, \tau)-\sqrt{2-z^2}\right|\leq \varepsilon.
\end{equation*}
    \item Tip region: For any $s\leq - e^{-\tau_0}$, setting $\lambda(s)=\sqrt{|s|^{-1}\log|s|}$, and denoting by $p_s^\varphi\in M_s$ the tip point in direction $\varphi$, the flow
$\widetilde{M}^{\varphi,s}_t:= \lambda(s)\cdot (M_{s+\lambda(s)^{-2}t} - p_s^\varphi)$
is $\eps$-close in $C^{\lfloor 1/\eps\rfloor}$ in $B_{\eps^{-1}}(0) \times (-\eps^{-2},\eps^{-2})$ to $N_t\times\mathbb{R}$, where $N_t$ is the $2d$-bowl with tip $0\in N_0$ that translates in negative $\cos(\varphi)e_1+\sin(\varphi)e_2$ direction with speed $1/\sqrt{2}$. 
\end{enumerate}
\end{theorem}

Here, the tip point in direction $\varphi$ is given by
\begin{equation}
p_s^\varphi=(R(\varphi,s)\cos\varphi,R(\varphi,s)\sin\varphi,0,0),
\end{equation}
where the radius in direction $\varphi$ can be expressed as
\begin{equation}
R(\varphi,s)=|s|^{1/2}\sup \{ y\geq 0 : v(y,\varphi,-\log(-s)) > 0\}.
\end{equation}

To prove Theorem \ref{strong_uniform0} (uniform sharp asymptotics), we
 first establish the uniform sharp asymptotics under the a priori assumption of so-called strong $\kappa$-quadraticity, and then justify this assumption by applying quantitative Merle-Zaag type arguments to the spectral ODEs originating from \cite{DH_hearing_shape}.

\bigskip

In Section \ref{sec_concavity}, we prove a quadratic almost concavity estimate. As discussed above, the maximum principle for $(v^2)_{yy}$ completely breaks down in our setting, due to all the angular terms. Here, we instead establish a novel matrix estimate, which implies quadratic almost concavity in the radial direction as a corollary. To this end, we view $v^2$ as an intrinsic time-dependent quantity on $S^3$. Specifically, denoting the parametrized renormalized mean curvature flow by $\bar{F}_\tau:S^3\to\mathbb{R}^4$, so that $\bar{M}_\tau=\bar{F}_\tau(S^3)$, we set
 \begin{equation}
q(p,\tau):=\big(\bar{F}_\tau(p)\cdot \omega\left(\bar{F}_\tau(p)\right)\big)^2,
 \end{equation}
 where $\omega$ denotes the vector field in $\mathbb{R}^4$ defined by
 \begin{equation}
 \omega(x_1,x_2,x_3,x_4):=\big(0,0,x_3/(x_3^2+x_4^2)^{1/2},x_4/(x_3^2+x_4^2)^{1/2}\big).
 \end{equation}
Writing $\bar{g}_{ij}= \partial_i \bar{F}_\tau \cdot \partial_j \bar{F}_\tau$ for the induced metric, and $\bar{\nabla}$ for its Levi-Civita connection, we  prove the following Hessian estimate:

\begin{theorem}[quadratic almost concavity]\label{prop-concavity}
There exist constants $\kappa>0$ and $\tau_*>-\infty$ with the following significance. If  $\mathcal{M}$ is $\kappa$-quadratic at time $\tau_0 \le \tau_*$, then for all $\tau\leq\tau_0$ we have
\begin{equation}
\bar{\nabla}^2 q(X,X) \le \, \left(\frac{1}{|\tau| q}\right)^{3/2} \,  \bar{g}(X,X)
\end{equation}
for all $X\perp \partial_\vartheta$.
\end{theorem}      

In order to show above theorem, for any $\delta>0$, we consider the tensor
 \begin{equation}
  \label{eq-defAij'} A_{ij} := \bar{\nabla}^2_{ij} q-\left(\left(\frac{1}{|\tau| q}\right)^{3/2}+\delta\right) \, \bar{g}_{ij}.
  \end{equation}
We first observe that in the parabolic region and in the soliton region we have $A(X,X)< 0$ for all $0\neq X\perp \partial_\vartheta$. The major computational step is then to show, using the maximum principle, that if $A(X,X)\le 0$ for all $X\perp \partial_\vartheta$ initially, then it stays so for later times as well. Finally, we let $\delta \to 0$, at the same time ensuring that $\kappa$ and $\tau_\ast$ do not depend on $\delta$, to conclude our goal.

Let us try to convey some intuition for how to come up with the appropriate tensor for applying the maximum principle. As a first attempt, one might try to work with the tensor $\bar{\nabla}^2_{ij} q$. However, then one runs into problems showing negativity in the soliton region and at the initial time. To remedy this one might then try to subtract $\delta \bar{g}_{ij}$. However, as $\delta$ becomes smaller this only enforces the negativity in the soliton region further and further back in time. To overcome this problem, the tricky part is to find a suitable perturbation that satisfies all of the following 3 competing properties: (i) yields negativity in the soliton region uniformly as $\delta\to 0$, (ii) is small enough so that it still implies the corollary discussed below, and (iii) has a good evolution equation so that the maximum principle still applies. It turns out that subtracting $(|\tau|q)^{-3/2}\bar{g}_{ij}$ accomplishes (i), (ii) and (iii).\footnote{Here, we work with the exponent $3/2$, since to easily get the right sign in the parabolic region it helps to have an exponent less than $2$, and to establish the corollary via integrating the corresponding ODE in radial direction we need an exponent bigger than $1$.}

Let us also point out that Hamilton's tensor maximum principle from \cite{Ham_pco} is not directly applicable in our setting, since the vector field $\partial_\vartheta$ is not parallel. However, fortunately we are able to absorb the extra error terms originating from $\bar{\nabla}\partial_\vartheta\neq 0$ into our good reaction term.
 
As a very important corollary we obtain the following estimate in the collar region, which is crucial for connecting the bubble-sheet behaviour in the cylindrical region and the translator behaviour in the soliton region:

\begin{corollary}[almost Gaussian collar] \label{prop-great}
For every $\varepsilon > 0$, there exist constants $\kappa>0$, $\tau_*>-\infty$, $L<\infty$, and $\theta >0$ with the following significance. If  $\mathcal{M}$ is $\kappa$-quadratic at time $\tau_0 \le \tau_*$, then for all $\tau \le \tau_0$ we have
\begin{equation}
\left|y (v^2)_y + 4\right| < \varepsilon \quad \mbox{in the collar region } \, \mathcal{K}= \bigl\{  L/\sqrt{|\tau|} \le v \le 2  \theta \bigr\}.
\end{equation}
 \end{corollary}  

Indeed, we will see that the corollary is a consequence of the almost monotonicity of $y\mapsto(v^2)_y(y,\varphi,\tau)$ that follows from Theorem \ref{prop-concavity} (quadratic almost concavity) together with the behavior of $(v^2)_y$ at the boundaries of the collar region $\mathcal{K}$ that follows from Theorem \ref{strong_uniform0} (uniform sharp asymptotics).

\bigskip

In Section \ref{spectral_uniqueness}, we upgrade the uniform sharp asymptotics to the following spectral uniqueness result:
\begin{theorem}[spectral uniqueness]\label{thm-spectral_uniqueness}
There exist $\kappa>0$ and $\tau_{*}>-\infty$ with the following significance. If $\mathcal{M}^{1}=\{M^1_t\}$ and $\mathcal{M}^{2}=\{M^2_t\}$ are  bubble-sheet ovals in $\mathbb{R}^{4}$ that are $\kappa$-quadratic at  time $\tau_{0}$, where $\tau_0\leq \tau_{*}$, and if their truncated renormalized profile functions $v^{1}_{\cC}$ and $v^{2}_{\cC}$ satisfy the spectral condition
\begin{equation}
      \mathfrak{p}_{0}v^{1}_{\cC}(\tau_{0})=\mathfrak{p}_{0}v^{2}_{\cC}(\tau_{0}),
    \end{equation}
then 
\begin{equation}
    \mathcal{M}^{1}=\mathcal{M}^{2}.
\end{equation}
\end{theorem}
The proof is based on energy estimates in carefully chosen weighted norms.  We of course need to establish such estimates in both regions: the cylindrical region $\mathcal{C}$ defined in \eqref{eq-cyl-region}
and the  tip region $\mathcal{T}$ defined in \eqref{eq-tip-region}.

For these energy estimates, in addition to the Gaussian $L^2$-norm $\| \,\, \|_{\mathcal{H}}$, we also need the Gaussian $H^1$-norm
\begin{equation}
\|f\|_{\hD} := \left( \int \left(f^2 +f_y^2+\tfrac{1}{y^2}f_\varphi^2\right) e^{-y^2/4} \, yd\varphi\, dy \right)^{1/2},
\end{equation}
and its dual norm $\| \,\, \|_{\hD^\ast}$. Moreover, for time-dependent functions this induces the parabolic norms 
\begin{equation}
\|f \|_{\mathcal{X},\infty}:=\sup_{\tau\leq \tau_0 }\left( \int_{\tau-1}^\tau \| f(\cdot,\sigma)\|^2_{\mathcal{X}} \, d\sigma \right)^{1/2},
\end{equation}
where $\mathcal{X}=\mathcal{H},\hD$ or $\hD^\ast$. Furthermore, in the tip region we work with the norm
\begin{equation}
\| F\|_{2,\infty}:= \sup_{\tau\leq \tau_0} \frac{1}{|\tau|^{1/4}} \left( \int_{\tau-1}^\tau \int_0^{2\theta}\int_0^{2\pi} F^2 e^{\mu}\, d\varphi\, dv\, d\sigma \right)^{1/2},
\end{equation}
where $\mu=\mu(v,\varphi,\sigma)$ is a carefully chosen weight function. Roughly speaking, this weight function nicely interpolates between the Gaussian weight in the cylindrical region and a certain natural weight defined in terms of $\mathbb{R}\times$2d-bowl, which describes the behavior of our solution in the tip region, in the appropriate scale. More precisely, we arrange that
\begin{equation}\label{weight_large}
\mu(v,\varphi,\sigma)=-\tfrac{1}{4}Y_1(v,\varphi,\sigma)^2\quad\mathrm{ for }\, v\geq   \theta /4,
\end{equation}
and
\begin{equation}
\mu_v(v,\varphi,\sigma)=\frac{1+Y_{B,v}(v,\sigma)^2}{v}\quad\mathrm{ for }\, v\leq   \theta /8,
\end{equation}
where $Y_{1}$ is the inverse profile function of $v_{1}$, and $Y_B$ is defined in terms of the profile function of the 2d bowl in the appropriate scale. Corollary \ref{prop-great} (almost Gaussian collar) is crucial to ensure that our weight function $\mu$ has the properties that are needed to establish a weighted Poincar\'e inequality.

In the cylindrical region, we then consider the difference function
\begin{equation}
w_{\cC} := v_{\cC}^1 - v_{\cC}^2,
\end{equation}
and prove that for every $\eps>0$ there exist $\kappa>0$ and $\tau_*>-\infty$, such that if $\mathcal{M}^1$ and $\mathcal{M}^2$  are $\kappa$-quadratic at time $\tau_0 \le \tau_*$, then
\begin{equation}
\label{eq-cyl-coercive1}
\left\Vert w_{\cC}  - \mathfrak{p}_0w_{\cC}  \right\Vert_{\hD ,\infty } \le \eps  \left( \left\Vert  w_{\cC} \right\Vert_{\hD ,\infty} + \left\Vert w\, 1_{\{  \theta /2\leq v_1\leq  \theta \}}\right\Vert_{\mathcal{H},\infty}  \right).
\end{equation}

In the tip region, we work with the difference function
\begin{equation}
W:=Y_1-Y_2,
\end{equation}
where  $Y_i(\cdot,\varphi,\tau)$ is defined as inverse function of $v_i(\cdot,\varphi,\tau)$ as in \eqref{Y_inverse profile}. Furthermore, we consider the truncated function
\begin{equation}
W_{\mathcal T} := \chi_{\mathcal{T}} \, W,
\end{equation}
where $\chi_{\mathcal{T}}$ is a suitable cutoff function that localizes in the tip region. We prove that for every $\eps>0$ there exist $\kappa>0$ and $\tau_*>-\infty$, such that if $\mathcal{M}^1$ and $\mathcal{M}^2$  are $\kappa$-quadratic at time $\tau_0 \le \tau_*$, then
  \begin{equation}
    \label{eqn-tip1}
    \| W_{\mathcal{T}} \|_{2,\infty} \leq \eps \, \| W \, 1_{\{\theta\leq v\leq 2\theta\} } \|_{2,\infty}.
  \end{equation}
In particular, to be able to absorb all the extra terms coming from the angular derivatives into the good term from our weighted Poincar\'e inequality, we have to establish quite sharp a priori estimates in the tip region.

Finally,  we combine the energy estimates \eqref{eq-cyl-coercive1} and \eqref{eqn-tip1}, taking also into account the equivalence of norms in the transition region thanks to \eqref{weight_large}, to conclude that $w_{\cC}=0$ and $W_{\mathcal{T}}=0$, hence $\mathcal{M}^1=\mathcal{M}^2$.

\bigskip

In Section \ref{classification}, we argue how to go from spectral uniqueness to the classification result. To this end, we have to deal with the spectral conditions
\begin{equation}
\mathfrak{p}_{0}v^{1}_{\cC}(\tau_{0})=\mathfrak{p}_{0}v^{2}_{\cC}(\tau_{0})\quad\textrm{ and }\quad \mathfrak{p}_{+}\big(v^{i}_{\cC}(\tau_{0})-\sqrt{2}\big)=0.
\end{equation}
The bulk of our argument for this takes place in terms of the class $\mathcal{A}^{\circ}$, which has the advantage that we have $\mathbb{Z}_2^2$-symmetry a priori, and general bubble-sheet ovals will only enter towards the end of our argument. Specifically, thanks to the $\mathbb{Z}_2^2$-symmetry, for any $\mathcal{M}\in \mathcal{A}^{\circ}$ we automatically have orthogonality with respect to the eigenfunctions $y\cos\varphi,y\sin\varphi$ and $y^2\sin(2\varphi)$.\\
To deal with the orthogonality relations with respect to the eigenfunctions $1$ and $y^2-4$ 
we consider the transformed flow
\begin{equation}
\mathcal{M}^{\beta,\gamma}=\big\{ e^{\gamma/2} M_{e^{-\gamma}(t-\beta)}\big\},
\end{equation}
and find parameters $\beta,\gamma$ such that the truncated renormalized profile functions of $\mathcal{M}^{\beta,\gamma}$ satisfies
\begin{equation}\label{Psi=00intro}
 \Big\langle 1, v^{\beta,\gamma}_\cC(\tau_0)-\sqrt{2}  \Big\rangle_{\mathcal{H}}=0,\,\,\,
 \Big\langle y^2-4 ,v^{\beta,\gamma}_\cC(\tau_0) +\frac{y^2-4}{\sqrt{8}|\tau_0|} \Big\rangle_{\mathcal{H}}=0.
\end{equation}
We recall that all prior related orthogonality arguments in the literature, see e.g. \cite{ADS2,DH_ovals}, were based on degree theory, which only gives existence and no uniqueness, hence no continuous dependence. In contrast, here we prove a novel Jacobian estimate, which ensures that locally under the $\kappa$-quadraticity assumption we can find canonical parameters $\beta,\gamma$ that depend continuously on $\mathcal{M}$. Specifically, we consider the map
\begin{equation}\label{Psi=00_intro}
\Psi_\tau(b, \Gamma)=\left(  \Big\langle  1 ,v^{\beta,\gamma}_\cC(\tau)-\sqrt{2} \Big\rangle_{\mathcal{H}}, 
\Big\langle  y^2-4, v^{\beta,\gamma}_\cC(\tau) +\frac{y^2-4}{\sqrt{8}|\tau|}\Big\rangle_{\mathcal{H}} \right),
\end{equation}
where \begin{equation}\label{bgamma_intro}
   \beta=e^{-\tau}\left((1+b)^2-1\right), \qquad \gamma=\tau\Gamma+2\ln(1+b),
\end{equation}
and prove:
\begin{proposition}[Jacobian estimate]\label{JPhiestimates_intro} There exist constants $\kappa>0$ and $\tau_\ast>-\infty$ with the following significance. If $\mathcal{M}$ is $\kappa$-quadratic at time $\tau_0\leq \tau_\ast$, then 
the Jacobi matrix of $\Psi_\tau$ satisfies
\begin{equation}
    \mathrm{det}(J\Psi_\tau(b, \Gamma))>0
\end{equation}
for all $\tau\leq\tau_0$ and all $(b,\Gamma)$ with $|\tau|^2b^2+\Gamma^2\leq 100 \kappa^2 $.
\end{proposition}

It is clear that the above Proposition ensures  the existence of parameters $\beta, \gamma$ so that both orthogonality conditions \eqref{Psi=00intro} hold. 
 To deal with the remaining sixth orthogonality condition, we consider the spectral width ratio map\footnote{The spectral width ratio compares the amount of inwards quadratic bending in $y_1$-direction and $y_2$-direction at time $\tau_0$, and thus -- at least heuristically -- is related to the more intuitive geometric width ratio $\max |y_2| / \max |y_1|$ of the oval at time $\tau_0$.}
\begin{equation}
\mathcal{R}( \mathcal{M})=\frac{\langle v^{\mathcal{M}}_{\cC}(\tau_{0}), y^{2}\cos^2\varphi-2 \rangle_{\mathcal{H}}}{\langle v^{\mathcal{M}}_{\cC}(\tau_{0}), y^{2}\sin^2\varphi-2 \rangle_{\mathcal{H}}}\, ,
\end{equation}
and prove:
\begin{theorem}[existence with prescribed spectral width ratio]\label{prescribed_eccentricity_intro} 
There exist constants $\kappa>0$, $\delta>0$, and $\tau_\ast>-\infty$ with the following significance.  For every $\tau_{0}\leq \tau_{*}$ and every $r\in [(1+ \delta |\tau_0|^{-1})^{-1},1+ \delta |\tau_0|^{-1}]$, there exists a bubble-sheet oval $\mathcal{M}$ that is $\kappa$-quadratic at time $\tau_0$, and satisfies \eqref{Psi=00intro} and 
\begin{equation}
    \mathcal{R}(\mathcal{M})=r,
\end{equation}
and that up to transformation belongs to the class $ \mathcal{A}^\circ$.
\end{theorem}
To prove this theorem we use a continuity argument in the ellipsoidal parameter $a$ from the construction of the class $\mathcal{A}^\circ$. The continuous dependence of $\beta,\gamma$ on $\mathcal{M}$, which follows from Proposition \ref{JPhiestimates_intro} (Jacobian estimate), is crucial for this step. Moreover, since no Rado-type argument is available in our setting, we have to set up the continuity argument in a more involved way than in \cite{CHH_translator}, making use in particular of the $\mathbb{Z}_2^2$-symmetry.\\
On the other hand, given any bubble-sheet oval $\mathcal{M}^1$, via a more standard argument based on degree theory and basic linear algebra, we can arrange that
\begin{equation}
\Big\langle v_{\cC}^{1}(\tau_0), y^2\sin(2\varphi)\Big\rangle_{\mathcal{H}}=0,\quad  \Big\langle v_{\cC}^{1}(\tau_0) +\frac{y^2-4}{\sqrt{8}|\tau_0|},y^2-4  \Big\rangle_{\mathcal{H}}=0,
\end{equation}
and such that $\mathcal{M}^1$ is $\kappa$-quadratic at time $\tau_0$, in particular
\begin{equation}
\fp_+ \big(v_{\cC}^{1}(\tau_0)-\sqrt{2}\big)=0.
\end{equation}
Thanks to Theorem \ref{prescribed_eccentricity_intro} (existence with prescribed spectral width ratio) we can then find a bubble-sheet oval $\mathcal{M}^2$ that is obtained as suitable transformation of an element of the class $\mathcal{A}^\circ$, such that
\begin{equation}
\mathcal{R}(\mathcal{M}^1)=\mathcal{R}(\mathcal{M}^2).
\end{equation}
Finally, applying Theorem \ref{spectral_uniqueness} (spectral uniqueness) we can then complete the proof of Theorem \ref{classification_theorem} (classification of bubble-sheet ovals).

\bigskip

\bigskip

\noindent\textbf{Acknowledgments.} We thank the referees for helpful comments.
The first author has been supported by the POSCO Science Fellowship of POSCO TJ Park Foundation and the National Research Foundation of Korea grant NRF-2022R1C1C1013511.
The second author has been supported by the NSF grants DMS-1266172 and DMS-1900702. 
The third and fourth author have been supported by the NSERC Discovery grants RGPIN-2016-04331 and RGPIN-2023-04419 and the Sloan Research Fellowship of the fourth author, and the third author is in addition very grateful to Professor Xiaobo Liu at Beijing International Center for Mathematical Research for visiting funding and hospitality. The fifth author has been supported by the NSF Grant DMS-2154782.

\bigskip

\section{Uniform sharp asymptotics}\label{Uniform sharp asymptotics}

In this section, we establish uniform sharp asymptotics for our bubble-sheet ovals. Our scheme of proof, similarly to the one for translators from \cite[Section 3]{CHH_translator}, is to first derive uniform sharp asymptotics under a stronger a priori  assumption, called strong $\kappa$-quadraticity, and then to use quantitative Merle-Zaag type arguments to justify this a priori assumption.

Throughout this section $\mathcal{M}=\{M_t\}$ denotes a bubble-sheet oval in $\mathbb{R}^4$ (see Definition \ref{def_bubble_sheet_oval}), where by \cite[Theorem 1.4]{DH_hearing_shape} we can always assume that we have $\textrm{SO}(2)$-symmetry in the $x_3x_4$-plane centered at  the origin. Since the tangent flow at $-\infty$ is given by \eqref{bubble-sheet_tangent_intro}, for $\tau \to -\infty$ the renormalized flow
\begin{equation}
\bar M_\tau = e^{\frac{\tau}{2}}  M_{-e^{-\tau}}
\end{equation}
converges smoothly on compact subsets to the static bubble-sheet
\begin{equation}
\Gamma:=\mathbb{R}^2\times S^{1}(\sqrt{2}).
\end{equation}
We denote points in $\mathbb{R}^2$ by
\begin{equation}
{\bf y} = (y_1,y_2)=(y\cos\varphi, y \sin\varphi), \quad \mathrm{ where } \quad y=|{\bf y}|.
\end{equation}
Let $\bar{\Omega}_\tau$ be the set of points ${\bf y}\in\mathbb{R}^2$ such that $({\bf y},r\cos\vartheta,r\sin\vartheta)\in \bar{M}_\tau$ for some $r\geq 0$, and define $u({\bf y},\tau)$, where ${\bf y}\in\bar{\Omega}_\tau$, by
\begin{equation}
({\bf y},(\sqrt{2}+u({\bf y},\tau))\cos\vartheta,(\sqrt{2}+u({\bf y},\tau))\sin\vartheta) \in \bar M_\tau\, .
\end{equation}
Note that the graphical function $u$ and the profile function $v$ are related by
\begin{equation}\label{rel_v_u}
v(y,\varphi,\tau)=\sqrt{2}+u(y\cos\varphi,y\sin\varphi,\tau).
\end{equation}
Since $\bar{M}_\tau$ evolves by renormalized mean curvature flow, $u$ satisfies
\begin{align}\label{equation_u} 
    u_\tau= \left(\delta_{ij}-\tfrac{u_{y_i}u_{y_j}}{1+|Du|^2}\right)\, u_{y_iy_j}-\frac{1}{2}  \, y_i u_{y_i} + \frac{ \sqrt{2}+u}2 -\frac{1}{\sqrt{2}+u}\, ,
\end{align}
where the summation convention is used over all indices $i, j \in\{1,2\}$.  
Furthermore, fixing a smooth cutoff function with $\chi(s)=1$ for $s\leq 1$ and $\chi(s)=0$ for $s\geq 2$, we often consider the truncated graphical function
\begin{equation}
\hat{u}({\bf y},  \tau)=u({\bf y},  \tau)\chi\left(\frac{|{\bf y}|}{\rho(\tau)}\right),
\end{equation}
where $\rho(\tau)$ is any admissible graphical radius, i.e.
 \begin{equation}
\label{univ_fns}
\lim_{\tau \to -\infty} \rho(\tau)=\infty, \quad \textrm{and}\quad  -\rho(\tau) \leq \rho'(\tau) \leq 0,
\end{equation}
and
\begin{equation}\label{small_graph_admissible}
\|u(\cdot,\tau)\|_{C^4(\Gamma \cap B_{2\rho(\tau)}(0))} \leq  \rho(\tau)^{-2}.
\end{equation}
Finally, we recall that our Gaussian inner product is given by the formula
\begin{equation}
\langle f,g\rangle_{\mathcal{H}}=\int_{\mathbb{R}^2} f({\bf y})g({\bf y})e^{-\frac{|{\bf y}|^2}{4}}\, d{\bf y}\, ,
\end{equation}
and that there is the well-known weighted Poincar\'e inequality
\begin{equation}\label{easy_Poincare}
\big\|  (1+|{\bf{y}}|) f \big\|_{\mathcal{H}} \leq C \big( \| f \|_{\mathcal{H}}+\| Df \|_{\mathcal{H}}\big).
\end{equation}
Indeed, by approximation it is enough to check this for smooth compactly supported functions $f$, and for such functions this follows by computing
\begin{align}
\int_{\mathbb{R}^2} \left(\tfrac12 |{\bf{y}}|^2f^2 -2f^2\right)\, e^{-\frac{|{\bf y}|^2}{4}}\, d{\bf y}&=\int_{\mathbb{R}^2} D(f^2)\cdot {\bf{y}}\, e^{-\frac{|{\bf y}|^2}{4}}\, d{\bf y}\nonumber\\
&\leq \int_{\mathbb{R}^2} \left(\tfrac14 |{\bf{y}}|^2 f^2 +4|Df|^2 \right)\, e^{-\frac{|{\bf y}|^2}{4}}\, d{\bf y}.
\end{align}

\subsection{Uniform sharp asymptotics assuming strong $\kappa$-quadraticity}\label{sec2.1}

In this subsection, we establish uniform sharp asymptotics under the following a priori assumption:
\begin{definition}[{strong  $\kappa$-quadraticity, c.f. \cite[Definition 3.7]{CHH_translator}}]\label{strong}
We say that a bubble-sheet oval $\mathcal{M}$ in $\mathbb{R}^{4}$ (with coordinates chosen as above) is \emph{strongly $\kappa$-quadratic from time $\tau_{0}$},  if 
\begin{enumerate}[(i)]
\item $\rho(\tau)=|\tau|^{1/10}$ is an admissible graphical radius for $\tau\leq \tau_{0}$, and
\item  the truncated graphical function $\hat{u}(\cdot,  \tau)=u(\cdot,  \tau)\chi\left(\frac{|\cdot|}{\rho(\tau)}\right)$ satisfies 
\begin{equation}
    \left\| \hat{u}({\bf{y}},  \tau)+\frac{|{\bf{y}}|^2-4}{\sqrt{8}|\tau|}  \right\|_{\mathcal{H}}\leq \frac{\kappa}{|\tau|}\quad \text{for}\,\,\tau\leq \tau_{0}.
\end{equation}
\end{enumerate}
\end{definition}

We will now upgrade the sharp asymptotics from \cite[Theorem 1.4]{DH_hearing_shape} to uniform sharp asymptotics for families of strongly  $\kappa$-quadratic solutions.\footnote{The reader might wonder whether deriving sharp asymptotics twice is inefficient. However, we first needed the nonuniform asymptotics in \cite{DH_hearing_shape} to prove $\textrm{SO}(2)$-symmetry. Having established the symmetry, we can now upgrade the estimates to uniform estimates.} 

\begin{proposition}[parabolic region]\label{uniform_par}
For every $\varepsilon>0$, there exists $\kappa>0$ and $\tau_{*}>-\infty$, such that if $\mathcal{M}$  is strongly $\kappa$-quadratic from time $\tau_{0}\leq \tau_{*}$, then for every $\tau\leq \tau_{0}$ we have:
     \begin{equation}
    \sup_{|{\bf{y}}|\leq \eps^{-1}}\left| u({\bf{y}}, \tau)+\frac{|{\bf{y}}|^2-4}{\sqrt{8}|\tau|}\right|\leq \frac{\varepsilon}{|\tau|} .
    \end{equation}
\end{proposition}

\begin{proof}
For ease of notation, let us abbreviate
\begin{equation}
{\mathcal{D}}({\bf{y}}, \tau):=\hat{u}({\bf{y}}, \tau)+\frac{|{\bf{y}}|^2-4}{\sqrt{8}|\tau|} .
\end{equation}
By the strong $\kappa$-quadraticity assumption, for all $\tau\leq \tau_0$ we have
\begin{equation}\label{def_d_small}
    \| {\mathcal{D}}(\cdot,\tau)\|_{\mathcal{H}}\leq \frac{\kappa}{|\tau|}
\end{equation}
Moreover, since $\rho(\tau)=|\tau|^{1/10}$, for any sufficiently negative $\tau$ we have
\be\label{hat_no_hat}
\hat{u}({\bf{y}}, \tau)=u({\bf{y}}, \tau)\quad\textrm{ for}\,\, |{\bf{y}}|\leq 2\eps^{-1}.
\ee
Hence, using \eqref{equation_u} and \eqref{small_graph_admissible} we can find constants $C=C(\varepsilon)<\infty$ and $\tau_{*}(\varepsilon)>-\infty$  such that 
\begin{equation}\label{w32_est}
     \big\| {\mathcal{D}}(\cdot,\tau)\big\|_{W^{3,2}(B(0,\eps^{-1}))}\leq \frac{C}{|\tau|}
\end{equation}
holds for all $\tau\leq \tau_{0}$, provided $\tau_{0}\leq \tau_{*}(\varepsilon)$. Applying Agmon's inequality with \eqref{def_d_small} and \eqref{w32_est}, we conclude that for all $\tau\leq\tau_0\leq\tau_\ast(\eps)$ we have
\begin{equation}
  \sup_{|{\bf y}|\leq \eps^{-1}}  |{\mathcal{D}}({\bf y},\tau)| \leq \frac{\eps}{|\tau|},
\end{equation}
provided $\kappa=\kappa(\eps)$ is sufficiently small. Remembering \eqref{hat_no_hat}, this proves the proposition.
\end{proof}

Next, to capture the intermediate region we consider the function
\begin{equation}
\bar{v}(z,\varphi,\tau):=v(|\tau|^{1/2}z,\varphi,\tau),
\end{equation}
where $v$ is the renormalized profile function, c.f. equation \eqref{rel_v_u}.

\begin{proposition}[intermediate region]\label{prop_interm}
For every $\varepsilon>0$, there exists $\kappa>0$ and $\tau_{*}>-\infty$, such that if $\mathcal{M}$  is strongly $\kappa$-quadratic from time $\tau_{0}\leq \tau_{*}$, then for every angle $\varphi$ and every time $\tau\leq \tau_{0}$ we have:
  \begin{equation}
       \sup_{z\leq \sqrt{2}-\eps}\left|\bar{v}(z,\varphi,\tau)-\sqrt{2-z^2}\right|\leq \varepsilon.
\end{equation}
\end{proposition}

\begin{proof}
Using the same barrier argument as in \cite[Proof of Proposition 6.3]{DH_hearing_shape}, our uniform sharp asymptotics from Proposition \ref{uniform_par} (parabolic region) can be promoted to a uniform sharp lower bound in the intermediate region, yielding
\begin{equation}
   \inf_{z\leq \sqrt{2}-\varepsilon}\left(\bar{v}(z,\varphi, {\tau})-\sqrt{2-z^{2}}\right)\geq -\varepsilon
\end{equation}
for all $\varphi$ and all $\tau\leq\tau_0\leq \tau_\ast$, provided $\kappa>0$  is sufficiently small and $\tau_\ast$ is sufficiently negative.

To establish the matching upper bound, note that  by the evolution equation \eqref{equation_u}  and  by  convexity the profile function $v=v(y,\varphi,\tau)$ satisfies
\begin{align}\label{v_evolution_inequ}
    v_\tau\leq 
     -\frac{1}{v}+\frac{1}{2}\left(v- yv_{y}\right) \, .
\end{align}
Hence, given any angle $\varphi$, the function
\begin{equation}
w^\varphi(y,\tau):=v(y, \varphi, \tau)^2-2\, 
\end{equation}
satisfies
\begin{equation}\label{ODE_w}
   w^\varphi_\tau\leq w^\varphi-\tfrac12 y w^\varphi_y. 
\end{equation}
Moreover, by Proposition \ref{uniform_par} (parabolic region), given any $A<\infty$, there are $\kappa_{*}(A)>0$ and $\tau_{*}(A)>-\infty$, such that if the bubble-sheet oval is strongly $\kappa$-quadratic from time $\tau_{0}\leq \tau_{*}$, where $0<\kappa\leq \kappa_{*}$, then
\begin{equation}
     w^\varphi(y, \tau)\leq |\tau|^{-1}({4-y^2})+A^{-1}|\tau|^{-1}
\end{equation}
holds for all $y\leq A$. Using this, we can integrate \eqref{ODE_w} along characteristic curves, similarly as in \cite[Proof of Proposition 6.3]{DH_hearing_shape}, to conclude that
\begin{equation}
   \sup_{z\leq \sqrt{2}-\varepsilon}\left(\bar{v}(z,\varphi, {\tau})-\sqrt{2-z^2}\right)\leq \varepsilon
\end{equation}
for all $\varphi$ and all $\tau\leq\tau_0\leq \tau_\ast$, provided $\kappa$  is sufficiently small and $\tau_\ast$ is sufficiently negative.
This finishes the proof of the proposition.
\end{proof}

In the tip region, instead of with the polar angle $\varphi$, we will first work with the outward unit normal angle $\phi$, which is more suitable for applying Hamilton's Harnack inequality. Specifically, given any angle $\phi$, denote by $p_s^\phi\in M_s$ the point that maximizes $\langle p,\cos(\phi)e_1+\sin(\phi)e_2\rangle$ among all $p\in M_s$, and set
\begin{equation}
\widehat{M}^{\phi,s}_t:= \lambda(s)\cdot (M_{s+\lambda(s)^{-2}t} - p_s^\phi),
\end{equation}
where
\begin{equation}
\lambda(s):=\sqrt{|s|^{-1}\log|s|}.
\end{equation}

\begin{proposition}[tip region in terms of normal angle]\label{uniform_tip_normal}
For every $\varepsilon>0$, there exist $\kappa>0$ and $\tau_{*}>-\infty$ with the following significance. If $\mathcal{M}$  is strongly $\kappa$-quadratic from time $\tau_{0}\leq \tau_{*}$, then for every angle $\phi$ and every $s \leq -e^{- \tau_{0}}$ the flow $\widehat{M}^{\phi,s}_t$ is $\eps$-close in $C^{\lfloor 1/\eps\rfloor}$ in $B_{\eps^{-1}}(0) \times (-\eps^{-2},\eps^{-2})$ to $N_t\times\mathbb{R}$, where $N_t$ is the $2d$-bowl with tip $0\in N_0$ that translates in negative $\cos(\phi)e_1+\sin(\phi)e_2$ direction with speed $1/\sqrt{2}$. 
\end{proposition}

\begin{proof}
Suppose towards a contradiction that for some $\eps>0$ there are bubble-sheet ovals $\mathcal{M}^{i}$ that are strongly $\kappa_{i}$-quadratic from time $\tau_{ i}$, where $\kappa_{i}\to 0$ and $\tau_{ i}\to -\infty$, but such that for some $s_i \leq -e^{- \tau_{i}}$ the flows $\widehat{M}^{i}_t:=\widehat{M}^{0,s_i}_t$ are not $\eps$-close in $C^{\lfloor 1/\eps\rfloor}$ in $B_{\eps^{-1}}(0) \times (-\eps^{-2},\eps^{-2})$ to $N_t\times\mathbb{R}$. (Here, suitably rotating coordinates we arranged that $\phi=0$ and in particular denoted by $N_t$ the $2d$-bowl with tip $0\in N_0$ that translates in negative $e_1$ direction with speed $1/\sqrt{2}$.)\\
Now, for $s\leq s_i$ denote by $p^i_s \in M_{s}^i$ the unique point where  $\max_{p\in M^i_{s}}\langle p,   e_1\rangle$
is attained, and consider the function
\begin{equation}
d_i(s):=\langle p_s^i,   e_1\rangle.
\end{equation}
By Proposition \ref{prop_interm} (intermediate region) and convexity we have
\begin{equation}\label{diameter_asymptotics}
  \lim_{i\to \infty}\sup_{s\leq s_i}\left|\frac{d_i(s)}{\sqrt{2|s|\log|s|}}-1\right| = 0.
\end{equation}
Note that by our definition of $d_i$ we have
\begin{equation}
d'_i(s)=-H(p_s^i).
\end{equation}
Together with Hamilton's Harnack inequality \cite{Hamilton_Harnack} this yields
\begin{equation}\label{tip curvature}
 \lim_{i\to \infty} \frac{H(p^i_{s_{i}})}{\lambda(s_i)}= \frac{1}{\sqrt{2}}.
\end{equation}
Hence, arguing similarly as in  \cite[Proof of Proposition 6.6]{DH_hearing_shape} we see that the sequence $\widehat{M}^{i}_t$ converges to $N_t\times\mathbb{R}$.
For $i$ large enough this contradicts our assumption that  the flows $\widehat{M}^{i}_t$ are not $\eps$-close to $N_t\times\mathbb{R}$, and thus proves the proposition.
\end{proof}

To reformulate the result in terms of the polar angle $\varphi$, let
\begin{equation}
p_s^\varphi=(R(\varphi,s)\cos\varphi,R(\varphi,s)\sin\varphi,0,0),
\end{equation}
where the radius in direction $\varphi$ can be expressed as
\begin{equation}
R(\varphi,s)=|s|^{1/2}\sup \{ y\geq 0 : v(y,\varphi,-\log(-s)) > 0\},
\end{equation}
and consider the flow
\begin{equation}
\widetilde{M}^{\varphi,s}_t:= \lambda(s)\cdot (M_{s+\lambda(s)^{-2}t} - p_s^\varphi).
\end{equation}

\begin{corollary}[tip region in terms of polar angle]\label{uniform_tip}
For every $\varepsilon>0$, there exist $\kappa>0$ and $\tau_{*}>-\infty$ with the following significance. If $\mathcal{M}$  is strongly $\kappa$-quadratic from time $\tau_{0}\leq \tau_{*}$, then for every angle $\varphi$ and every $s \leq -e^{- \tau_{0}}$ the flow $\widetilde{M}^{\varphi,s}_t$ is $\eps$-close in $C^{\lfloor 1/\eps\rfloor}$ in $B_{\eps^{-1}}(0) \times (-\eps^{-2},\eps^{-2})$ to $N_t\times\mathbb{R}$, where $N_t$ is the $2d$-bowl with tip $0\in N_0$ that translates in negative $\cos(\varphi)e_1+\sin(\varphi)e_2$ direction with speed $1/\sqrt{2}$. 
\end{corollary}

\begin{proof}
	First observe the following derivative bound for convex polar curves in the plane: If $r=r(\varphi)$ represents a closed convex polar curve  in $\mathbb{R}^2$ and
\begin{equation}	
	\max_{\varphi} |r(\varphi)-1| \leq \delta,
\end{equation}	 
then
\begin{equation}	
	\max_{\varphi }|r_\varphi| \leq \eps( \delta ),
\end{equation}	
	where $\eps(\delta )   \to 0$ as $\delta \to 0$.
Indeed, denoting by $\nu$ the outward unit normal, we have
\begin{equation}
(r\cos \varphi, r\sin \varphi) \cdot \nu = r(1+(r_\varphi/r)^2)^{-1/2}\geq 1-\delta,
\end{equation}
since for a closed convex  polar curve the minimum of the support function cannot be less than the radius lower bound. This gives the estimate
\begin{equation}\label{rphi}
\max _{\varphi} r_\varphi ^2 \le (1+\delta)^2  \left[ \left(\frac{1 +\delta }{1-\delta  }\right)^2  -1   \right].
\end{equation}
Now, in our setting thanks to Proposition \ref{prop_interm} (intermediate region) and convexity, given any $\delta>0$, by choosing $\kappa$ small enough and $\tau_\ast$ negative enough, we can arrange that
\begin{equation}
\left|\frac{R(\varphi,s)}{\sqrt{2|s|\log|s|}}-1 \right|\leq \delta.
\end{equation}
This yields
\begin{equation}
\sup_{\varphi}|R_\varphi (\varphi,s)|\leq \eps(\delta) \sqrt{2 |s| \log| s|},
\end{equation}
and consequently the outward unit normal angle $\phi$ and the polar angle $\varphi$ differ by an arbitrarily small amount (mod $2\pi$). Hence, the corollary follows from Proposition \ref{uniform_tip_normal} (tip region in terms of normal angle).
\end{proof}

\medskip

\subsection{From $\kappa$-quadraticity to strong $\kappa$-quadraticity}\label{sec2.2}

In this subsection, we upgrade $\kappa$-quadraticity (see Definition \ref{k_tau00}) to strong $\kappa$-quadraticity (see Definition \ref{strong}). We will use a quantitative Merle-Zaag type argument similarly as in  \cite[Section 3.4]{CHH_translator}. However, while in \cite{CHH_translator} the dominant term was captured by a single bending coefficient, in our setting we have to analyze a more complicated system of spectral ODEs, c.f. \cite{DH_hearing_shape}.

\begin{lemma}[initial graphical radius]\label{poly_graph}
There exists some universal number $\gamma>0$ with the following significance. For every $\kappa>0$ sufficiently small, there exists a constant $\tau_{\ast}>-\infty$,  such that if a bubble-sheet oval in $\mathbb{R}^4$ is $\kappa$-quadratic at time $\tau_0\leq \tau_{\ast}$, then  $\rho(\tau)=|\tau|^\gamma$ is an admissible graphical radius function for  $\tau \leq \tau_0$, namely \eqref{univ_fns} and \eqref{small_graph_admissible} hold for  $\tau \leq \tau_0$.
\end{lemma}
\begin{proof}
 This follows from the  Lojasiewicz-Simon inequality \cite{CM_uniqueness} and condition \eqref{condition2}, arguing similarly as in \cite[Proof of Proposition 2.5]{DH_hearing_shape}.
\end{proof}

Now, given a bubble-sheet oval in $\mathbb{R}^4$ that is $\kappa$-quadratic at time $\tau_0\leq\tau_\ast$, by Lemma \ref{poly_graph} (initial graphical radius) we can work with the truncated graphical function
\begin{equation}
\hat{u}(\cdot,  \tau)=u(\cdot,  \tau)\chi\left(\frac{|\cdot|}{|\tau|^\gamma}\right).
\end{equation}
Consider
\begin{align}
U_0(\tau) := \|\mathfrak{p}_0 \hat{u}(\cdot,\tau)\|_{\mathcal{H}}^2,\qquad U_{\pm}(\tau) := \|\mathfrak{p}_{\pm} \hat{u}(\cdot,\tau)\|_{\mathcal{H}}^2,
\end{align}
where $\mathfrak{p}_{0}$ and $\mathfrak{p}_{\pm}$ are the orthogonal projections to $\mathcal{H}_0$ and $\mathcal{H}_{\pm}$ respectively. Then, by \cite[Equation (2.24)]{DH_hearing_shape} we have the differential inequalities
\begin{align}\label{modes_inequality}
&\dot{U}_+ \geq U_+ - C_0|\tau|^{-\gamma} \, (U_+ + U_0 + U_-),\nonumber \\
&\Big | \dot{U}_0 \Big | \leq C_0|\tau|^{-\gamma} \, (U_+ + U_0 + U_-), \\\
&\dot{U}_- \leq -U_- + C_0|\tau|^{-\gamma} \, (U_+ + U_0 + U_-), \nonumber
\end{align}
for $\tau\leq\tau_0$, where $C_0<\infty$ is a numerical constant.

The following two results are closely related to the recent improved estimates from \cite{DH_no_rotation}, but with some changes to incorporate $\kappa$-quadraticity.

\begin{lemma}[{quantitative Merle-Zaag type estimate}]\label{quant_MZ}
For every $\kappa>0$ sufficiently small, there exists a constant $\tau_{\ast}>-\infty$,  such that if a bubble-sheet oval in $\mathbb{R}^4$ is $\kappa$-quadratic at time $\tau_0\leq \tau_{\ast}$, then for $\tau \leq\tau_0$ we have the estimate
\begin{equation}\label{zero_dom_mz}
U_{+}(\tau)+U_-(\tau) \leq \frac{C'}{|\tau|^\gamma}U_{0}(\tau),
\end{equation}
where $C'<\infty$ is a numerical constant.
\end{lemma}

\begin{proof}Let $\lambda(\tau)=C_{0}|\tau|^{-\gamma}$. Possibly after decreasing $\tau_\ast$, we may assume that $\lambda\leq 1/10$ and $\dot{\lambda}\leq \lambda/{10}$ for all $\tau\leq \tau_\ast$.
We will first show that
\begin{equation}\label{-leq0+}
    U_{-}\leq 2\lambda(U_{0}+U_{+}).
\end{equation}
Indeed, if at some time $\bar{\tau}\leq \tau_{0}$ the quantity $f:=U_{-}-2\lambda(U_{+}+U_{0})$ was positive, then at this time we would have
\begin{equation}
    \dot{f}\leq -U_{-}+\lambda(1+4\lambda)\left(1+\frac{1}{2\lambda}\right)U_{-}-2\dot{\lambda}(U_{+}+U_{0})\leq 0,
\end{equation}
which would imply that $ f(\tau)\geq f(\bar{\tau})>0$
for all $\tau\leq \bar{\tau}$, contradicting $\lim_{\tau\rightarrow -\infty}f(\tau)=0$. This proves \eqref{-leq0+}.
To conclude the proof we will show that
\begin{equation}\label{+leq0}
    U_{+}< 8\lambda  U_{0}.
\end{equation}
Indeed, using Definition \ref{k_tau00} ($\kappa$-quadratic), in particular the centering condition $\mathfrak{p}_{+}v_{\cC}(\tau_{0})=0$ from \eqref{condition2}, we see that \eqref{+leq0} holds at $\tau=\tau_0$, provided $\kappa$ is small enough and $\tau_\ast$ is negative enough.  Now, if the inequality \eqref{+leq0} failed  at some time less than $\tau_0$, then at the largest time $\bar\tau<\tau_0$ where it failed we would have $U_{+} = 8\lambda U_0$. Together with \eqref{modes_inequality} and \eqref{-leq0+} this would imply
\begin{align}
\frac{d}{d\tau}\left(8\lambda U_0-U_{+}\right) &\leq \lambda(8\lambda+1)(U_{+}+U_0+U_-)-U_{+} { +}8\dot{\lambda}U_0\nonumber\\
&\leq \lambda(8\lambda+1)(8\lambda+1+2\lambda(1+8\lambda))U_0 - 8(\lambda{ -}\dot{\lambda} )U_0\nonumber\\
&\leq -(\lambda {-} 8\dot{\lambda} )U_0<0,
\end{align}
contradicting the definition of $\bar\tau$. This finishes the proof of the lemma.
\end{proof}

Now, as in \cite[Section 3.1]{DH_hearing_shape}, we consider the spectral coefficients
\begin{equation}
    \alpha_{j}(\tau):=  \frac{\langle   \hat{u}(\cdot,\tau), \psi_{j} \rangle_{\cH}}{|| \psi_{j} ||_{\cH}^{2}},
\end{equation}
with respect to the neutral eigenfunctions
\begin{equation}
\psi_{1}=y_1^2-2,\quad \psi_{2}=y_2^2-2, \quad \psi_{3}=2y_1y_2.
\end{equation}

It has been shown in \cite[Proposition 3.1]{DH_hearing_shape} that the spectral coefficients $\vec{\alpha}=(\alpha_1,\alpha_2,\alpha_3)$ satisfy the ODE system
\begin{equation}\label{ODEs}
 \begin{cases}
   \dot{\alpha}_{1}=-\sqrt{8}(\alpha^2_{1}+\alpha_{3}^2)+E_{1}\\
   \dot{\alpha}_{2}=-\sqrt{8}(\alpha^2_{2}+\alpha_{3}^2)+E_{2}\\
    \dot{\alpha}_{3}=-\sqrt{8}(\alpha_{1}+\alpha_2)\alpha_{3}+E_{3},\\
    \end{cases}
\end{equation}
where the error terms satisfy
\begin{equation}
|E_{j}(\tau)|=o(|\vec{\alpha}(\tau)|^2+|\tau|^{-100}).
\end{equation}

Here, we improve the error estimate as follows:

\begin{proposition}[improved error estimate for spectral ODEs]\label{ODE system}
For any sufficiently small $\kappa>0$, there exists a constant $\tau_{\ast}>-\infty$,  such that if a bubble-sheet oval in $\mathbb{R}^4$ is $\kappa$-quadratic at time $\tau_0\leq \tau_{\ast}$, then for all $\tau \leq\tau_0$ with $|\vec{\alpha}(\tau)|\geq e^{-|\tau|^{\gamma/2}}$ the error terms in \eqref{ODEs} can be estimated by
\begin{equation}
|E_{j}(\tau)|\leq C\frac{|\vec{\alpha}(\tau)|^2}{|\tau|^{\gamma/2}}.
\end{equation}
\end{proposition}

\begin{proof}Consider the remainder
\begin{equation}
w:=\hat{u}-\sum_{j=1}^3 \alpha_j \psi_j.
\end{equation}
Inspecting the proof of \cite[Proposition 3.1]{DH_hearing_shape}, and dropping the terms that vanish thanks to the $\mathrm{SO}(2)$-symmetry, we see that it is enough to show that for all $\tau \leq\tau_0$ with $|\vec{\alpha}(\tau)|\geq e^{-|\tau|^{\gamma/2}}$ we have
\begin{equation}\label{error_uniform_control}
   \sum_{i=1}^3  \big|\alpha_i(\tau) \langle \psi_i \psi_k ,w(\tau)\rangle_\mathcal{H} \big|+ \big|\langle w(\tau)^2, \psi_k \rangle_\mathcal{H} \big| \leq \frac{C|\vec{\alpha}(\tau)|^2}{|\tau|^{\gamma/2}}.
\end{equation}
To this end, note first that by Lemma \ref{quant_MZ} (quantitative Merle-Zaag type estimate) we can estimate
\begin{equation}\label{w_hbound}
\|w(\tau)\|_{\cH} \leq C\frac{|\vec{\alpha}(\tau)|}{|\tau|^{\gamma/2}},
\end{equation}
In particular, this implies
\be
\sum_{i=1}^3  \big|\alpha_i(\tau) \langle \psi_i \psi_k ,w(\tau)\rangle_\mathcal{H} \big|\leq C\frac{|\vec{\alpha}(\tau)|^{2}  }{|\tau|^{\gamma/2}}.
\ee
Moreover, by the weighted Poincar\'e inequality \eqref{easy_Poincare} we can estimate
\begin{equation}
\big|\langle w^2, \psi_k \rangle_\mathcal{H}\big|\leq C\left(  \|w\|_{\cH}^2+  \|\nabla w\|_{\cH}^2\right).
\end{equation}
Next, to estimate the gradient term, note that thanks $\mathrm{SO}(2)$-symmetry the evolution expansion from \cite[Proposition 2.8]{DH_hearing_shape} simplifies to 
\begin{equation}
(\partial_\tau-\mathcal{L})\hat{u}=-\frac{1}{\sqrt{8}}\hat{u}^2+\hat{E},\qquad \| \hat{E}(\tau)\|_{\mathcal{H}} \leq \frac{C}{|\tau|^\gamma}\| \hat{u}(\tau)\|^2_{\mathcal{H}}  + e^{-|\tau|^{\gamma}/5}.
\end{equation}
Considering the projection $\mathcal{P}^\perp$ to the orthogonal complement of $\mathcal{H}_0$ yields
\be\label{eq_w}
(\partial_\tau-\mathcal{L})w=g,\qquad g:=\mathcal{P}^\perp \left(-\frac{1}{\sqrt{8}}\hat{u}^2+\hat{E}\right).
\ee
Remembering \eqref{small_graph_admissible} and \eqref{w_hbound} we can estimate
\be
\| \hat{u}(\tau)^2\|_{\cH}\leq \frac{C}{|\tau|^{\gamma}}\| \hat{u}(\tau)\|_{\cH}\leq C\frac{|\vec{\alpha}(\tau)|}{|\tau|^\gamma},
\ee
and, since projections do not increase the norm, this implies
\begin{equation}\label{g_hbound}
\| g(\tau) \|_{\cH}\leq  C\left(\frac{|\vec{\alpha}(\tau)|}{|\tau|^\gamma}+e^{-|\tau|^{\gamma}/5}\right).
\end{equation}
Now, given any $\bar{\tau}\leq \tau_0-1$, by \cite[Equation (3.16)]{DH_hearing_shape} for $\tau\in [\bar{\tau},\bar{\tau}+1]$ we have
\begin{align}
 \frac{d}{d\tau}\int \left((\tau-\bar{\tau})|\nabla w|^2+\tfrac{e^{\bar{\tau}-\tau}}{2} w^2\right)\, e^{-q^2/4}\leq \int g^2\, e^{-q^2/4}\, .
 \end{align}
Hence, together with \eqref{w_hbound} and  \eqref{g_hbound}  for all $\tau\leq\tau_0$ we get
\begin{equation}
\|\nabla w(\tau)\|_{\mathcal{H}} \leq C\left(\frac{\max_{\tau' \in [\tau-1,\tau]}|\vec{\alpha}(\tau')|}{|\tau|^{\gamma/2}}+e^{-|\tau|^{\gamma}/5}\right).
\end{equation}
Finally, by the Merle-Zaag ODEs \eqref{modes_inequality}, for $\tau$ sufficiently negative we get
\begin{equation}
\max_{\tau' \in [\tau-1,\tau]}|\vec{\alpha}(\tau')|^2\leq 2 |\vec{\alpha}(\tau)|^2\, .
\end{equation}
This shows that for all $\tau \leq\tau_0$ with $|\vec{\alpha}(\tau)|\geq e^{-|\tau|^{\gamma/2}}$ we have
\begin{equation}
\|\nabla w(\tau)\|_{\cH} \leq C\frac{|\vec{\alpha}(\tau)|}{|\tau|^{\gamma/2}}.
\end{equation}
Combining the above facts, we conclude that \eqref{error_uniform_control} holds. This proves the proposition.
\end{proof}

With the above ingredients, we can now prove the main result in this subsection:

\begin{theorem}[strong $\kappa$-quadraticity]\label{point_strong}
    For every $\kappa>0$, there exist $\kappa'>0$ and $\tau_{*}>-\infty$, such that if a bubble-sheet oval in $\mathbb{R}^4$ is $ \kappa'$-quadratic at some time $\tau_{0}\leq \tau_{*}$, then it is strongly $\kappa$-quadratic from time $\tau_{0}$.
\end{theorem}

\begin{proof}
As above, denoting by $\hat{u}(\cdot,\tau)=u(\cdot,\tau)\chi(|\cdot|/|\tau|^\gamma)$ the truncated graphical function of the bubble-sheet oval, we consider the expansion coefficients
\begin{equation}
    \alpha_{j}(\tau):=  \frac{\langle   \hat{u}(\cdot,  \tau), \psi_{j} \rangle_{\cH}}{|| \psi_{j} ||_{\cH}^{2}}.
\end{equation}
Recall that they satisfy the ODE system
\begin{equation}\label{ODEs_restated}
 \begin{cases}
   \dot{\alpha}_{1}=-\sqrt{8}(\alpha^2_{1}+\alpha_{3}^2)+{E}_{1},\\
   \dot{\alpha}_{2}=-\sqrt{8}(\alpha^2_{2}+\alpha_{3}^2)+{E}_{2},\\
    \dot{\alpha}_{3}=-\sqrt{8}(\alpha_{1}+\alpha_2)\alpha_{3}+{E}_{3},\\
    \end{cases}
\end{equation}
where by Proposition \ref{ODE system} (improved error estimate for spectral ODEs)  for all $\tau \leq\tau_0$ with $|\vec{\alpha}(\tau)|\geq e^{-|\tau|^{\gamma/2}}$ the error terms can be estimated by
\begin{equation}\label{error_rest}
|E_{j}(\tau)|\leq \frac{C|\vec{\alpha}(\tau)|^2}{|\tau|^{\gamma/2}},
\end{equation}
provided the bubble-sheet oval under consideration is $\kappa'$-quadratic at time $\tau_0\leq \tau_\ast$, with $\kappa'$ sufficiently small and $\tau_\ast$ sufficiently negative.

To analyze these ODEs, similarly as in \cite[Section 3.2]{DH_hearing_shape}, we consider
\begin{equation}
S:=\alpha_1+\alpha_2,\qquad D:=\alpha_1\alpha_2-\alpha_3^2.
\end{equation}
Using \eqref{ODEs_restated} and \eqref{error_rest}, a direct computation yields
\begin{equation}\label{Riccati ODEs}
    \begin{cases}
      \dot{S}=-\sqrt{8}(S^2-2D)+{F}_{1},\\
       \dot{D}=-\sqrt{8}SD+{F}_{2},
    \end{cases}
    \end{equation}
with the error estimate
\begin{equation}\label{error_rest2}
|F_{1}(\tau)|\leq \frac{CS(\tau)^2}{|\tau|^{\gamma/2}},\qquad |F_{2}(\tau)|\leq \frac{C|S(\tau)|^3}{|\tau|^{\gamma/2}}
\end{equation}
 for all $\tau \in (\tau_1,\tau_0]$, where
 \begin{equation}
 \tau_1:=\inf\left\{ \tau'\leq \tau_0 : 100|S(\tau)|\geq |\vec{\alpha}(\tau)|\geq e^{-|\tau|^{\gamma/2}}\,\,\forall\tau\in[\tau',\tau_0] \right\}.
 \end{equation}
Observe that since our bubble-sheet oval is $\kappa'$-quadratic at time $\tau_{0}$, we have
\begin{align}
     \left| S(\tau_{0})+\frac{1}{\sqrt{2}|\tau_{0}|}\right| \leq \frac{C\kappa'}{|\tau_{0}|},\qquad \left| D(\tau_{0}) -\frac{1}{8|\tau_{0}|^2}\right| \leq \frac{C\kappa'}{|\tau_{0}|^2},
\end{align}
and there is some $\delta>0$ such that $\tau_1\leq \tau_0-\delta$.

To proceed, we change variables to
\begin{equation}\label{phiXY}
   \xi(\sigma):=\left(\begin{array}{c}
       \sqrt{2}\tau(\sigma) S(\tau(\sigma))-1\\
       8\tau(\sigma)^2 D(\tau(\sigma))-1
    \end{array}
    \right) ,\qquad \textrm{where}\,\,\, \tau(\sigma)=-e^\sigma.
\end{equation}
Denoting the components by $\xi_1$ and $\xi_2$, using \eqref{Riccati ODEs} we see that
\begin{align}\label{phi1'}
    \xi_{1}'(\sigma)&=-3\xi_{1}(\sigma)+\xi_{2}(\sigma)-2\xi^{2}_{1}(\sigma)+\sqrt{2}\tau(\sigma)^2{F}_{1}(\tau({\sigma})),\nonumber\\
    \xi_{2}'(\sigma)&=-2\xi_{1}(\sigma)-2\xi_{1}(\sigma)\xi_{2}(\sigma)+8\tau(\sigma)^3{F}_{2}(\tau(\sigma)).
\end{align}
Moreover, if $\tau(\sigma)\in (\tau_1,\tau_0]$ then using \eqref{error_rest2} we can estimate
\begin{multline}
2\xi_{1}^2(\sigma)+2|\xi_{1}(\sigma)\xi_{2}(\sigma)| +\sqrt{2}|\tau(\sigma)^2{F}_{1}(\tau(\sigma))|
   +8|\tau(\sigma)^3{F}_{2}(\tau(\sigma))| \\
   \leq 4 |\xi(\sigma)|^2+Ce^{-\frac12\gamma\sigma}\max\{1, |\tau(\sigma)S(\tau(\sigma))|^3\}.
\end{multline}
Furthermore, observe that by basic linear algebra we have the implication
\begin{equation}
|\xi(\sigma)|\leq 1/10\quad\Rightarrow\quad 100|S(\tau(\sigma))|\geq |\vec{\alpha}(\tau(\sigma))|\geq e^{-|\tau(\sigma)|^{\gamma/2}}.
\end{equation}
Hence, in the new variables we get
\begin{equation}\label{phi_ode}
   \xi'(\sigma)=A\xi(\sigma)+N(\sigma,\xi(\sigma)),\qquad
  |\xi(\sigma_{0})|\leq C\kappa',
\end{equation}
with
\begin{equation}\label{B}
    A=\left(\begin{array}{cc}
        -3 & 1 \\
         -2 & 0
    \end{array}
    \right),
\end{equation}    
and
\begin{align}\label{N_estimates'}
   | N(\sigma, \xi(\sigma))|\leq 4 |\xi(\sigma)|^2+Ce^{-\frac12 \gamma\sigma}
\end{align}
for all $\sigma \in [  \sigma_0,\sigma_1)$, where  $\sigma_0=\log(-\tau_0)$ and
\begin{equation}\label{aprior1_10}
 \sigma_1:=\sup\left\{ \sigma'\geq \sigma_0:  |\xi(\sigma)|\leq 1/10 \,\,\forall\sigma\in[\sigma_0,\sigma'] \right\}.
 \end{equation}
 
Notice that the matrix $A$ has eigenvalues $-2$ and $-1$, so $\xi=0$ is a stable limit point. More precisely, given any $\eps>0$, we have
\begin{equation}
|\xi(\sigma)|\leq \eps
\end{equation}
for all $\sigma\in [  \sigma_0,\sigma_1)$, provided $\kappa'$ is small enough and $\sigma_0$ is large enough (here we also used that $\int_{\sigma_0}^\infty e^{-\frac12 \gamma\sigma}\, d\sigma$ can be made arbitrarily small by choosing $\sigma_0$ large enough). In particular, by continuity this implies $\sigma_1=\infty$, and hence $\tau_1=-\infty$.  Translating back to our original variables, this shows that
\begin{align}
     \left| S(\tau)+\frac{1}{\sqrt{2}|\tau|}\right| \leq \frac{C\eps}{|\tau|},\qquad \left| D(\tau)- \frac{1}{8|\tau|^2}\right| \leq \frac{C\eps}{|\tau|^2},
\end{align}
for all $\tau\leq \tau_{0}$, provided $\kappa'$ is small enough and $\tau_0$ is negative enough. This shows that both eigenvalues of the matrix
\begin{equation}\label{alpha_matrix}
    \left(\begin{array}{cc}
        \alpha_{1}(\tau) & \alpha_{3}(\tau) \\
         \alpha_{3}(\tau) & \alpha_{2}(\tau)
    \end{array}\right)
    \end{equation}
are $\frac{C\eps}{|\tau|}$-close to $\frac{1}{\sqrt{8}|\tau|}$. Hence, choosing $\eps=\eps(\kappa)$ sufficiently small, we conclude that
\begin{align}\label{u_strong_kappa}
     \left\| \hat{u}(y_{1}, y_{2},  \tau)+\frac{y_{1}^2+y_{2}^2-4}{\sqrt{8}|\tau|}  \right\|_{\mathcal{H}}\leq \frac{\kappa/2}{|\tau|}
\end{align}
for all $\tau\leq \tau_{0}$.

Finally, having established \eqref{u_strong_kappa}, we can consider the quantity
\begin{equation}\label{def_beta}
    \beta(\tau):=\sup_{\tau'\leq \tau}\left(\int_{\mathbb{R}^2}\hat{u}(y_{1}, y_{2},  \tau')^2e^{-\frac{|y|^2}{4}}dy\right)^{1/2},
\end{equation}
and argue similarly as in \cite[Section 2.2]{DH_hearing_shape} to upgrade the initial graphical radius $\rho(\tau)=|\tau|^\gamma$ to the improved graphical radius $\rho(\tau)=|\tau|^{1/10}$ for $\tau\leq \tau_{0}$. Observing also that with this new  graphical radius \eqref{u_strong_kappa} still holds with $\kappa/2$ replaced by $\kappa$, this concludes the proof of the theorem.
\end{proof}

As a corollary of the proof we also obtain:

\begin{corollary}[full rank]\label{cor_full_rank}
For every $\kappa>0$ small enough, there exists $\tau_\ast>-\infty$ with the following significance. Let $\mathcal{M}$ be an ancient noncollapsed flow in $\mathbb{R}^4$, whose tangent flow at $-\infty$ is given by \eqref{bubble-sheet_tangent_intro}, and suppose that $\mathcal{M}$ is $\mathrm{SO}(2)$-symmetric in the $x_3x_4$-plane centered at the origin. If $\mathcal{M}$ is $\kappa$-quadratic (defined literally the same as in Definition \ref{k_tau00}) at some time $\tau_0\leq\tau_\ast$, then its fine bubble-sheet matrix $Q$ satisfies $\mathrm{rk}(Q)=2$.
\end{corollary}

\begin{proof}
Indeed, this follows by inspecting the above proof of Theorem \ref{point_strong} (strong $\kappa$-quadraticity).
\end{proof}

Together with the results from the previous subsection, we  get:

\begin{proof}[{Proof of Theorem \ref{strong_uniform0} (uniform sharp asymptotics)}]
This now follows from
Proposition \ref{uniform_par} (parabolic region),
Proposition \ref{prop_interm} (intermediate region)
and Corollary \ref{uniform_tip} (tip region in terms of polar angle), which establish the uniform asymptotics under the a priori of strong $\kappa$-quadraticity, together with Theorem \ref{point_strong} (strong $\kappa$-quadraticity), which justifies this a priori assumption.
\end{proof}

To conclude this section, we note that as a consequence of Theorem \ref{strong_uniform0} (uniform sharp asymptotics) we obtain the following standard cylindrical estimate, which will be used frequently throughout the paper:

\begin{corollary}[cylindrical estimate]
\label{lemma-cylindrical}
 For every $\varepsilon > 0$, there exist $L < \infty$, $\kappa > 0$  and $\tau_* > -\infty$
such that if $\mathcal{M}$ is $\kappa$-quadratic at time $\tau_0 \le \tau_*$, then for all $\tau\leq \tau_0$ we have
\begin{equation}\label{eqn-cyl0}  
\sup_{\{v(\cdot,\tau)\geq L/\sqrt{|\tau|}\}}\max_{1\leq k+\ell \leq 10} \left| v^{k+\ell-1} y ^ {-k} \partial^k_\varphi \partial ^\ell_y v\right | < \varepsilon.
\end{equation}
\end{corollary}

\begin{proof}
Observe first that by Theorem \ref{strong_uniform0} (uniform sharp asymptotics), for every $\eps_1>0$ there exist $L_1<\infty$, such that for $\tau\leq\tau_0$ sufficiently negative we have
\begin{equation}\label{first_der_cyl_st}
\sup_{\{v(\cdot,\tau)\geq L_1/\sqrt{|\tau|}\}} |\partial_y v|+ \sup_{v(\cdot,\tau)\geq \eps_1}|y^{-1}\partial_\varphi v| \leq \eps_1,
\end{equation}
since otherwise some tangent plane would enter the region enclosed by the oval, contradicting convexity. To show that the angular derivative is small in the collar region as well, we work with the inverse profile function $Y$ defined by
$Y(v(y,\varphi,\tau),\varphi,\tau)=y$.
Note that $Y_vv_y=1$ and $Y_vv_\varphi +Y_\varphi =0$.  Now, considering the convex polar curves represented by $r(\varphi):= {(2|\tau|)}^{-1/2} Y(v,\varphi,\tau)$ for given $v$ and $\tau$, and arguing similarly as in the proof of Corollary \ref{uniform_tip}, we see that
\begin{equation}\label{eqn-Yphi0}
\sup _{v\leq 2\theta}\sup_{\varphi } \, \big | Y_\vp (v,\varphi,\tau)\big |<\varepsilon_1 \, \sqrt{|\tau|},
\end{equation}
provided  $\theta >0$ is chosen sufficiently small and $\tau \leq \tau_0$ is sufficiently negative. Hence, both suprema in \eqref{first_der_cyl_st} can be taken over $\{v(\cdot,\tau)\geq L_1/\sqrt{|\tau|}\}$.

Now, suppose towards a contradiction that for some $\varepsilon > 0$ there are bubble-sheet ovals $\mathcal{M}^i$ that are $\kappa_i$-quadratic at time $\tau_{i,0}\to -\infty$, where $\kappa_i\to 0$ and $\tau_{i,0}\to -\infty$, but such that for some $\tau_i \leq \tau_{i,0}$ there are $y_i,\varphi_i$ with
\begin{equation}
\label{eq-away-tip}
v_i(y_i,\varphi_i,\tau_i) \, \sqrt{|\tau_i|} \to \infty,
\end{equation}
but 
\begin{equation}
\label{eq-contra-ineq}
\max_{1\leq k+\ell \leq 10} \left| {v_i}^{k+\ell-1} y ^ {-k} \partial^k_\varphi \partial ^\ell_y {v_i}\right |(y_i,\varphi_i,\tau_i)\geq \eps.
 \end{equation}
Set $t_i = -e^{-\tau_i}$, and let $p_i \in M^i_{t_i}$ be points in the unrescaled flow corresponding to points with coordinates $(y_i,\varphi_i)$ in the renormalized flow $\bar{M}^i_{\tau_i}$. By the noncollapsing property we have
\begin{equation}
\label{eq-rescaling-factor}
H^i(p_i,t_i) \ge \frac{c}{\sqrt{|t_i|} {v_i}(y_i,\varphi_i,\tau_i)}
\end{equation}
for some $c >0$. 
Let $\widetilde{M}_t^i$ be the sequence of flows obtained from $M^i_t$ by shifting $(p_i,t_i)$ to the origin, and parabolically rescaling by $H^i(p_i,t_i)$. By the global convergence theorem \cite[Theorem 1.12]{HaslhoferKleiner_meanconvex} we can pass to a subsequential limit $\widetilde{M}_t^{\infty}$. It follows from the first derivative estimate from above, together with \eqref{eq-away-tip} and \eqref{eq-rescaling-factor} that $\widetilde{M}_t^{\infty}$ splits off 2 lines. Hence, $\widetilde{M}_t^{\infty}$ must be a round shrinking bubble-sheet. For $i$ large enough this contradicts \eqref{eq-contra-ineq}, and thus proves the corollary.
\end{proof}

\bigskip


\section{{Quadratic almost concavity}}\label{sec_concavity}

The goal of this section is to prove the quadratic almost concavity estimate and its corollary. Throughout this section, it will be most convenient to work with the original flow $M_t$. Thanks to the tangent flow property \eqref{bubble-sheet_tangent_intro} and the $\mathrm{SO}(2)$-symmetry we can then parametrize our bubble-sheet ovals via
\begin{equation}
(x_1,x_2,\vartheta)\mapsto (x_1,x_2,V(x_1,x_2,t)\cos\vartheta,V(x_1,x_2,t)\sin\vartheta).
\end{equation}
In these $(x_1,x_2,\vartheta)$-coordinates the metric takes the form
\begin{equation}
\label{eq-metric}
g = 
\begin{bmatrix}
 1+ V_{x_1}^2 &  V_{x_1}V_{x_2} & 0 \\
 V_{x_1}V_{x_2} & 1+V_{x_2}^2 & 0 \\
 0 & 0 & V^2 
\end{bmatrix}.
\end{equation}
Hence, the inverse metric is
\be \label{eq-ginverse}
g^{-1} = \frac{1}{1+|DV|^2}\,
\begin{bmatrix}
 1+ V_{x_2}^2 & - V_{x_1}V_{x_2} & 0 \\
 -V_{x_1}V_{x_2} & 1+V_{x_1}^2 & 0 \\
 0 & 0 & \frac{ 1+|DV|^2}{V^2}
\end{bmatrix},
\ee 
where 
\be
|DV|^2:=V_{x_1}^2 + V_{x_2}^2.
\ee
Furthermore, observing that the outward unit normal is
\begin{equation}
\nu=\frac{(-V_{x_1}, -V_{x_2}, \cos\vartheta, \sin \vartheta)}{\sqrt{1+|DV|^2}}\, ,
\end{equation}
we see that the second fundamental form is given by
\be \label{eq-h}h = \frac{1}{\sqrt{1+|DV|^2}}\,
\begin{bmatrix}
-V_{x_1x_1} & -V_{x_1x_2} & 0 \\
-V_{x_1 x_2} & -V_{x_2x_2} & 0 \\
0 & 0 &  V
\end{bmatrix}.
\ee 
In particular, convexity of our hypersurfaces $M_t$ is now captured by the analytic condition that $(x_1,x_2)\mapsto V(x_1,x_2,t)$ is concave and nonnegative.\footnote{Moreover, note that $-(1+|DV|^2)^{1/2}\textrm{tr}_g h=\Big(\delta_{ij}-\frac{V_{x_i}V_{x_j}}{1+|DV|^2}\Big)V_{x_ix_j}-\frac{1}{V}$, which is of course consistent with the evolution equation \eqref{equation_u} for the renormalized profile function.}

Observing that
\be\label{obs_grad}
|\nabla V|^2 :=g^{ij}  V_{i} V_{j} = \frac{|DV|^2}{1+ |DV|^2},
\ee
throughout this section we will abbreviate
\be\label{abbr_gra}
\gra := (1+|DV|^2)^{1/2}= (1-|\nabla V|^2)^{-1/2}.
\ee
Finally, throughout this section we will work in the region where $V>L\sqrt{|t|/\log |t|}$ (this will be justified below in Proposition \ref{lemma-soliton-region}).

\bigskip

 \subsection{Intrinsic quantities and their evolution}\label{sec_intrinsic}
 Throughout this subsection, we will view the unrescaled profile function as an intrinsic time-dependent quantity on $S^3$. Specifically, denoting the parametrized mean curvature flow by $F_t:S^3\to\mathbb{R}^4$, so that $M_t=F_t(S^3)$, we set
 \begin{equation}\label{intrinsic_prof}
 V(p,t):=F_t(p)\cdot \omega\left(F_t(p)\right),
 \end{equation}
 where $\omega$ denotes the vector field in $\mathbb{R}^4$ defined by
 \begin{equation}
 \omega(x_1,x_2,x_3,x_4):=\big(0,0,x_3/(x_3^2+x_4^2)^{1/2},x_4/(x_3^2+x_4^2)^{1/2}\big).
 \end{equation}
To begin with, we observe that we can express the second fundamental form in term of the intrinsic Hessian of the profile function:

\begin{lemma}[second fundamental form]
\label{lemma-equivalent} We have
\be \label{eq-hij} h_{ij} = -\eta \nabla^2_{ij}V + \eta V\,  \nabla_i \vartheta \nabla_j \vartheta   \, .  \ee  
\end{lemma}   

\begin{proof}We work in the $(x_1,x_2,\vartheta)$-coordinates. Using the standard formula for the Christoffel symbols,
	\begin{equation}
	\label{eq-Chris}
	\Gamma_{ij}^k = \frac12 g^{k\ell }\big(\partial_i g_{j\ell}+\partial_j g_{i\ell} - \partial_\ell g_{ij}\big),
	\end{equation}
and the equations \eqref{eq-metric} and \eqref{eq-ginverse}, for $i,j,k\in \{1,2\}$ we get
\be\label{Christoffel}
\Gamma_{ij}^k = \frac{V_{x_k} V_{x_ix_j}}{1+|DV|^2},\qquad \Gamma_{i3}^k = 0,\qquad \Gamma_{33}^k  =  \frac{-V V_{x_k}}{1 + |DV|^2}.
\ee
Remembering the formula $\nabla^2_{ij}V = \partial_i \partial_j V - \Gamma_{ij}^k \partial_k V$ this yields
		\be \label{eq-nablasqu} \nabla^2  V = (1+|DV|^2)^{-1} \begin{bmatrix}
V_{x_1x_1} & V_{x_1x_2} & 0 \\
V_{x_1x_2} & V_{x_2x_2} & 0 \\
0 & 0 & |DV|^2 V
\end{bmatrix}.\ee 
Together with \eqref{eq-h} and \eqref{abbr_gra} this implies the assertion.
\end{proof}

We will now compute the evolution equations of several intrinsic quantities. Throughout, we will briefly denote by $\Delta=\Delta_{g(t)}$ the Laplace-Beltrami operator with respect to the metric $g(t)$ induced by the embedding $F_t$.

\begin{proposition}[evolution of profile function]
\label{lemma-ev-eq}
The profile function $V$, considered as an intrinsic quantity, satisfies the evolution equation
\be
 (\partial _t - \Delta ) V= - V^{-1}.
 \ee
 \end{proposition}
 
\begin{proof} 
Under the flow we have $\partial _t \omega(F_t)= 0$ due to the symmetry. Also, for functions $f$ that depend only on the angle $\vartheta$, we have $\Delta f =  V^{-2} f_{\vartheta\vartheta} $. Hence, \be (\partial_t - \Delta) \omega(F_t)  = \omega(F_t) V^{-2} .   \ee
Together with the mean curvature flow equation $\partial_t F_t  = \Delta F_t$ this yields
 \be
 (\partial_t - \Delta ) (F_t\cdot \omega(F_t) )
	=  F_t\cdot \omega(F_t)V^{-2} - 2 g^{\vartheta\vartheta}  \partial_\vartheta F_t\cdot \partial_\vartheta\omega(F_t)=-V^{-1}.
\ee	
This proves the proposition.
\end{proof}

We will now compute the evolution of the intrinsic Hessian $\nabla^2_{ij} Q$, where
\begin{equation}
Q:=V^2
\end{equation}
denotes the square of the profile function from \eqref{intrinsic_prof}.  As usual in tensor computations, we use the extended summation convention, where indices are raised using the metric and summed over, e.g. $h_{ip}h_{pk}=\sum_{j,\ell=1}^3 h_{ij}g^{j\ell} h_{\ell k}$.

\begin{proposition}[evolution of Hessian]\label{prop_evol_hess} The Hessian of the square of the profile function, viewed as an intrinsic function, evolves by
\begin{align}\label{evolution_Qij}
&\!\!\!(\partial_t - \Delta) \nabla^2_{ij}Q=-Q^{-1}\nabla_k  Q\,  \nabla_k \nabla^2_{ij}Q - Q^{-1}\nabla^2_{ik}Q  \nabla^2_{jk} Q+
\tfrac{1}{2} Q^{-2} |\nabla  Q|^2 \nabla^2_{ij} Q\\
 & + Q^{-2}(\nabla_i Q\nabla_k Q \nabla^2_{jk} Q  + \nabla_j Q\nabla_k Q  \nabla^2_{ik} Q)
   - Q^{-3}|\nabla  Q|^2\nabla_i  Q \nabla_j Q \nonumber\\
 & +2(h_{ij} h_{pq} - h_{iq}h_{jp}) \nabla^2_{pq} Q - (H h_{ik}-h_{ip} h_{pk}  ) \nabla^2_{jk} Q- (H h_{jk}-h_{jp} h_{pk}  )\nabla^2_{ik} Q\nonumber\\
  & +2 h_{kp}\nabla_k h_{ij}\nabla_p Q  - (h_{ij}h_{kp}-h_{ip}h_{jk})Q^{-1}\nabla_k Q \nabla_p Q.\nonumber
\end{align}
\end{proposition}

 \begin{proof}
Applying Proposition \ref{lemma-ev-eq} (evolution of profile function) yields
 \begin{equation}
 \label{eq-f}
 (\partial_t-\Delta) Q = -\tfrac{1}{2}Q^{-1}|\nabla Q|^2 - 2.
 \end{equation}
Differentiating we get
\begin{equation}
\nabla_{j} (\partial_t-\Delta) Q =- Q^{-1}\nabla_k Q\nabla^2_{jk}Q+\tfrac{1}{2}Q^{-2}|\nabla Q|^2 \nabla_j Q ,
\end{equation}
and differentiating again we obtain
\begin{align} \label{comp_begin_with}
\nabla^2_{ij} (\partial_t-\Delta)& Q = - Q^{-1}\nabla_k Q {\nabla_i\nabla^2_{jk} Q}- Q^{-1}\nabla^2_{ik}Q  \nabla^2_{jk} Q + \tfrac{1}{2}Q^{-2}|\nabla  Q|^2 \nabla^2_{ij} Q\nonumber\\
&+ Q^{-2}(\nabla_i Q\nabla_k Q \nabla^2_{jk} Q  + \nabla_j Q\nabla_k Q  \nabla^2_{ik} Q) - Q^{-3}|\nabla  Q|^2\nabla_i  Q \nabla_j Q.
\end{align}
Hence, our main task is to compute the commutator of the heat operator and the Hessian.

In general, the time derivative of the Hessian of a function $f$ equals
\begin{equation}
\partial_t \big(\nabla^2_{ij} f\big) 
 = \nabla^2_{ij}  \big(\partial_t f\big) - \big(\partial_t \Gamma_{ij}^k\big) \,\partial_k f\, ,
\end{equation}
where the variation of the Christoffel symbols is given by the formula (see e.g.  \cite[Lemma 2.27]{MR2274812})
\begin{equation}\label{eq-firstvarchristoffel}
\partial_t \Gamma_{ij}^k= \frac{1}{2}g^{k\ell}\big(\nabla_i (\partial_t g_{j\ell}) +\nabla_j (\partial_t g_{i\ell}) -\nabla_\ell (\partial_t g_{ij})  \big).
\end{equation}
Since under mean curvature flow we have $\partial_t g_{ij} = -2H h_{ij}$ together with the Codazzi equation $\nabla_i h_{jk}=\nabla_j h_{ik}$ this yields
 \begin{multline}
 \label{eq-comm-tder}
\partial_t \big(\nabla^2_{ij} Q\big) -\nabla^2_{ij}\big(\partial_t Q\big)\\
 =    \big(H \nabla_i h_{jk}  - h_{ij} \nabla_k H +  h_{ik} \nabla_j H  +  h_{jk} \nabla_i H \big)\, \nabla_k Q.
 \end{multline}
On the other hand, thanks to the second Bianchi identity we have the general commutator formula (see e.g. \cite[Equation (2.34)]{MR2274812})
 \begin{multline}
 \label{comm-lap-hess}
 \nabla^2_{ij} \Delta f - \Delta  \nabla^2_{ij} f 
 =2R_{i p j q}\nabla^2_{pq} f-R_{ik} \nabla^2_{jk} f-R_{jk} \nabla^2_{ik} f\\
 -(\nabla_i R_{jk}+\nabla_j R_{ik}-\nabla_k R_{ij})\nabla_k f,
 \end{multline}
where $R_{ijk\ell}$ is the Riemann tensor and $R_{ik}=R_{i p k p}$ is the Ricci tensor. In our setting of hypersurfaces, we in addition have the Gauss equation
\begin{equation}
R_{ijk\ell} = h_{ik} h_{j\ell} - h_{i\ell}h_{jk},
\end{equation}
and its trace
\begin{equation}
R_{ik} = H h_{ik}-h_{ij}h_{jk}.
\end{equation}
In particular, together with the Codazzi equation this yields
\begin{multline}
\nabla_i R_{jk}+\nabla_j R_{ik}-\nabla_k R_{ij}\\
=\big(H \nabla_i h_{jk}  - h_{ij} \nabla_k H +  h_{ik} \nabla_j H  +  h_{jk} \nabla_i H \big)-2h_{kp}\nabla_p h_{ij} .
\end{multline}
Combining the above formulas we infer that
\begin{multline}
 \label{eq-comm-heat-op}
\!\!\!\!(\partial_t -\Delta) \nabla^2_{ij} Q - \nabla^2_{ij} (\partial_t -\Delta)Q=  2 h_{kp}\nabla_p h_{ij}\nabla_k Q\\
\qquad +2(h_{ij} h_{pq} - h_{iq}h_{jp})\, \nabla^2_{pq} Q  - (  H h_{ik}-h_{ip} h_{pk} )\, \nabla^2_{jk} Q - (H h_{jk} -h_{jp} h_{pk})\, \nabla^2_{ik} Q .
\end{multline}
Together with \eqref{comp_begin_with}, where we rewrite the third derivative term using
\begin{equation}
{\nabla_i\nabla^2_{jk} Q}=\nabla_{k}\nabla^2_{ij} Q+(h_{ij}h_{kp}-h_{ip}h_{jk})\nabla_p Q,
\end{equation}
this yields the assertion.
 \end{proof}
 
To make the evolution equation more feasible for applying the maximum principle, we rewrite the Codazzi term from the last line of \eqref{evolution_Qij} in terms of $V$:

 \begin{lemma}[Codazzi term] \label{lemma-nabla-h}
 We have
  \bea 2  h_{kp}\nabla_k h_{ij}\nabla_p Q = -2\gra  \nabla_p V h_{pk} \nabla_k \nabla^2_{ij}Q + \Psi_{ij},   \eea
  where
   \begin{multline}   \label{eq-psi}
   \Psi_{ij} = 4\gra  \nabla_pV h_{pk}( \nabla_k V \nabla^2_{ij}V+\nabla_j V \nabla^2_{ik}V +\nabla_i V \nabla^2_{jk}V ) \\
   - 4 \nabla_p V h_{pk}(\gra ^3 V \nabla^2_{kq}V \nabla_q V \nabla^2_{ij}V-V G_{kij} ),
 \end{multline}
and $G_{kij}=G_{kji}$ is a 3-tensor that satisfies $G_{kij}X^iX^j=0$ for all $X\perp \partial_\vartheta$.
 \end{lemma}
 
 \begin{proof} Using Lemma \ref{lemma-equivalent} (second fundamental form) we see that
  \bea 
 \label{eq-nabla-h-U}
 \nabla_k h_{ij}  = -\gra  ^3 \nabla_q V \nabla^2_{kq}V\nabla^2_{ij}V - \gra  \nabla_k\nabla^2_{ij}V + G_{kij} \, , \eea 
 where
  \begin{equation}
 \label{eq-Gijk}
G_{kij}=\nabla_k(\eta V)\nabla_i \vartheta\nabla_j\vartheta+\eta V \nabla_{i}\vartheta \nabla^2_{jk}\vartheta  +\eta V \nabla^2_{ik}\vartheta \nabla_{j}\vartheta\, .
 \end{equation}
In particular, note that $G_{kij}X^iX^j=0$ whenever $X\perp \partial_\vartheta$. Moreover, substituting $Q=V^2$ we see that
\begin{equation}
\nabla _k \nabla^2_{ij}Q = 2 ( V\nabla_k \nabla^2_{ij}V+\nabla_k V \nabla^2_{ij}V   +\nabla_i V \nabla^2_{kj}V +\nabla_j V \nabla^2_{ki}V ).
\end{equation}
Combining the above facts yields the assertion.
 \end{proof}
 
 Given any $\delta>0$, we now consider the tensor
 \be A_{ij} =\nabla_{ij}^2Q -(\gamma+\delta) g_{ij},\ee
where
\begin{equation}
 \label{eq-Q}
\gamma:= \left(\frac{-t}{\log(-t)}\right)^{3/2} V^{-3}.
 \end{equation}
The tensor $A_{ij}$  will be used in  the proof of Theorem \ref{prop-concavity_restated} in the next subsection. 
 
\begin{corollary}[evolution of $A$-tensor]
\label{lemma-eq-A}
We have 
 \bea \label{eq-Aij} (\partial_t - \Delta)& A_{ij}=  - ( Q^{-1} {\nabla_ k Q}+2 \gra  \nabla_p V h_{pk}) \nabla_k A_{ij} + N_{ij}\, ,\\
\eea 
where
\begin{multline} \label{eq-tilde-N}
 N_{ij} = \hat N_{ij}  + 2(\gamma+\delta) Hh_{ij}\\
  -\frac{3}{V^2}  \left[1-6|\nabla V|^2-2 \eta V h(\nabla V,\nabla V) +\frac{V^2}{2t}  \left(1-\frac{1}{\log(-t)}\right) \right] \gamma {g_{ij}},
   \end{multline}
with  $\gamma$ given by \eqref{eq-Q}, and where
  \begin{align}
 \label{eq-hat-N}
 \hat N_{ij}
=& - Q^{-1}\nabla^2_{ik}Q  \nabla^2_{jk} Q+
\tfrac{1}{2} Q^{-2} |\nabla  Q|^2 \nabla^2_{ij} Q
  + Q^{-2}(\nabla_i Q\nabla_k Q \nabla^2_{jk} Q  + \nabla_j Q\nabla_k Q  \nabla^2_{ik} Q)\\
 &  - Q^{-3}|\nabla  Q|^2\nabla_i  Q \nabla_j Q + \Psi_{ij} {- (h_{ij} h_{kp}-h_{ip} h_{jk})Q^{-1}\nabla_k Q\nabla_p Q}\nonumber  \\
  &  + 2(h_{ij}h_{kp} - h_{ip}h_{jk}) \nabla^2_{kp}Q+ (h_{ik}h_{kp}-h_{ip}H) \nabla^2_{jp}Q +(h_{jk}h_{kp}-h_{jp}H) \nabla^2_{ip}Q \, ,\nonumber 
 \end{align}
 with $\Psi_{ij}$ given by \eqref{eq-psi}.
 \end{corollary}

\begin{proof}
We have already seen in Proposition \ref{prop_evol_hess} (evolution of Hessian) and Lemma \ref{lemma-nabla-h} (Codazzi term) that
\begin{equation}
(\partial_t -\Delta )\nabla^2_{ij} Q =  - ( Q^{-1} {\nabla_ k Q}+2 \gra  \nabla_p V h_{pk}){ \nabla_k \nabla^{2}_{ij}Q} + \hat{N}_{ij},
\end{equation}
where $\hat{N}_{ij}$ is given by \eqref{eq-hat-N}. Next, observe that
\begin{multline}
(\partial_t -\Delta )((\gamma+\delta)g_{ij})= -2(\gamma+\delta)Hh_{ij}\\
+ \left[  \left(\frac{-t}{\log(-t)}\right)^{3/2} (\partial_t - \Delta) V^{-3} +\frac{3}{2t}\left(1-\frac{1}{\log(-t)}\right)\gamma \right] g_{ij}
\end{multline}
Moreover, using Proposition \ref{lemma-ev-eq} (evolution of profile function) we get
\begin{align}
(\partial_t - \Delta) V^{-3}=3V^{-5}(1-4|\nabla V|^2).
\end{align}
Furthermore, we have
\be
( Q^{-1} {\nabla_ k Q}+2 \gra  \nabla_p V h_{pk}) \nabla_k   { \gamma}= -6V^{-2}|\nabla V|^2 \gamma - 6\eta V^{-1}h(\nabla V,\nabla V) \gamma\, .
\ee
Combining the above facts the assertion follows.
\end{proof}

\bigskip
 
 \subsection{Proof of the quadratic almost concavity estimate} \label{sec-apriori}
 
 We will now prove the quadratic almost concavity estimate and its corollary. Similarly as in the previous subsection, we work with the tensor
  \be \label{eq-defAijrest} A =\nabla^2Q -\eps g,\ee
 where $Q=V^2$ is the square of the profile function viewed as an intrinsic function on $S^3$, and where
 \be
 \eps=\gamma+\delta\, .
 \ee
Also recall that $\gamma$ denotes the function defined in \eqref{eq-Q}, and $\delta$ is an arbitrarily small positive constant.
We have seen in Corollary \ref{lemma-eq-A} (evolution of $A$-tensor) that $A$ satisifies an evolution equation of the form
\begin{equation}
\label{eq-right-form}
(\partial_t - \Delta+\nabla_Z) A =  N.
\end{equation}
We will show that $A(X,X)\leq 0$ for all $X\perp \partial_\vartheta$ is preserved along the flow. To this end, we will first estimate the reaction term when evaluated on null eigenvectors of $A$, away from a certain region, then check the sign in the region that has been excluded in the first step, and then conclude by adapting the proof of Hamilton's tensor maximum principle to our setting.

\begin{lemma}[null eigenvector]\label{lemma_null_eigen}
If $X\perp \partial_\vartheta$ is a null eigenvector of $A$, then we have the inequality
\begin{equation}\label{null_useful_inequ}
2\eps |X|^2 \leq Q^{-1}|\nabla_X Q|^2.
\end{equation}
Furthermore, we have the identities\footnote{Recall that throughout this section, $\eta$ denotes the quantity defined in \eqref{abbr_gra}.}
\begin{align}\label{eq-h-ident}
h(X,X)= \tfrac14 \eta V^{-1} \left(Q^{-1}|\nabla_X Q|^2-2\eps |X|^2 \right)\, ,
\end{align}
and
\be\label{eq-h2-ident}
h^2(X,\nabla V)= \gra V^{-1}\Big(  h(\nabla V,\nabla V) -\tfrac{1}{8}  \eps \gra V^{-1}\left(Q^{-1}|\nabla Q|^2-2\eps\right) \Big)\nabla_X V.
\ee
\end{lemma}

\begin{proof}
By  the null eigenvector assumption we have
 \begin{equation}\label{eq_null_eigen}
\nabla^2Q(X,Y)  = \eps g(X,Y)
 \end{equation}
 for all $Y$. Substituting $Q=V^2$, we thus get
\begin{equation}\label{eq_null_eigen_restated}
\nabla^2V(X,Y) = V^{-1}\left(\tfrac{1}{2}\eps \langle X,Y\rangle - \nabla_X V \nabla_Y V\right)\, .
\end{equation}
Since $\nabla^2 V(X,X)\leq 0$ thanks to the assumption $X\perp\partial_\vartheta$, this implies 
\begin{equation}
\tfrac{1}{2}\eps |X|^2\leq |\nabla_X V|^2,
\end{equation}
which yields \eqref{null_useful_inequ}. Moreover, using again the assumption $X\perp\partial_\vartheta$, Lemma \ref{lemma-equivalent} (second fundamental form) and \eqref{eq_null_eigen_restated} we compute
\be
h(X,X)=-\gra \nabla^2V(X,X)=-\gra V^{-1}\left(\tfrac12 \eps |X|^2-|\nabla_X V|^2\right)\, ,
\ee
which yields \eqref{eq-h-ident}. Finally, arguing similarly we compute
\begin{align}
h^2(X,\nabla V) &=\gra^2V^{-1}  \sum  \left(\tfrac{1}{2}\eps \langle X, e_j\rangle - \nabla_X V \nabla_{e_j} V\right)\nabla^2 V(e_j,\nabla V)\nonumber\\
& =\tfrac{1}{2}\eps  \eta^2 V^{-1} \nabla^2V(X,\nabla V)+ \gra V^{-1}  h(\nabla V,\nabla V) \nabla_X V\, ,
\end{align}
and
\begin{equation}
 \nabla^2V(X,\nabla V)=V^{-1}\left(\tfrac12 \eps -  |\nabla V|^2\right) \nabla_XV
\end{equation}
This yields \eqref{eq-h2-ident}, and thus concludes the proof of the lemma.
\end{proof}

\begin{proposition}[reaction term]
\label{lemma-full-null-cond}
For all $\zeta>0$, there exist constants $\kappa>0$, $\tau_*>-\infty$ and $L<\infty$, which are all independent of the parameter $\delta$, with the following significance. If $\mathcal{M}$ is  $\kappa$-quadratic at time $\tau_0 \le \tau_*$, then for all times $t\le -e^{-\tau_0}$ in the region $\{ L\sqrt{|t|/\log|t|}\leq V\leq \sqrt{2|t|}-L\sqrt{|t|}/\log|t| \}$ we have the following. If $X\perp \partial_\vartheta$ is a null eigenvector of $A$, then
\begin{equation}
N(X,X) \leq - \tfrac 12\, (1-\zeta-\eps)\, Q^{-2}|\nabla_XQ|^2(Q^{-1}|\nabla Q|^2-2\eps) \, .
\end{equation}
\end{proposition} 

 \begin{proof}
During the proof we will frequently use Lemma \ref{lemma_null_eigen} (null eigenvector), which together with strict convexity in particular implies that
\be
Q^{-1}|\nabla Q|^2-2\eps> 0.
\ee
We start by computing $\hat{N}(X,X)$, where $\hat{N}$ is from equation \eqref{eq-hat-N}. To this end, note that
 \begin{equation}
 2(h_{ij}h_{kp} - h_{ip}h_{jk}) g_{kp}+ (h_{ik}h_{kp}-h_{ip}H) g_{jp} +(h_{jk}h_{kp}-h_{jp}H) g_{ip}=0.
 \end{equation}
Together with the null eigenvector assumption this yields
\begin{multline}
(h_{ik}h_{kp}-h_{ip}H) \nabla^2_{jp}Q\, X^{i}X^j 
+(h_{jk}h_{kp}-h_{jp}H) \nabla^2_{ip}Q\, X^iX^j\\
=-2(h_{ij}h_{kp} - h_{ip}h_{jk})\eps g_{kp}X^{i}X^j.
\end{multline}
Also note that
\be
2\nabla^2_{kp}Q-Q^{-1}\nabla_k Q\nabla_p Q=4 V\nabla^2_{kp}V.
\ee
Hence, using the null eigenvector assumption again we infer that
\begin{align} \label{eq-363}
           \hat{N}(X,X) =& -(Q^{-1}|\nabla Q|^2-2\eps) \left(Q^{-2} |\nabla_X Q|^2-\tfrac{1}{2}\eps Q^{-1}|X|^2\right) \nonumber\\
&+    \Psi(X,X)   +B(X,X) \, ,
\end{align}
where
\begin{equation}
B(X,X)=2(h_{ij}h_{kp}-h_{ip}h_{jk})(2 V\nabla^2_{kp}V-\eps g_{kp})X^iX^j.
\end{equation}

To estimate $B$, recalling the block-diagonal structure from the beginning of this section, we choose an orthonormal basis $e_1,e_2,e_3$ for $\nabla^2V$ such that
\begin{equation}
\nabla^2V(e_k,e_p)=\lambda_k \delta_{kp},\qquad g(e_k,e_p)=\delta_{kp},\qquad e_1,e_2\perp\partial_\vartheta\, .
\end{equation}
Using such a basis we can express our quantity as
\begin{multline}\label{new_bad_term}
B(X,X)=\sum_{k=1}^2 2\left( h(X,X)h(e_k,e_k)-h(X,e_k)^2\right) (2 V\lambda_k-\eps)\\
+2 h(X,X)h(e_3,e_3)(2V\lambda_3-\eps).
\end{multline}
Now, since $h$ is positive definite, we have the Cauchy Schwarz inequality
\begin{equation}
h(X,e_k)^2\leq h(X,X)h(e_k,e_k).
\end{equation}
Together with the fact that $\lambda_1 \leq 0$ and $\lambda_2\leq 0$ this shows that the term in the first line of \eqref{new_bad_term} has the good sign. To deal with the term in the second line, note that by \eqref{eq-metric}, \eqref{eq-h} and \eqref{abbr_gra} we have
\be
h(e_3,e_3)=g(\partial_\vartheta,\partial_\vartheta)^{-1}h(\partial_\vartheta,\partial_\vartheta)=\eta^{-1}V^{-1}\, ,
\ee
and that using also  \eqref{obs_grad} and \eqref{eq-nablasqu} we get
\be
\lambda_3=g(\partial_\vartheta,\partial_\vartheta)^{-1}\nabla^2V(\partial_\vartheta,\partial_\vartheta)=|\nabla V|^2 V^{-1}\, .
\ee
Together with \eqref{eq-h-ident} this yields
\begin{equation} \label{eq-370}
B(X,X)\leq  \tfrac{1}{4}(Q^{-1}|\nabla Q|^2-2\eps) \left(Q^{-2} |\nabla_X Q|^2-2\eps Q^{-1}|X|^2\right). 
\end{equation}
Hence, remembering \eqref{eq-363}, we infer that
\begin{align}\label{n-psi-ineq}
\hat{N}(X,X)-\Psi(X,X)\leq & -(Q^{-1}|\nabla Q|^2-2\eps) \left(Q^{-2} |\nabla_X Q|^2-\tfrac{1}{2}\eps Q^{-1}|X|^2\right) \nonumber\\
&- (Q^{-1}|\nabla Q|^2-2\eps) \left(-\tfrac{1}{4} Q^{-2} |\nabla_X Q|^2+\tfrac{1}{2}\eps Q^{-1}|X|^2\right)\nonumber\\
= & -\tfrac{3}{4}Q^{-2} |\nabla_X Q|^2 (Q^{-1}|\nabla   Q|^2-2\eps)\, .
\end{align}

Next, using \eqref{eq-psi} and the assumption $X\perp\partial_\vartheta$ we compute
   \begin{multline}  
   \Psi(X,X) = 4\gra\,  h_{pq}\nabla_p V\nabla_q V \nabla^2_{ij}VX^iX^j+8\gra  \nabla_pV h_{pq}\nabla_j V \nabla^2_{iq}VX^i X^j  \\
   - 4\gra ^3V \nabla_p V h_{pk}  \nabla^2_{kq}V \nabla_q V \nabla^2_{ij}V X^iX^j.
 \end{multline}
To write this in a more useful way, observe that since $X^3=0$ and $\partial_\vartheta V=0$ the sums only run over indices $i,j,p,q,k\in\{1,2\}$. Hence, by Lemma \ref{lemma-equivalent} (second fundamental form) we can replace $\nabla^2V$ by $-\gra^{-1}h$, yielding
   \begin{multline}  
   \Psi(X,X) =- 4\,  h(\nabla V,\nabla V) h(X,X)-8   h^2(X,\nabla V) \nabla_X V  \\
   - 4\gra  V h^2(\nabla V,\nabla V)  h(X,X).
 \end{multline}
Since $h$ is positive definite, the first term and the third term have the good sign, and we can estimate the second term using \eqref{eq-h2-ident}. This yields
\be\label{psi-ineq}
\Psi(X,X)\leq \tfrac{1}{4}\eta^2 \eps Q^{-2} |\nabla_X Q|^2 (Q^{-1}|\nabla Q|^2-2\eps).
\ee

\bigskip

On the other hand, by \eqref{eq-tilde-N} we have
\begin{multline} \label{eq_on_the_other_hand}
 (N- \hat N)(X,X) = 2\eps Hh(X,X) \\
   -\frac{3}{V^2}  \left[1-6|\nabla V|^2-2 \eta V h(\nabla V,\nabla V) +\frac{V^2}{2t}  \left(1-\frac{1}{\log(-t)}\right) \right]  \gamma |X|^2\, .
\end{multline}
To estimate this, note that by Corollary \ref{lemma-cylindrical} (cylindrical estimate) in the region $\{   V \geq L\sqrt{|t|/\log|t|}\}$ we have
\begin{equation}\label{grad_basic_o1}
|DV|=o(1), \quad H=(1+o(1))V^{-1}
\end{equation}
where $o(1)$ denotes terms that can be made arbitrarily small by choosing $\kappa>0$ small enough, $\tau_*>-\infty$ negative enough, and $L<\infty$ large enough. Moreover, by Proposition \ref{strong_uniform0} (sharp asymptotics) and convexity, far away from the tip region we have a sharper gradient estimate, specifically 
\begin{equation}
\sup_{\{V^2/(-2t)\geq 1/2\}}|DV|\leq \frac{C}{\sqrt{-\log(-t)}},
\end{equation}
where $C<\infty$ is a uniform constant.
Hence, restricting to the region $\{ L\sqrt{|t|/\log|t|}\leq V\leq \sqrt{2|t|}-L\sqrt{|t|}/\log|t| \}$, for $L<\infty$ sufficiently large, the term in the second line of \eqref{eq_on_the_other_hand} has the good sign. Therefore, using the identities \eqref{null_useful_inequ} and \eqref{eq-h-ident} and the estimate \eqref{grad_basic_o1}
we conclude that
\begin{equation}\label{eqn-N-hatN}
 (N- \hat N)(X,X) \leq   \tfrac{1}{4} (1+o(1)) Q^{-2} \, |\nabla_XQ|^2 \, (Q^{-1}|\nabla   Q|^2-2\eps). 
\end{equation}
Together with the estimates \eqref{n-psi-ineq} and \eqref{psi-ineq}, taking into account again the fact that $\eta=1+o(1)$ thanks to \eqref{grad_basic_o1}, this proves the proposition.
\end{proof}

Next, we check the sign in the excluded region:

\begin{proposition}[sign in excluded region]\label{lemma-soliton-region} For every $L<\infty$ there exist constants $\kappa>0$ and $\tau_*>-\infty$ with the following significance. If $\mathcal{M}$ is $\kappa$-quadratic at time $\tau_0 \le \tau_*$ then for all times $t\le -e^{-\tau_0}$
 in $\{ V \leq L\sqrt{|t|/\log|t|}\}$ and in $\{ V\geq \sqrt{2|t|}-L\sqrt{|t|}/\log|t| \}$ we have $\nabla^2 Q(X,X)\leq \gamma \, g(X,X)$
 for all $X\perp\partial_\vartheta$.
\end{proposition}

\begin{proof}
Unlike $V(p,t)$, note that the quadratic profile $Q(p,t)$ is  a smooth function on the whole manifold unless it is the singular time of the flow.
Recall from \cite[Lemma 5.4]{ADS2} that the profile function of the 2d-bowl is quadratically concave. Since $\nabla^2 Q$ is a scale invariant quantity, applying Theorem \ref{strong_uniform0} (uniform sharp asymptotics) we thus infer that
\be
\nabla^2 Q(X,X)\leq \lambda \, |X|^2
\ee
in $\{ V \leq L\sqrt{|t|/\log|t|}\}$, where $\lambda>0$ can be made arbitrarily small by choosing $\kappa>0$ small enough and $\tau_\ast<-\infty$ negative enough. On the other hand, inserting $V \leq L\sqrt{|t|/\log|t|}$ in equation \eqref{eq-Q} we see that
\begin{equation}
\gamma \, g(X,X) \geq L^{-3} |X|^2.
\end{equation}
Suppose now $V\geq\sqrt{2|t|}-L\sqrt{|t|}/\log|t|$. Then by Theorem \ref{strong_uniform0} (uniform sharp asymptotics) and convexity we have\footnote{One has to apply this twice, specifically first considering the tangent plane at any point with say renormalized radial coordinate $y=100$ to show that $\{V\geq\sqrt{2|t|}-L\sqrt{|t|}/\log|t|\}$ is contained in the parabolic region, and then again to get the gradient bound.} 
\begin{equation}
|\nabla V|\leq \frac{C}{\log|t|},
\end{equation}
hence
\be
\nabla^2 Q(X,X)\leq \frac{C}{(\log |t|)^2}|X|^2.
\ee
On the other hand, inserting the rough bound $V\geq \sqrt{|t|}$ in the definition of $\gamma$ we get
\begin{equation}
\gamma \, g(X,X)\geq \frac{1}{(\log|t|)^{3/2}}|X|^2.
\end{equation}
This implies the assertion.
\end{proof}

 We are now ready to present the proof of Theorem \ref{prop-concavity} (quadratic almost concavity), which we restate here in terms of the unrescaled variables:
 
 \begin{theorem}[quadratic almost concavity]
      \label{prop-concavity_restated}
There exist constants $\kappa>0$ and $\tau_*>-\infty$ with the following significance. If  $\mathcal{M}$ is $\kappa$-quadratic at time $\tau_0 \le \tau_*$, then for all $t\le -e^{-\tau_0}$ we have
\begin{equation}
\nabla^2Q(X,X) \le \gamma \,  g(X,X)
\end{equation}
for all $X\perp \partial_\vartheta$, where $\gamma$ denotes the function defined in \eqref{eq-Q}.
\end{theorem}
 
\begin{proof}Let $\kappa>0$, $\tau_*>-\infty$, and $L<\infty$ be the constants from Proposition \ref{lemma-full-null-cond} (reaction term), where we choose $\zeta$ sufficiently small, say $\zeta=1/100$.
Adjusting $\kappa$ and $\tau_\ast$ we can arrange that the conclusion of Proposition \ref{lemma-soliton-region} (sign in excluded region) holds as well, and also that for all $t\leq - e^{-\tau_\ast}$ we have $\gamma\leq 1/100$.
Suppose now that  $\mathcal{M}$ is $\kappa$-quadratic at time $\tau_0 \le \tau_*$. Given any $\delta\in (0,1/100)$, we work with the tensor
\be
A =\nabla^2Q -(\gamma+\delta) g.
\ee
By Proposition \ref{lemma-soliton-region} (sign in excluded region) for all $t\leq -e^{-\tau_0}$ in the region 
\be
E_t:=\big\{ V \leq L\sqrt{|t|/\log|t|}\big\} \cup \big\{ V\geq \sqrt{2|t|}-L\sqrt{|t|}/\log|t| \big\}
\ee
we have
\be
\nabla^2 Q(X,X)\leq \gamma g(X,X) \quad \textrm{for all}\,\,  X\perp \partial_\vartheta.
\ee
Moreover, by the derivative estimate from Corollary \ref{lemma-cylindrical} (cylindrical estimate),\footnote{One can apply this for some $\kappa'\ll \kappa$, since the inwards quadratic bending improves when one goes further back in time (as we have seen in the proof of Theorem \ref{point_strong}).} there exists some $T_\delta < -e^{-\tau_0}$ such that for all $t\leq T_\delta$ the estimate
\be\label{a_leq_ineq}
A(X,X)<0  \quad \textrm{for all}\,\, 0\neq X\perp \partial_\vartheta
\ee
holds at all points (here, by convention we set $A(X,X):=-\infty$ at points with $V=0$). 

\smallskip 

Suppose towards a contradiction that there is some time $t\in (T_\delta, -e^{-\tau_0}]$ such that \eqref{a_leq_ineq}  fails, and let $\bar{t}$ be the first such time. Let $\bar{p}\in S^3$ be a point where this happens. By the above, we have $\bar{p}\notin E_{\bar{t}}$. We now choose a null eigenvector $X\in T_{\bar{p}}S^3\cap \partial_\vartheta^\perp$ and extend it to a vector field around $\bar{p}$ as follows:

\begin{claim}[extension]\label{claim_extension}  There exists an extension of $X$ to a vector field $X(p)$ in an open neighborhood of $\bar p$, say $U$, with the following properties: 

\begin{enumerate}[(i)]
\item $X\perp \partial_\vartheta$ in $U$,	
\item $\nabla _{\partial _\vartheta} X =(\tfrac12 Q^{-1} \nabla_X Q)  \, \partial_{\vartheta} $ in $U$,
\item$\nabla _{\dot \gamma(t)} X = 0$ for any geodesics $\gamma$ in $U$ with $\gamma(0)=\bar p$ and $\dot \gamma(0)  \perp \partial _\vartheta  $.   
\end{enumerate}\end{claim}

\begin{proof}[Proof of the claim]Working in $(x_1,x_2,\vartheta)$-coordinates we can construct a vector field of the form
\be
X=X^1(x_1,x_2)\partial_{x_1}+X^2(x_1,x_2)\partial_{x_2},
\ee
by first parallel transporting in the $(x_1,x_2)$-plane along radial geodesics emanating from the point under consideration, and then declaring that in this local formula for $X$ is independent of $\vartheta$. Then, in a neighborhood of $\bar p$, the properties (i) and (iii) hold by construction, and moreover using \eqref{Christoffel} and $\Gamma_{3j}^3 =V_j /V$ for $j\in \{1,2\}$  we get (ii). This proves the claim.
\end{proof}

Continuing the proof of the theorem, we consider the function
\begin{equation}
f(p,t):= A_{(p,t)}( X(p), X(p)).
\end{equation}
Then, by the second derivative test from calculus at $(\bar{p},\bar{t})$ we have
\be\label{2nd_der_test}
\partial_t f \geq 0,\quad \nabla f = 0,\quad \Delta f \leq 0.
\ee
Moreover, recall that by Corollary \ref{lemma-eq-A} (evolution of $A$-tensor) we have
\begin{equation}\label{ev_a_tens}
(\partial_t - \Delta+\nabla_{Z}) A = N,
\end{equation}
where $Z$ is defined by $\langle Z,Y\rangle=Q^{-1}\nabla_Y Q+2\eta h(\nabla V, Y)$ for all $Y$.
Using also the null eigenvector assumption, at the point $(\bar{p},\bar{t})$ we thus get
\begin{multline}
(\partial_t - \Delta-\nabla_{Z})f\\
=N(X,X)-2\sum_k A(\nabla_{e_k} X, \nabla_{e_k} X) - 4\sum_k(\nabla_{e_k} A)( \nabla_{e_k} X, X),
\end{multline}
where $e_1,e_2,e_3$ is any orthonormal basis at $\bar p$.
Now, thanks to Proposition \ref{lemma-full-null-cond} (reaction term) we have the estimate
\begin{equation}
N(X,X) \leq - \tfrac 12(1-\zeta-\eps)Q^{-2}|\nabla_XQ|^2(Q^{-1}|\nabla Q|^2-2\eps) \, .
\end{equation}
On the other hand,  using \eqref{obs_grad} and \eqref{eq-nablasqu} we see that
\be
g(\partial_\vartheta,\partial_\vartheta)^{-1}A(\partial_\vartheta,\partial_\vartheta)=\tfrac12 (Q^{-1}|\nabla Q|^2-2\eps)\, .
\ee
Hence, applying Claim \ref{claim_extension} (extension) we infer that
\be
-2\sum_k A(\nabla_{e_k} X, \nabla_{e_k} X)=-\tfrac{1}{4}Q^{-2}|\nabla_X Q|^2(Q^{-1}|\nabla Q|^2-2\eps).
\ee
Moreover, working in a frame $\{\bar e_1,\bar e_2, \bar e_3\}$ that extends $\{e_1,e_2,e_3\}$ to a neighborhood of $\bar p$, the null eigenvector condition implies
\begin{multline}\label{eq-tensorial}
\sum_k\nabla_{\bar e_k} (A(\nabla_{\bar e_k} X,X))=\sum_k(\nabla_{\bar e_k} A)(\nabla_{\bar e_k} X,X)+ \sum_k A(\nabla_{\bar e_k} X,\nabla_{\bar e_k} X),
\end{multline}
and  we infer that the left-hand side of \eqref{eq-tensorial} evaluated at $\bar p$ is independent of the choice of $\{e_i\}$ and its extension as a frame. Hence, for any three curves $\{\gamma_k\}_{k=1,2,3}$ starting at $\bar p$, if $\{\dot \gamma_k(0)\}$ forms an orthonormal basis, then 
\be \sum_k\nabla_{\bar e_k} (A(\nabla_{\bar e_k} X,X)) = \sum_k \frac{d}{dt}\bigg \vert_{t=0}A_{\gamma_k(t)}(\nabla_{\dot \gamma_k(t)}X,X(\gamma_k(t))).\ee
Choosing $\gamma_1$ and $\gamma_2$ to be unit speed geodesics satisfying $\dot \gamma_i(0)\perp \partial_\vartheta$ and $\gamma_3$ to be the integral curve of $V^{-1} \partial_{\vartheta}$, by Claim \ref{claim_extension} (extension), remembering also the block-diagonal structure of $A$, we infer that  for each $k\in\{1,2,3\}$ the function $t \mapsto A_{\gamma_k(t)}(\nabla_{\dot \gamma_k(t)}X,X(\gamma_k(t)))$ is identically zero. This yields
\begin{align}
- 4\sum_k(\nabla_{e_k} A)( \nabla_{e_k} X, X)= \tfrac{1}{2}Q^{-2}|\nabla_X Q|^2(Q^{-1}|\nabla Q|^2-2\eps).
\end{align}
Finally, recall that thanks to Lemma \ref{lemma_null_eigen} (null eigenvector) we have
\be
|\nabla_X Q|^2(Q^{-1}|\nabla Q|^2-2\eps)>0.
\ee
Combining the above, we thus conclude that
\be
(\partial_t - \Delta-\nabla_{Z})f< 0
\ee
at $(\bar{p},\bar{t})$, which gives the desired contradiction with \eqref{2nd_der_test}. Since $\delta>0$ was arbitrary, this finishes the proof of the theorem.
\end{proof}
       
Finally, we can now prove the crucial Corollary \ref{prop-great} (almost Gaussian collar), which we restate here for convenience of the reader.

\begin{corollary}[almost Gaussian collar] 
For every $\varepsilon > 0$, there exist constants $\kappa>0$, $\tau_*>-\infty$, $L<\infty$, and $\theta >0$ with the following significance. If  $\mathcal{M}$ is $\kappa$-quadratic at time $\tau_0 \le \tau_*$, then for all $\tau \le \tau_0$ we have
\begin{equation}
\left|y (v^2)_y + 4\right| < \varepsilon \quad \mbox{in the collar region } \, \mathcal{K}= \bigl\{  L/\sqrt{|\tau|} \le v \le 2  \theta \bigr\}.
\end{equation}
 \end{corollary}  
 
\begin{proof}
To begin with, given any $\delta>0$, using Theorem \ref{strong_uniform0} (uniform sharp asymptotics) and convexity, we see that for $\tau\leq \tau_0$ in the region $\{v\leq 2\theta\}$, provided $\theta=\theta(\delta)$ is small enough, we have
\be\label{cor_bas1}
2|\tau|(1-\delta)\leq y^2 \leq 2|\tau|(1+\delta)
\ee
for any $\mathcal{M}$ that is $\kappa=\kappa(\delta)$-quadratic from time $\tau_0\leq \tau_\ast(\delta)$. Moreover, Theorem \ref{strong_uniform0} (uniform sharp asymptotics) and convexity also yield
\begin{equation}\label{cor_bas2}
(vv_y)^2|_{v=2\theta} \geq \frac{2}{|\tau|}(1-\delta),
\end{equation}
and
\begin{equation}\label{cor_bas3}
(vv_y)^2|_{v=L/\sqrt{|\tau|}}\leq \frac{2}{|\tau|}(1+CL^{-1}),
\end{equation}
for $L$ large enough, possibly after decreasing $\kappa$ and ${\tau}_{\ast}$ (see e.g. \cite[Proof of Lemma 5.7]{ADS2} for a similar computation, with more details).

Now, applying Theorem \ref{prop-concavity} (quadratic almost concavity) in direction of the radial vector $X=\partial_y$, we get
\be
2\eta^{-2} vv_{yy}+2v_y^2\leq |\tau|^{-3/2}v^{-3}(1+v_y^2),
\ee
where
\be
\eta^2 = 1+v_y^2+y^{-2}v_\varphi^2\, .
\ee
Together with Corollary \ref{lemma-cylindrical} (cylindrical estimate) this implies
\be
(v v_y)_y \le  2  |\tau|^{-3/2}v^{-3}\, .
\ee
for all $\tau \le \tau_0$ in the region $\{v\geq L/\sqrt{|\tau|}\}$, provided $L$ is sufficiently large, $\kappa$ is sufficiently small, and $\tau_\ast$ is sufficiently negative.
For any fixed $\varphi$ and $\tau$, considering $v_y$ as a function of $v$, we can rewrite our inequality as
   \be   \label{eq-diff-ineq} \frac{d}{dv} (v^2 v_y^2)  \le  4  |\tau|^{-3/2}v^{-2}\, .
   \ee 
Integrating this from $v$ to $2\theta$ and from $L/\sqrt{\tau}$ to $v$ yields
\be
(vv_y)^2|_{v=2\theta} -4 L^{-1} |\tau|^{-1} \leq v^2 v_y^2 \leq (vv_y)^2|_{v=L/\sqrt{|\tau|}} + 4 L^{-1} |\tau|^{-1}.
\ee
Multiplying this by $y^2$, and using \eqref{cor_bas1}, \eqref{cor_bas2} and \eqref{cor_bas3}, the assertion follows.
\end{proof}
 
\bigskip


\section{Spectral uniqueness}\label{spectral_uniqueness}

In this section, we prove Theorem \ref{thm-spectral_uniqueness} (spectral uniqueness), which we restate here for convenience of the reader:

\begin{theorem}[spectral uniqueness]
There exists $\kappa>0$ and $\tau_{*}>-\infty$ with the following significance. If $\mathcal{M}^{1}=\{M^1_t\}$ and $\mathcal{M}^{2}=\{M^2_t\}$ are bubble-sheet ovals in $\mathbb{R}^{4}$ that are $\kappa$-quadratic at  time $\tau_0\leq \tau_{*}$, and if their truncated renormalized profile functions $v^{1}_{\cC}$ and $v^{2}_{\cC}$ satisfy the spectral condition
\begin{equation}\label{cent_0}
      \mathfrak{p}_{0}v^{1}_{\cC}(\tau_{0})=\mathfrak{p}_{0}v^{2}_{\cC}(\tau_{0}),
    \end{equation}
then 
\begin{equation}
\label{eq-unique}
    \mathcal{M}^{1}=\mathcal{M}^{2}.
\end{equation}
\end{theorem}

Here, denoting by $v_i$ the profile function of the renormalized flow $e^{\frac{\tau}{2}} M^i_{ - e^{-\tau}}$, the truncated renormalized profile function is defined by
\be
v_\cC^i=\chi_\cC(v_i)v_i,
\ee
where $\chi_\mathcal{C} : [0, \infty) \to [0, 1]$ is a fixed smooth function satisfying
\be
  (i)\,\, \chi_\mathcal{C} \equiv 0 \mbox{ on } \big[0,\tfrac58 \theta\big]   \qquad \mbox{and} \qquad  (ii) \,\, \chi_\mathcal{C} \equiv 1 \mbox{ on } \big[\tfrac78\theta, \infty\big). 
\ee
We also recall that the evolution of $v_\cC$ is governed by the 2d Ornstein-Uhlenbeck operator, which in Euclidean coordinates is given by the formula
\be\label{2dOU}
\mathcal{L}=\partial_{y_1}^2+\partial_{y_2}^2-\tfrac{1}{2}(y_1\partial_{y_1}+y_2\partial_{y_2})+1,
\ee
and which is a self-adjoint operator on the Gaussian $L^2$-space
\begin{equation}
\mathcal{H}=L^2\big(\mathbb{R}^2,e^{-|{\bf y}|^2/4} \,  d{\bf y}\big)= \mathcal{H}_+\oplus \mathcal{H}_0\oplus \mathcal{H}_-.
\end{equation}
Here, the unstable and neutral eigenspace are explicitly given by
\begin{equation}
    \mathcal{H}_+=\textrm{span}\big\{1, y_1, y_2\big\}, \qquad
\mathcal{H}_0=\textrm{span}\big\{y_1^2-2, y_2^2-2, y_1y_2\big\},
\end{equation}
and as before we denote the orthogonal projections by $\mathfrak{p}_{\pm}$ and $\mathfrak{p}_0$. In particular, note that by Definition \ref{k_tau00} ($\kappa$-quadratic) thanks to the $\kappa$-quadraticity assumption in addition to the hypothesis \eqref{cent_0} we also have
\be\label{p_plus_van}
\mathfrak{p}_{+}v_{\cC}^1(\tau_{0})=\mathfrak{p}_{+}v_{\cC}^2(\tau_{0}).
\ee
As before, given any renormalized profile function $v$ (e.g. $v=v_1$) we consider the associated cylindrical region and tip region
\begin{equation}
\label{eq-cyl-region_rest}
\cC= \big\{ v \ge \theta/2  \big\}\qquad\mathrm{ and }\qquad \mathcal{T}= \big\{v \le 2   \theta \big\},
\end{equation}
where the latter can be subdivided into the collar region and soliton region:
\begin{equation}
\collar = \bigl\{  L/\sqrt{|\tau|} \le v \le 2  \theta \bigr\} \qquad\mathrm{ and }\qquad 
\mathcal{S} = \bigl\{v \le L/\sqrt{|\tau|} \bigr\}.
\end{equation}
Note that $\chi_\cC$ indeed localizes in the cylindrical region, more precisely\footnote{The wiggle room between $\tfrac12\theta$ and $\tfrac58 \theta$ will be used for estimating $v_1\chi_\cC(v_1)-v_2\chi_\cC(v_2)$.}
\be\label{eq_wiggle_room}
\textrm{spt}(\chi_\cC)\subseteq \{ v\geq \tfrac58\theta\} \subset\cC.
\ee
To localize in the tip region, we fix a smooth function $\chi_\mathcal{T} : [0, \infty) \to [0, 1]$ satisfying
\be
  (i)\,\, \chi_\mathcal{T} \equiv 1 \mbox{ on } [0,\theta]   \qquad \mbox{and} \qquad  (ii) \,\, \chi_\mathcal{T} \equiv 0 \mbox{ on } [2\theta, \infty). 
\ee
In the tip region we work with the inverse profile function $Y$ defined by
\begin{equation}\label{def_inv_sec4}
Y(v(y,\varphi,\tau),\varphi,\tau)=y,
\end{equation}
and its zoomed in version $Z$ defined by
\begin{equation}\label{zoomed_in_5}
Z(\rho,\varphi,\tau)= |\tau|^{1/2}\left(Y(|\tau|^{-1/2}\rho,\varphi,\tau)-Y(0,\varphi,\tau)\right).
\end{equation}

Throughout this section we use the convention that $\theta>0$ is a fixed small constant and $L<\infty$ is a fixed large constant. During the proofs one is allowed to decrease $\theta$ and increase $L$ at finitely many instances, as needed or convenient.

\subsection{Energy estimate in the cylindrical region}\label{energy-cylindrical}
In this subsection, we prove an energy estimate in the cylindrical region, by generalizing \cite[Section 6]{ADS2} and \cite[Section 5.4]{CHH_translator} to the bubble-sheet setting.

Recall that for any bubble-sheet oval in $\mathbb{R}^4$ the renormalized profile function $v$, viewed as function depending on $y_1,y_2$ and $\tau$, satisfies 
 \be\label{eqn-v}
 v_\tau= \left(\delta_{ij}-\tfrac{v_{y_i}v_{y_j}}{1+|Dv|^2}\right)\, v_{y_iy_j}-\tfrac{1}{2}  \, y_i v_{y_i} + \frac{v}2 -\frac{1}{v}. 
\ee
Now, given $\mathcal{M}^{1}$ and $\mathcal{M}^2$ we consider the difference of their renormalized profile functions
\be
w=v_1-v_2\, .
\ee
In the following, we abbreviate $D=(\partial_{y_i})$, $D^2 = (\partial^2_{y_iy_j})$ and $A\! :\! B=a_{ij}b_{ij}$.

\begin{lemma}[evolution of $w$]\label{evol_w} The function $w$ evolves by
\be (\partial_\tau - \mathcal{L})w = \cE[w],\ee
where $\mathcal{L}$ is the 2d Ornstein-Uhlenbeck operator from \eqref{2dOU}, and 
\begin{multline}\label{eqn-E15}
 \cE[w] =  - \frac{Dv_1\otimes Dv_1\! :\! D^2 w}{1+|Dv_1|^2 } - \frac{D^2v_2\! :\! D(v_1+v_2)\!\otimes\! Dw}{(1+|Dv_1|^2)} \\
+\frac{D^2v_2\! :\! Dv_2\otimes Dv_2\, D(v_1+v_2)\!\cdot\! Dw}{(1+|Dv_1|^2)(1+|Dv_2|^2)} + \frac{2-v_1v_2}{2v_1v_2}w,
 \end{multline}
\end{lemma}

\begin{proof}
Subtracting the evolution equations of $v_1$ and $v_2$ yields
\bea w_\tau 
 &= \mathcal{L} w - \frac{Dv_1\otimes Dv_1\!:\! D^2v_1 }{1+|Dv_1|^2} + \frac{Dv_2\otimes Dv_2\!:\! D^2v_2 }{1+|Dv_2|^2} +\frac{2-v_1v_2}{2v_1v_2} w  .\eea 
Now, using the product rule for differences we compute
\bea   & \!\!\!\!\!\!\!\!\!\!\!\!\!\!\!  \frac{Dv_{1}\otimes  Dv_{1}\! :\!  D^2 v_{1} }{1+|Dv_1|^2} -\frac{Dv_{2}\otimes  Dv_{2}\! :\!  D^2 v_{2} }{1+|Dv_2|^2} \\  & =   \frac{Dv_{1}\otimes  Dv_{1}\! :\!  D^2 w }{1+|Dv_1|^2}  + \frac{Dv_{1}\otimes  Dw\! :\!  D^2 v_2 }{1+|Dv_1|^2} + \frac{Dw\otimes  Dv_2\! :\!  D^2 v_2 }{1+|Dv_1|^2} \\ &\quad   + \left( \frac{1}{1+|Dv_1|^2}- \frac{1}{1+|Dv_2|^2}\right) {Dv_2\otimes  Dv_2\! :\!  D^2 v_2 } .\eea 
Moreover, we have
\bea Dv_{1}\otimes  Dw\! :\!  D^2 v_2 +Dw\otimes  Dv_2\! :\!  D^2 v_2 = Dw\otimes D ( v_1+v_2)\! :\!  D^2 v_2,\eea
and
\bea \frac{1}{1+|Dv_1|^2}- \frac{1}{1+|Dv_2|^2}= - \frac{ Dw \cdot D(v_1+v_2)}{(1+|Dv_1|^2)(1+|Dv_2|^2)}.\eea 
This implies the assertion.
\end{proof}

Next, we consider the truncated difference
\be
w_\cC=v_1\chi_\cC(v_1)-v_2\chi_\cC(v_2)\, .
\ee
To simplify the notation in the computations that follow,  for any scalar function $\chi$ (e.g. for $\chi_\cC$, $\chi_\cC'$ or $\chi_\cC''$) we denote by $w^\chi$  the function 
\begin{equation}
\label{eq-diff-cut}
w^\chi (\bry, \tau) = \chi  (v_1(\bry,\tau)) -\chi (v_2(\bry,\tau)). 
\end{equation}
In particular, note that $w^{\textrm{id}}=w$ and $\left|w^\chi\right| \leq  |w| \sup|\chi'|$.

\begin{lemma}[{evolution of $w_\cC$}]\label{lemma-evolution-wC} The function $w_\cC$ satisfies
\begin{align}
(\partial_\tau -\mathcal{L})w_\cC =&\cE[w_\cC] + \bar{\cE}[w,\chi_\cC(v_1)]-\mathcal{E}[v_2w^{\chi_\cC}]- 2 D  v_2 \cdot D w^{\chi_\cC}\nonumber \\
&+ v_2 (\partial_\tau -\mathcal{L}) w^{\chi_\cC} + w^{\chi_\cC} (\partial_\tau -\mathcal{L}+1)v_2,
 \end{align}
where
\begin{equation}
\label{eq-line-E}
\bar{\cE}[w,\chi_\cC(v_1)]= (\partial_\tau -\mathcal{L})(w\chi_\cC(v_1))- \cE[ w \chi_\cC(v_1)].
\end{equation}
\end{lemma}

\begin{proof} By the product rule we have
\begin{multline}
(\partial_\tau -\mathcal{L})(v_2 w^{\chi_\cC})= v_2 \, \big (\partial_\tau -\mathcal{L} \big )w^{\chi_\cC}
+ w^{\chi_\cC}  (\partial_\tau -\mathcal{L}+1)v_2 \\ - 2 D  v_2 \cdot D w^{\chi_\cC} \,  .
\end{multline}
Observing also that
\be
w_\cC=w\chi_\cC(v_1)+v_2 w^{\chi_\cC},
\ee
this  implies the assertion.
\end{proof}

The goal of this subsection is to prove the following energy estimate:

\begin{proposition}[energy estimate in cylindrical region]\label{prop-cyl-est} For every $\eps>0$, there exists $\kappa>0$ and $\tau_*>-\infty$ with the following significance. If $\mathcal{M}^1$ and $\mathcal{M}^2$  are $\kappa$-quadratic at time $\tau_0 \le \tau_*$, then
\begin{equation}
\label{eqn-cylindrical1}
\Vert w_\cC - \mathfrak{p}_0w_\cC \Vert_{\hD ,\infty } \le \eps  \Big( \Vert  w_\cC\Vert_{\hD ,\infty} + \Vert w 1_{\{ \theta/2\leq v_1\leq \theta\}} \Vert_{\mathcal{H},\infty}  \Big).
\end{equation}
\end{proposition}

Let us review the relevant norms and spaces for this energy estimate.
In addition to the Gaussian $L^2$-space $\mathcal{H}$, which is equipped with the norm
\begin{equation}\label{eqn-normp0}
  \|f\|_{\mathcal{H}} = \left(\int_{\R^2}  f(\bry)^2 e^{-|\bry|^2/4}\, d\bry\right)^{1/2},
\end{equation}
we also need the Gaussian $H^1$-space $\hD:= \{f\in\mathcal{H}  : Df  \in \mathcal{H}\}$ with the norm
\be
  \|f\|_\hD = \left(\int_{\R^2} \big( f(\bry)^2 +|Df(\bry)|^2 \big) \, e^{-|\bry|^2/4}d\bry\right)^{1/2},
\ee
and its dual space $\hD^\ast$ equipped with the dual norm
\be
\|f\|_{\hD^\ast}=\sup_{ \|g\|_\hD\leq 1 }\langle f,g\rangle,
\ee
where $\langle \,\,\, , \,\, \rangle: \hD^\ast \times \hD:\to\mathbb{R}$ denotes the canonical pairing.

For time-dependent functions the above induces the parabolic norms
\be
  \|f\|_{\mathcal{X} ,\infty} = \sup_{\tau \le \tau_0} \left(\int_{\tau -
    1}^{\tau} \|f(\cdot,\sigma)\|_{\mathcal{X} }^2\, d\sigma\right)^{\frac12}, 
\ee
where $\mathcal{X}=\mathcal{H},\hD$ or $\hD^\ast$.

Let us also recall a few basic facts that will be used frequently in the following proof. To begin with, by the weighted Poincar\'e inequality \eqref{easy_Poincare} multiplication by $1+|\bry|$ is a bounded operator from $\hD$ to $\mathcal{H}$, and hence by duality from $\mathcal{H}$ to $\hD^\ast$ as well, namely
  \be
    \|(1+|\bry|)f\|_\mathcal{H}  \leq C \|f\|_\hD\quad\textrm{and}\quad
    \|(1+|\bry|)f\|_{\hD^\ast}  \leq C \|f\|_\mathcal{H}.
  \ee
Consequently, $\pd_{y_i}$ and $\pd_{y_i}^*= -\pd_{y_i}+ \frac{1}{2} y_i$ are bounded operators from  $\hD$ to $\mathcal{H}$, and hence by duality from $\mathcal{H} $ to $\hD^*$ as well. In particular, this implies that the Ornstein-Uhlenbeck operator $\mathcal{L}:\hD\to \hD^\ast$ is well-defined.
Finally, for estimating the $\hD^\ast$-norm it is useful to observe that if $g\in \hD$ and $h\in W^{1,\infty}$, then by the product rule we have $\| hg\|_{\hD}\leq C \| h \|_{W^{1,\infty}}\| g\|_{\hD}$, hence by duality
\begin{equation}\label{eq_product_rule_norm}
\| h f \|_{\hD^\ast} \leq C \| h\|_{W^{1,\infty}} \| f \|_{\hD^\ast}\, .
\end{equation}
We are now ready to prove the energy estimate in the cylindrical region.

\begin{proof}[{Proof of Proposition \ref{prop-cyl-est}}]
To begin with, thanks to \eqref{p_plus_van} and \cite[Lemma 6.7]{ADS2}, we have the general estimate
\begin{equation}
\label{eq-gen-energy}
\|w_\cC - \mathfrak{p}_0 w_\cC \, \|_{\hD,\infty} \le C\|(\partial_{\tau} - \mathcal{L}) w_\cC\|_{\hD^*,\infty}.
\end{equation}
To estimate the expression on the right hand side, we rewrite the conclusion of Lemma \ref{lemma-evolution-wC} (evolution of $w_\cC$) in the form
\begin{equation}
\label{eq-ev-wC}
(\partial_\tau -\mathcal{L})\, w_\cC = \cE[w_\cC] + \bar{\cE}[w,\chi_\cC(v_1)]+J+K,
\end{equation}
where
\be\label{eqn-IJK}
\begin{split}
&J \, =  (v_{2,\tau} - v_{2,y_iy_i}  +\tfrac{1}{2}  y_i v_{2,y_i}  -\cE[v_2] ) w^{\chi_\cC} -  2 D v_2 \cdot Dw^{\chi_\cC},\\
&K\, = \cE[v_2] w^{\chi_\cC} - \cE[v_2w^{\chi_\cC}] + v_2\, (\partial_\tau -\mathcal{L})w^{\chi_\cC}.\end{split}
\ee
Following the same arguments as in \cite[Proof of Lemma 5.13]{CHH_translator}, which in turn is similar to \cite[Proof of Lemma 6.8 and Lemma 6.9]{ADS2}, but now using the bubble-sheet evolution equation from Lemma \ref{evol_w} (evolution of $w$) and the bubble-sheet derivative estimates from Lemma \ref{lemma-cylindrical} (cylindrical estimate), one  can easily show that given any $\eps > 0$, there exist $\kappa > 0$ and $\tau_* > -\infty$, such that assuming $\kappa$-quadraticity at time $\tau_0 \leq  \tau_*$, for all $\tau\leq\tau_0$ we have
\begin{equation}\label{eq-I1}
\big\| \cE[w_\cC(\tau)]   \big\|_{\hD^\ast}   \leq  \eps  \|w_\cC  (\tau) \|_{\hD} ,
\end{equation}
and
\begin{equation}\label{eq-I2}
\big\| \bar{\cE}[w(\tau),\chi_\cC(v_1(\tau))]  \big\|_{\hD^\ast} \leq  \eps \| w(\tau) 1_{D_\tau} \|_{\mathcal{H}} ,
\end{equation}
where
\be
D_\tau:=\Big\{ y: \tfrac58 \theta \leq v_1(y,\tau)\leq \tfrac78 \theta\,\,\,\textrm{or}\,\,\,\tfrac58 \theta \leq v_2(y,\tau)\leq \tfrac78 \theta\Big\}.
\ee
Also, following the same arguments as in \cite[Proof of Lemma 5.14]{CHH_translator}, but now using the bubble-sheet asymptotics from Theorem \ref{strong_uniform0} (uniform sharp asymptotics), we easily get the estimate
\begin{equation}
\label{eq-J}
\| J(\tau) \|_{\hD^*} \leq \varepsilon \| w(\tau)  \, 1_{D_\tau } \|_{\mathcal{H}}.
\end{equation}
Hence, our only new task that necessitates somewhat nontrivial modifications is to estimate the $\hD^*$-norm of $K(\tau)$. To this end, note that
\begin{align}
\cE[v_2] w^{\chi_\cC} - \cE[v_2 w^{\chi_\cC}] =&\frac{v_2\, Dv_1\otimes Dv_1\! :\! D^2w^{\chi_\cC}+2 Dv_1\!\cdot\! Dv_2\, Dv_1\!\cdot\! Dw^{\chi_\cC}}{1+|Dv_1|^2}\nonumber\\
&  +\frac{v_2\, D^2v_2\! :\! D(v_1+v_2)\otimes Dw^{\chi_\cC}}{1+|Dv_1|^2}\nonumber\\
&  -\frac{v_2\,D^2v_2\!:\! Dv_2\otimes Dv_2\, D(v_1+v_2)\!\cdot\! Dw^{\chi_\cC} }{(1+|Dv_1|^2)(1+|Dv_2|^2)}\, .
\end{align}
On the other hand, differentiating the defining identity \eqref{eq-diff-cut} we get
\be
Dw^{\chi_\cC} = \chi_\cC'(v_1) Dw + w^{\chi_\cC'} Dv_2\,  ,
\ee
and
\be
\begin{split}
D^2w^{\chi_\cC}= \chi_\cC'(v_1) D^2 w &+ \chi_\cC''(v_1) (Dv_1\otimes Dw +Dw \otimes Dv_2) \\&+  w^{\chi_\cC'}  D^2 v_2 + w^{\chi_\cC''} Dv_2 \otimes  Dv_2\,  .
\end{split}
\ee
Hence, we can rewrite the above as
\begin{equation}
\begin{split}
\label{eq-diff-E}
\cE[v_2] w^{\chi_\cC} &- \cE[v_2 w^{\chi_\cC}]+ \chi_\cC'(v_1)v_2 \cE[w]    \\ &\qquad\qquad = a\!\cdot\!Dw + b\, w + c\, w^{\chi_\cC'} + d \, w^{\chi_\cC''},
\end{split}
\end{equation}
where
\be\begin{split}
a& = \frac{v_2\chi_\cC''(v_1) Dv_1\!\cdot\! D(v_1+v_2)  + 2\chi_\cC(v_1) Dv_1\!\cdot\! Dv_2 }{1+|Dv_1|^2}\, Dv_1,  \\
b &= \frac{2-v_1v_2}{2v_1v_2}\, \chi_\cC'(v_1) v_2,\\
c &= \frac{v_2D^2v_2\!:\!  ( Dv_1\otimes Dv_1+D(v_1+v_2)\otimes Dv_2)  +2(Dv_1\!\cdot\! D v_2)^2}{1+|Dv_1|^2} \\
&\qquad - \frac{v_2  Dv_2 \!\cdot\! D(v_1+v_2) \,  D^2v_2\!:\! Dv_2\otimes Dv_2  }{(1+|Dv_1|^2)(1+|Dv_2|^2)},\\
d &= \frac{v_2 (Dv_1\!\cdot\! D v_2)^2}{1+|Dv_1|^2}.
\end{split}
\ee
Now, using the basic facts reviewed above we can estimate
\be
\| a\cdot Dw \|_{\hD^\ast}=\| a\cdot D(w1_{D_\tau}) \|_{\hD^\ast}\leq C \| a\|_{W^{1,\infty}} \| w1_{D_\tau}\|_{\mathcal{H}},
\ee
and
\be
\| b w \|_{\hD^\ast}\leq C \left\| \frac{1}{1+|{\bf y}|} bw  \right\|_{\mathcal{H}} \leq \frac{C}{|\tau|^{1/2}} \| w 1_{D_\tau} \|_{\cH},
\ee
where in the last step we also used that $|{\bf y}|\geq |\tau|^{1/2}$ on the support of $1_{D_\tau}$ thanks to Theorem \ref{strong_uniform0} (uniform sharp asymptotics). Similarly, we see that
\be
\| c w^{\chi_\cC'} \|_{\hD^\ast} \leq C \| c\|_{W^{1,\infty}}\| w 1_{D_\tau} \|_{\cH},
\ee
and
\be
\| d w^{\chi_\cC''}\|_{\hD^\ast} \leq C \| d\|_{W^{1,\infty}}\| w 1_{D_\tau} \|_{\cH},
\ee
where we also used that
\be
\left| w^{\chi_\cC^{(k)}} \right|=\left| \int_{v_1}^{v_2} \chi_\cC^{(k+1)}(v) \, dv\right| \, 1_{D_\tau} \leq C |w| \, 1_{D_\tau}.
\ee
Furthermore, thanks to Corollary \ref{lemma-cylindrical} (cylindrical estimate) we can make the $W^{1,\infty}$-norm of $a,c$ and $d$ arbitrarily small by choosing $\kappa>0$ small enough and $\tau_\ast>-\infty$ negative enough. Combining the above we thus obtain
\be
\left\| \cE[v_2] w^{\chi_\cC} - \cE[v_2 w^{\chi_\cC}]+\chi_\cC'(v_1)v_2  \cE[w] \right\|_{\hD^\ast}\leq \eps \left\| w 1_{D_\tau} \right\|_{\cH}.
\ee
To capture the remaining terms, we compute
\begin{align}
v_2(\partial_\tau-\mathcal{L})& w^{\chi_\cC}- \chi_\cC'(v_1)v_2 \cE[w] =  \chi_\cC'(v_1)v_2 w-\chi_\cC''(v_1)v_2D(v_1+v_2)\!\cdot\! Dw\nonumber\\
&-v_2w^{\chi_\cC} - v_2 |Dv_2|^2 w^{\chi_\cC''} + v_2 \big( v_{2,\tau} - v_{2,y_iy_i}  +\tfrac{1}{2}  y_i v_{2,y_i}  \big)w^{\chi_\cC'} .
\end{align}
Arguing similarly as above, we see that
\begin{align}
&\| \chi_\cC'(v_1)v_2 w
-\chi_\cC''(v_1)v_2D(v_1+v_2)\!\cdot\! Dw
-v_2 w^{\chi_\cC}- v_2 |Dv_2|^2 w^{\chi_\cC''}\|_{\hD^\ast}\nonumber \\
&\qquad +\| v_2 \big( v_{2,\tau} - v_{2,y_iy_i}  +\tfrac{1}{2}  y_i v_{2,y_i}  \big) w^{\chi_\cC'}\|_{\hD^\ast} \leq \eps \left\| w 1_{D_\tau} \right\|_{\cH},
\end{align}
Summing up, we have thus shown that
\begin{equation}
\label{eq-K}
\| K(\tau) \|_{\hD^*} \leq \varepsilon \| w(\tau)  \, 1_{D_\tau } \|_{\mathcal{H}}.
\end{equation}
Finally, thanks to  Theorem \ref{strong_uniform0} (uniform sharp asymptotics) for $\kappa$ small enough and $\tau_\ast$ negative enough we have
\be
v_2({\bf y},\tau)\geq \tfrac{5}{8}\theta\, \Rightarrow v_1({\bf y},\tau)\geq \tfrac{1}{2}\theta, \quad  v_2({\bf y},\tau)\leq \tfrac{7}{8}\theta\, \Rightarrow v_1({\bf y},\tau)\leq\theta,
\ee
hence
\be
1_{D_\tau} \leq 1_{ \{ \theta/2 \leq v_1(\cdot,\tau)\leq \theta\} }.
\ee
This concludes the proof of the proposition.
\end{proof}


\medskip

\subsection{Derivative and weight estimates in the tip region}\label{sec-apriori-tip}
In this subsection, we prove derivative estimates for the inverse profile function and weight function in the tip region, by generalizing some arguments from \cite[Section 7]{ADS2} and \cite[Section 5.5]{CHH_translator} to the bubble-sheet setting.

Throughout, we denote by $Z_B=Z_B(\rho)$ the profile function of the 2d-bowl with speed $1/\sqrt{2}$, namely the unique solution of
\be
\frac{Z_{B,\rho\rho}}{1+Z_{B,\rho}^2}+\frac{1}{\rho} Z_{B,\rho} + \frac{1}{\sqrt{2}}=0,\qquad Z_B(0)=Z_{B,\rho}(0)=0.\footnote{Note that in our sign convention $Z_B\leq 0$, which is consistent with \eqref{zoomed_in_5}.}
\ee
We recall the well-known fact that one has the asymptotic expansions
\begin{equation}
  \label{eq-ZB-asymptotics}
  Z_B(\rho) =
  \begin{cases}
    -\sqrt{2}\rho^2/4 + O(\log \rho) & \rho\to\infty \\[2pt]
    -\sqrt{2}\rho^2/8 +O(\rho^4) & \rho\to0,
  \end{cases}
\end{equation}
and that these expansions may be differentiated, see e.g. \cite{AV_degenerate_neckpinch}.

\begin{proposition}[first tip derivatives]\label{lemma-mainY}
For every $\eta>0$, there exist $\theta > 0$, $\kappa > 0$  and $\tau_* > -\infty$,
such that if $\mathcal{M}$ is $\kappa$-quadratic at time $\tau_0 \le \tau_*$, then
\be
\frac{1}{4}|\tau|^{1/2}\leq \left|\frac{Y_v}{v}\right|  \leq |\tau|^{1/2},\qquad |Y_\varphi|\leq \eta |\tau|^{1/2}\quad \textrm{and} \quad |Y_\tau|\leq \eta \left|\frac{Y_v}{v}\right|
\ee
holds for all $v \le 2\theta$ and $\tau\le \tau_0$.  
\end{proposition} 

\begin{proof} First, the estimate for $Y_\vp$ has already been established in \eqref{eqn-Yphi0}. 

Next, thanks to Theorem \ref{strong_uniform0} (uniform sharp asymptotics) the zoomed in profile function $Z=Z(\rho,\varphi,\tau)$, as defined in \eqref{zoomed_in_5}, satisfies
\be\label{der_sol1}
\max_{0\leq i+j\leq 10}\sup_{\varphi}\sup_{\tau\leq \tau_0} \sup_{\rho\leq \eps^{-1}}|\partial_\tau^i\partial_\rho^j(Z-Z_B)|\leq \eps.
\ee
This yields the desired estimates in the soliton region, namely
\be\label{der_sol2}
\frac{1}{4}|\tau|^{1/2}\leq \left|\frac{Y_v}{v}\right|  \leq  |\tau|^{1/2} \quad \textrm{and} \quad |Y_\tau|\leq \eta \left|\frac{Y_v}{v}\right|
\ee
for all $v\leq  L/|\tau|^{1/2}$ and $\tau\leq \tau_0$ (see e.g. \cite[Proof of Lemma 7.4]{ADS2} for a similar computation, with more details).

Furthermore, by Corollary \ref{prop-great} (almost Gaussian collar) we have
\be
\left| 1+\frac{vY}{2Y_v} \right| \leq\eps
\ee
for $L/|\tau|^{1/2}\leq v\leq 2\theta$ and $\tau\leq\tau_0$. On the other hand, thanks to Theorem \ref{strong_uniform0} (uniform sharp asymptotics) for $v\leq 2\theta$ and $\tau\leq\tau_0$ we have
\be\label{Y_squared_2}
\left| \frac{Y(v,\tau)}{\sqrt{2|\tau|}} -1\right|\leq \eps,
\ee
where $\eps$ can be made as small as we want by adjusting $\theta,\kappa$ and $\tau_\ast$. Hence, in the collar region we obtain the sharper estimate
\be\label{est_collar_sharper}
\left|\frac{Y_v}{v|\tau|^{1/2}}+\frac{1}{\sqrt{2}}\right|  \leq \eps.
\ee

Finally, to estimate $Y_\tau$ in the collar region we rewrite the evolution equation \eqref{eq-f-polar} from the introduction (for the derivation of this evolution equation see Lemma \ref{lemma-evol-Y} below) in the form
\begin{equation}
\label{eq-f-polar_rewritten}
\begin{split}
Y_{\tau} = &\frac{(Y^2 + Y_{\varphi}^2)Y_{vv}  - 2Y_{\varphi}Y_v Y_{v\varphi} + (1 + Y_v^2)Y_{\varphi\varphi}}{ Y^2\, (1 + Y_v^2)+Y_{\varphi}^2}   \\
& - \frac{Y_{\varphi}^2}{Y\,(  Y^2\, (1 + Y_v^2)+Y_{\varphi}^2)}
 + \frac{Y_v}{v} \left( 1+\frac{vY}{2Y_v} -\frac{v^2}{2} \right)-\frac{1}{Y}\, .
\end{split} 
\end{equation} 
By the estimates that we have already established, the expression in the second line is well controlled, specifically has absolute value less than say $\tfrac{1}{2}\eta |Y_v/v|$, possibly after adjusting $\theta,\kappa$ and $\tau_\ast$. On the other hand, differentiating the defining identity  \eqref{def_inv_sec4} of the inverse profile function, we get
\be
Y_v v_y=1,\qquad Y_{v}v_\vp + Y_\vp =0,\label{eq-inv1}
\ee
and differentiating again we obtain
\be Y_{vv}=-Y_v^3v_{yy},\qquad
Y_{v\vp}=  -Y_\vp Y_v^2  v_{yy} - Y_v^2 v_{y\vp},\label{eq-inv2}
\ee
and
\be
Y_{\varphi\varphi} = - Y_vY_\vp^2 v_{yy} - 2 Y_vY_\vp  v_{y\varphi}- Y_v v_{\varphi\varphi}.\label{eq-inv3}
\ee
Together with Corollary  \ref{lemma-cylindrical} (cylindrical estimate) and the already established estimates this implies
\be\label{eqn-cylf_rest}
\left|\frac{Y_{vv}}{1+Y_v^2}\right| + \left|\frac{Y_v  Y_{v\varphi} }{Y(1+Y_v^2)}\right| + \left|\frac{Y_{\varphi\varphi}}{Y^2}\right|  \leq \eps \, \left| \frac{Y_v}{v} \right|
\ee
in the collar region. Hence, we conclude that
\be
|Y_\tau|\leq \eta \left| \frac{Y_v}{v}\right|
\ee
holds in the collar region as well.
\end{proof}

\begin{corollary}[second tip derivatives] \label{cor-cylindrical} For every $\eta >0$, there exist $\theta > 0$, $L < \infty $, $\kappa > 0$  and $\tau_* > -\infty$
such that if $\mathcal{M}$ is $\kappa$-quadratic at time $\tau_0 \le \tau_*$, then for all $\tau\leq\tau_0$ in the collar region $L/|\tau|^{1/2}\leq v\leq 2\theta$ we have
\be\label{eqn-cylf}
\left|\frac{Y_{vv}}{1+Y_v^2}\right| + \left|\frac{Y_v  Y_{v\varphi} }{Y(1+Y_v^2)}\right| + \left|\frac{Y_{\varphi\varphi}}{Y^2}\right|  \leq \eta \,  \left| \frac{Y_v}{v} \right|,
\ee
and in the soliton region $v\leq L/|\tau|^{1/2}$ we have the sharper estimates
\be\label{eqn-uphi}
\left|Y_{vv}\right|  \le C|\tau|^{1/2}, \quad 
\left| Y_{v\varphi} \right|  \leq \eta\, |\tau|, \quad  |Y_{\vp\vp}|  \leq  \eta |\tau|^{3/2}.
\ee
\end{corollary} 

\begin{proof}We have already obtained the estimate \eqref{eqn-cylf} in the above proof, which in particular gives us some constant $L<\infty$. Now with this fixed $L$, we will deal with the soliton region $v\leq L/|\tau|^{1/2}$. To this end, note first that thanks to \eqref{eq-ZB-asymptotics} and strict convexity there is a constant $\Lambda<\infty$ such that
\be
\Lambda^{-1}\leq Z_{B,\rho\rho}\leq \Lambda\, .
\ee
Hence, using \eqref{der_sol1}, in the soliton region we get
\be\label{sol_useful1}
\frac{1}{2 \Lambda}\leq \frac{|Y_{vv}|}{|\tau|^{1/2}}\leq 2\Lambda \, .
\ee
To proceed, consider the rescaled profile function of 2d-bowl$\times\mathbb{R}$, namely
\be
\Upsilon(v,x,\tau) := \frac{1}{|\tau|^{1/2}}Z_B(|\tau|^{1/2}v).
\ee
We also remind the reader that a hypersurface $M'$ is called $\eps$-close in $C^{\lfloor 1/\eps\rfloor}$ in $B_{1/\eps}(0)$ to a hypersurface $M$, if it can be written as normal graph of a function $\psi$ over $M\cap B_{1/\eps}(0)$ with $\|\psi\|_{C^{\lfloor 1/\eps\rfloor}(M\cap B_{1/\eps}(0))}  \leq \eps$. In particular, if the hypersurfaces are $\eps$-close, then associated geometric quantities (such as curvatures, curvature ratios, Hessian ratios, etc) are also almost equal.\\
Now, setting $V=\partial_v$ and $X=\partial_x$ we have
\be\label{Upsilon_Hessian}
\Lambda^{-1}\leq \frac{D^2\Upsilon(V,V) }{|\tau|^{1/2}} \leq \Lambda ,\quad D^2\Upsilon(V,X)=0,\quad D^2\Upsilon(X,X)=0.
\ee
Then, since ratios are scaling invariant, by Theorem \ref{strong_uniform0} (uniform sharp asymptotics) and the definition of $\eps$-closeness, in the soliton region we get
\be\label{sol_useful_interm1}
\left| \frac{D^2Y\big(\partial_v, (2|\tau|)^{-1/2}\partial_\varphi\big)}{D^2Y(\partial_v,\partial_v)} - \frac{D^2\Upsilon(V,X)}{D^2\Upsilon(V,V)} \right|  \leq \delta,
\ee
and
\be\label{sol_useful_interm2}
\left| \frac{D^2Y\big((2|\tau|)^{-1/2}\partial_\varphi , (2|\tau|)^{-1/2}\partial_\varphi\big)}{D^2Y(\partial_v,\partial_v)} - \frac{D^2\Upsilon(X,X)}{D^2\Upsilon(V,V)} \right|  \leq \delta,
\ee
where $\delta=\delta(\eps)>0$ can be made as small as we want by decreasing $\kappa$ and $\tau_\ast$. Remembering that $D^2\Upsilon(\cdot,X)=0$
we have thus shown that
\be\label{sol_useful2}
\frac{|Y_{v\vp}|}{|\tau|^{1/2}}+\frac{|Y_{\vp\vp}|}{|\tau|}\leq \delta |Y_{vv}|,
\ee
Combining the above facts, the assertion follows.
\end{proof}

Denote by $Y_B$ the rescaled profile function of the 2d-bowl, namely
\begin{equation}
Y_B(v,\tau):=\frac{1}{|\tau|^{1/2}}Z_B(|\tau|^{1/2}v).
\end{equation}

\begin{corollary}[rescaled bowl profile]\label{cor-YZv}   For every $\eta >0$ and $L<\infty$, there exist $\kappa > 0$, $\tau_* > -\infty$ and $\theta > 0$,
such that if $\mathcal{M}$ is $\kappa$-quadratic at time $\tau_0 \le \tau_*$, then for all  $\tau\leq \tau_0$ and $v\leq 2\theta$ we have
\begin{equation}\label{eq-soliton-improvement}
\left|\frac{1+Y_{B,v}^2}{1+Y_v^2} - 1 \right| \leq \eta \min\left\{ 1,\frac{|\tau|^{1/2} v}{L}\right\}.
\end{equation}
\end{corollary}

\begin{proof} 
By \eqref{der_sol1} and \eqref{der_sol2} in the soliton region $v\leq L/|\tau|^{1/2}$ we have
\be
\left|\frac{1+Y_{B,v}^2}{1+Y_v^2} - 1\right| \leq \left|(Y_{B,v}-Y_{v})(Y_{B,v}+Y_{v})\right| \leq 2\eps |\tau|^{1/2}v.
\ee
On the other hand, by \eqref{est_collar_sharper} in the collar region $L/|\tau|^{1/2}\leq v\leq 2\theta$  we have
\be\label{est_collar_sharper_res}
\left|\frac{Y_v^2}{v^2|\tau|}-\frac{1}{2}\right|  \leq \eps.
\ee
Since \eqref{est_collar_sharper_res} of course in particular holds for $Y=Y_{B}$, the result follows.
\end{proof}

\bigskip

In the tip region we consider the weight function 
\begin{multline}\label{eqn-weight1}
 \mu(v,\vp, \tau) = - \frac14 Y^2(\theta, \vp, \tau)\\
  + \int_v^\theta\left[  \zeta (v') \, \left(\frac{Y^2({v'}, \vp, \tau)}{4}\right)_{v'} - (1-\zeta(v')) \, \frac{1+Y_{B,{v'}}^2(v',\tau)}{v'}\right] \, dv'\, ,
\end{multline}
where $\zeta: \mathbb{R} \to [0,1]$ is a smooth monotone function satisfying
\be\label{zeta_cutoff}
\zeta(v) = 0\,\, \textrm{for} \,\, v \leq \tfrac18 \theta\quad\textrm{ and }\quad \zeta(v) = 1\,\,\textrm{for} \,\, v \geq \tfrac14 \theta. 
\ee

\begin{proposition}[weight estimates]\label{lemma-imp}  For any $\eta >0$, there exist $\theta > 0$, $\kappa > 0$, and $\tau_* > -\infty$ with the following significance.
If $\mathcal{M}$ is $\kappa$-quadratic at time $\tau_0 \le \tau_*$,  then for all $\tau \le \tau_0$ and $v\leq 2\theta$ we have
\be\label{eqn-imp}
\Big |\frac{v \mu_v}{1+Y_{v}^2} - 1\Big | \,  \leq \eta,\qquad
 |\mu_\vp  | \leq \eta \, |\tau |,\qquad
|\mu_\tau | \leq \eta \,  |\tau|\, .
\ee
\end{proposition}
 
\begin{proof}
To begin with, note that
\be
\mu_v = \zeta \left(-\frac{Y^2}{4}\right)_{v} + (1-\zeta) \, \frac{1+Y_{B,v}^2}{v}.
\ee
This yields
\begin{align} 
\left| \frac{v \mu_v}{1+Y_{v}^2} - 1 \right| \leq   \zeta \left| \frac{Y_v^2}{1+Y_v^2}\frac{vY}{2Y_v} +1 \right|+ (1-\zeta) \left|\frac{1+Y_{B,v}^2}{1+Y_v^2} - 1 \right| \leq \eta,
\end{align}
where to estimate the first term we used \eqref{zeta_cutoff}, Proposition \ref{lemma-mainY} (first tip derivatives) and Corollary \ref{prop-great} (almost Gaussian collar), and to estimate the second term we used Corollary \ref{cor-YZv} (rescaled bowl profile).

Next, via integration by parts we see that
\be
|\mu_{\vp}| =\left|\zeta \left(\frac{Y^2}{4}\right)_\vp +\int_v^\theta \zeta'  \left(\frac{Y^2}{4}\right)_\vp\right|\leq \eta|\tau|,
\ee
where we used Proposition \ref{lemma-mainY} (first tip derivatives) and $Y\leq 2|\tau|^{1/2}$.

Finally, arguing similarly we can estimate
  \begin{align}\label{eqn-mutau5}
   | \mu_\tau| =  \left| \zeta \left(\frac{Y^2}{4}\right)_\tau +\int_v^\theta \zeta'  \left(\frac{Y^2}{4}\right)_\tau
    +  \int_\theta^v  (1-\zeta) \,   \frac {(Y_{B,v'}^2)_\tau}{v'}  \, dv'\right| \leq \eta|\tau|,
  \end{align}
where we used in addition that  $\big|(Y_{B,v}^2)_\tau\big | \leq C v^2$ thanks to \eqref{eq-ZB-asymptotics}.
\end{proof}

\begin{corollary}[weighted Poincar\'e inequality]
\label{prop-Poincare}
There exist  $C_0<\infty$, $\theta>0$, $\kappa > 0$ and $\tau_* > -\infty$ with the following significance.
If $\mathcal{M}$ is $\kappa$-quadratic at time $\tau_0 \le \tau_*$,  then for $\tau \le \tau_0$ and $\varphi \in [0,2\pi ]$ we have 
\begin{equation}
\label{eq-Poincare}
\int_0^{2\theta} F^2(v) \, e^{\mu(v,\varphi,\tau )}\, dv  \le \frac{C_0}{|\tau|}\, \int_0^{2\theta} \frac{F_v^2(v)}{1+Y_{v}^2(v,\varphi,\tau)}\, e^{\mu(v,\varphi,\tau)}\, dv
\end{equation}
for all smooth functions $F$ satisfying $F'(0)=0$ and $\textrm{spt}(F)\subset [0,2\theta)$.
\end{corollary}

\begin{proof} Thanks to \eqref{eq-ZB-asymptotics} there is a constant $\Lambda<\infty$ such that
\be\label{lambda_bound_bowl}
\Lambda^{-1} |\tau|^{1/2} v\leq |Y_{B,v}|\leq  \Lambda |\tau|^{1/2}v.
\ee
We fix $v_0=v_0(\tau)= 3\Lambda/ |\tau|^{1/2}$ and start with the integration by parts formula
\be
 - \int_{v_0}^{2\theta}  \frac{(F^2 )_v}v  \, e^{\mu}  dv \\
    =    \frac{F(v_0)^2}{v_0}e^{\mu(v_0,\vp, \tau)} +  \int_{v_0}^{2\theta}  (v\mu_v - 1) \frac{F^2}{v^2}  e^{\mu} dv.
\ee
Together with
\be
  -\frac {2\,F F_v}v  \leq \frac{4 F_v^2}{1+Y_{v}^2}  +   (1+Y_{v}^2)\, \frac{F^2}{4v^2},
 \ee
 this yields
  \begin{multline}  \label{eqn-Y13}
  \frac{F(v_0)^2}{v_0}e^{\mu(v_0,\vp, \tau)}+  \int_{v_0}^{2\theta} \Big ( v\mu_v -  \tfrac 14  (1+Y_{v}^2) -1 \Big ) \, \frac{F^2}{v^2} \, e^\mu dv \\
    \leq  4   \int_{v_0}^{2\theta}  \frac{F_v^2}{1+Y_{v}^2} \, e^{\mu}  dv .
  \end{multline}
Now, using Corollary \ref{cor-YZv} (rescaled bowl profile), Proposition \ref{lemma-imp} (weight estimates) and \eqref{lambda_bound_bowl} we can estimate our integrand by
\be
v\mu_v  -  \tfrac 14  (1+Y_{v}^2) -1  \geq \tfrac12 (1+Y_{B,v}^2)-1\geq \tfrac14 (1+Y_{B,v}^2).
\ee
This implies
\begin{equation}\label{eqn-good100}
  |\tau| \int_{v_0}^{2\theta}  F^2\, e^{\mu}\, dv
  \leq  C   \int_{v_0}^{2\theta}  \frac{F_v^2}{1+Y_{v}^2} \, e^\mu \, dv,
\end{equation}
and
\be\label{eqn_good_bdryterm}
\frac{F(v_0)^2}{v_0}e^{\mu(v_0,\vp, \tau)}\leq C   \int_{v_0}^{2\theta}  \frac{F_v^2}{1+Y_{v}^2} \, e^\mu \, dv.
\ee
Finally, for $v\leq v_0$ by our choice of weight function we have
\be
\left|\mu(v,\varphi,\tau)-\mu(v_0,\vp,\tau)-\log\left(\frac{v}{v_0}\right)\right| =\left| \int_{v}^{v_0}\frac{1}{v'}Y_{B,v'}^2 dv'\right| \leq C,
\ee
hence
\be
C^{-1} \left(\frac{v}{v_0}\right)\leq e^{\mu(v,\vp,\tau)-\mu(v_0,\varphi,\tau)}\leq C \left(\frac{v}{v_0}\right).
\ee
Together with the standard Poincar\'e inequality 
\be
\int_0^{v_0} (F(v)-F(v_0))^2\,  v  dv \leq C v_0^2 \int_0^{v_0} F_v^2\, v dv,
\ee
taking also into account \eqref{eqn_good_bdryterm}, this yields
\be
|\tau|\int_0^{v_0} F^2 e^{\mu} dv \leq C \int_0^{2\theta} \frac{F_v^2}{1+Y_v^2}e^\mu dv,
\ee
and thus concludes the proof of the corollary.
\end{proof}

To conclude this subsection, let us prove the evolution equation for $Y$, which has already been used above:

\begin{lemma}[{evolution of $Y$}] \label{lemma-evol-Y} 
The function $Y$ evolves by
\begin{equation}
\label{eq-f-polar_lemma}
\begin{split}
Y_{\tau} = &\frac{(Y^2 + Y_{\varphi}^2)Y_{vv}  - 2Y_{\varphi}Y_v Y_{\varphi v} + (1 + Y_v^2)Y_{\varphi\varphi}}{ Y^2\, (1 + Y_v^2)+Y_{\varphi}^2}   \\
&+  \left(\frac 1v - \frac v2\right)Y_v - \frac{Y_{\varphi}^2}{Y\,(  Y^2\, (1 + Y_v^2)+Y_{\varphi}^2)} + \frac{Y}{2} - \frac 1Y.
\end{split} 
\end{equation} 
\end{lemma}

\begin{proof}
Recall that $v=v(y_1,y_2,\tau)$ evolves by
\be v_\tau= \left( \delta_{ij} - \frac{v_{y_i}v_{y_j}}{1+|Dv|^2} \right) v_{y_iy_j}-\frac{y_i }{2}v_{y_i} + \frac{v}{2}  -\frac{1}{v}.\ee 
Setting $y_1= y \cos \varphi$ and $y_2=y \sin \varphi$, by the chain rule we have
\bea \partial_{y_1} = \cos \vp \partial_{y} - \frac{\sin \vp}{y} \partial_{\vp} ,\quad \partial_{y_2} = \sin \vp \partial_{y} + \frac{\cos \vp}{y} \partial_{\vp}. \eea 
Hence, $v=v(y,\varphi,\tau)$, viewed as a function of polar coordinates, evolves by\footnote{Note that for $v_\varphi=0$ this reduces to the formula $v_{\tau} = \frac{v_{yy}}{  (1 + v_y^2)} +  \frac{v_y}{y} - \frac {y }2v_y  + \frac v2 - \frac 1v$ for the $\mathrm{O}_2\times \mathrm{O}_2$-symmetric bubble-sheet ovals from \cite[Equation (2.75)]{DH_ovals}.}
 \begin{equation}
\begin{split} \label{eq-u-polar-deriv_re}
v_{\tau} &= \frac{(y^2+v_{\varphi}^2)v_{yy} - 2v_{\varphi}  v_y v_{\varphi y} + (1 + v_y^2)v_{\varphi\varphi}}{ y^2\, (1 + v_y^2)+v_{\varphi}^2} \,\\
&\quad +  \left(\frac{2}{y^2} - \frac 12 - \frac{1+v_y^2}{ y^2(1+v_y^2)+v_{\varphi}^2}\right) y v_y + \frac v2 - \frac 1v.
\end{split}
\end{equation}
Finally, differentiating the identity $y= Y(v(y,\varphi,\tau),\varphi,\tau)$ we see that
\bea
v_\tau =-Y_\tau /Y_v, \quad  v_y= 1/ Y_v,\quad   v_{\varphi}= - Y_{\varphi} /Y_v,\quad  v_{yy}= - Y_{vv}/Y_v^3,\\ 
v_{\varphi\varphi}=- \frac{Y_{\varphi\varphi}}{Y_v}+ \frac{2 Y_\varphi Y_{\varphi v}}{Y_v^2}  - \frac{Y_\varphi ^2 Y_{vv}}{Y_v^3},\quad v_{\varphi y} =- \frac{Y_{\varphi v}}{Y_v^2}+ \frac{Y_{\varphi}Y_{vv}}{Y_v^3} .
\eea  
Plugging this into \eqref{eq-u-polar-deriv_re} yields the assertion.
\end{proof}

\medskip

\subsection{Energy estimate in the  tip region}\label{sec-tip}
In this subsection, we prove an energy estimate in the tip region, by generalizing some arguments from \cite[Section 7]{ADS2} and \cite[Section 5.5]{CHH_translator} to the bubble-sheet setting.

We denote the inverse profile functions by $Y=Y_1$ and $\bY = Y_2$. Consider the difference
\be
W :=Y-\bY .
\ee

\begin{lemma}[{evolution of $W$}]\label{lemma-ev-W}
The function $W$ evolves by
\begin{align}\label{eq-W}
W_{\tau}& = \frac{(Y^2 + Y_{\varphi}^2) W_{vv} - 2Y_{\varphi} Y_v W_{\varphi v} + (1 + Y_v^2)\, W_{\varphi\varphi}}{D}\\
&+ a \,W_v + b\, W_{\varphi} + c\, W,\nonumber
\end{align}
where 
\be
D= Y^2\, (1 + Y_v^2)+Y_{\varphi}^2,
\ee
and where the coefficients $a$, $b$ and $c$ are specified in equation \eqref{eq-abc} below.
\end{lemma}

\begin{proof}
Let us abbreviate 
\be
 \bar D= \bar Y^2\, (1 + \bar Y_v^2)+\bar Y_{\varphi}^2,
\ee
and
\be\label{eq-barB}
\bar B=( \bar Y^2+\bar Y_\varphi^2)\bar Y_{vv} - 2\bar Y_\varphi \bar Y_v \bar Y_{\varphi v} +(1+\bar Y_v^2)\bar Y_{\varphi\varphi} .
\ee
Then, Lemma \ref{lemma-evol-Y} (evolution of $Y$) and the product rule for differences yield
\bea
W_\tau &= \frac{ (Y^2 + Y_\varphi^2 )W_{vv}  -2Y_\varphi Y_v W_{\varphi v} + (1+Y_v^2)W_{\varphi \varphi} }{D}+\frac{ \bar Y_{\varphi \varphi}(Y_v+\bar Y_v) W_v}{D}   \\
& + \frac{\bar Y_{vv}[(Y+\bar Y)W + (Y_{\varphi}+\bar Y_{\varphi})W_{\varphi}]}{D} - 2\frac{ \bar Y_{\varphi v}( Y_\varphi W_v +\bar Y_{v} W_\varphi) }{D}- \frac{\bar B(D-\bar D)}{D\bar D} \\ 
& + \bigg (\frac{1}{v} -\frac v{2}\bigg )W_v   -  \frac{Y_\varphi +\bar Y_\varphi}{YD}W_\varphi +\bar Y_{\varphi}^2  \bigg( \frac{W}{DY\bar Y} +  \frac{D-\bar D}{\bar Y D \bar D}\bigg ) +  \bigg (\frac12 +\frac{1}{Y \bar Y} \bigg )W.
 \eea
Furthermore, note that
\bea
D-\bar D = \bar Y^2 (Y_v+\bar Y_v)W_v+(Y_\varphi +\bar Y_\varphi)W_\varphi + (Y+\bar Y)(1+Y_v^2) W .
\eea
Hence, collecting the coefficients of $W_v$, $W_\varphi$ and $W$, we obtain the claimed evolution equation with
\bea\label{eq-abc}
a&=\frac1 v - \frac v 2 + \frac{\bar Y_{\varphi\varphi}(Y_v+\bar Y_v)-2\bar Y_{\varphi v}Y_\varphi}{D}  + \frac{\bar Y^2 (Y_v +\bar Y_v)}{D\bar D} \left[ \frac{\bar Y_\varphi^2}{\bar Y} -\bar B\right],\\
b&= \frac{\bar Y_{vv} (Y_\varphi+\bar Y_\varphi )-2\bar Y_{\varphi v}\bar Y_v }{D} - \frac{Y_\varphi +\bar Y_\varphi}{YD}  + \frac{ (Y_\varphi  +\bar Y_\varphi )}{D\bar D} \left[ \frac{\bar Y_\varphi^2}{\bar Y} -\bar B\right],\\
c&=\frac12 +\frac1 {Y\bar Y} + \frac{\bar Y_{vv} (Y+\bar Y ) }{D} + \frac{\bar Y_\varphi^2}{DY\bar Y}     + \frac{ (Y  +\bar Y )(1+Y_v^2)}{D\bar D} \left[ \frac{\bar Y_\varphi^2}{\bar Y} -\bar B\right].
\eea
This proves the lemma.
\end{proof}

Now, considering $W_{\cT}=\chi_{\cT}W$ we have the following energy inequality:

\begin{proposition}[energy inequality]\label{lemma-energy}
 There exist constants $\theta > 0$, $\kappa > 0$, $\tau_* > -\infty$
and $C = C(\theta) < \infty$ with the following significance. If $\mathcal{M}$ is $\kappa$-quadratic at time $\tau_0 \le \tau_*$,  then for $\tau \le \tau_0$ we have
\begin{align}
\frac12\frac{d}{d\tau}\, \int W_\cT^2 \mm \le  - \frac{1}{20} \int \frac{(W_\cT)^2_v}{1+Y_v^2}\, \mm + \frac{C}{|\tau|}\, \int W^2 1_{\{ \theta\leq v\leq 2\theta\}}\, \mm\nonumber\\
+ \int \left( \frac{ D}{Y^2}\tilde{a}^2 + \frac{D}{1+Y_v^2}\tilde{b}^2+\tilde c \right)\,  W_\cT^2\, \mm ,
\end{align}
where the coefficients $\tilde{a}$, $\tilde{b}$ and $\tilde c $ are specified in equation \eqref{eq-abc-tilde} below.
\end{proposition}

\begin{proof}
Throughout this proof, we denote $\chi_\cT$ simply by $\chi$ and denote $\int\! d\vp dv$ simply by $\int$. To begin with, using Lemma \ref{lemma-ev-W} (evolution of $W$) and integration by parts we see that
 \begin{align}\label{eqn-We1}
\frac12 \frac{d}{d\tau}  \int  W_\cT^2 e^\mu  =& - \,  \int \left( \frac{Y^2 + Y_{\varphi}^2}{D} \, W_v^2  - \frac{2  Y_\vp Y_v}{D} \,   W_\vp W_v + 
 \frac{1 + Y_v^2}D  \, W_\vp^2 \right) \chi^2 e^\mu\nonumber \\
 &+ \int \left( \tilde a W W_v +\tilde b W W_{\varphi}  +   \tilde c W^2\right) \chi^2e^\mu\\
 &-\int\left(\frac{Y^2+Y_\vp^2}{D} WW_v -\frac{2Y_\vp Y_v}{D} W W_\vp \right) (\chi^2)' e^{\mu},\nonumber
\end{align}
where
\begin{align}\label{eq-abc-tilde}
\tilde a &=  a  -  \left ( \frac{Y^2 + Y_{\varphi}^2}{D} \right)_v-  \frac {Y^2+Y_\vp^2}{D} \mu_v,\nonumber\\
\tilde b &=  b + \left(\frac{2 Y_v Y_\vp}{D }\right)_v + \frac{2 Y_v Y_\vp}{D} \, \mu_v  -  \left(\frac{1 + Y_v^2}{D}\right)_\vp  - \frac{(1+Y_v^2)}{D}\, \mu_{\vp},\\
\tilde c &= c + \frac{\mu_\tau}{2}.\nonumber
\end{align}
With the aim of absorbing various mixed terms in \eqref{eqn-We1} we estimate
\begin{align}
\tilde a  WW_v \chi^2 &  \leq  \frac14  \frac{Y^2}{D} W_v^2 \chi^2 +   \frac {D}{Y^2} \tilde a^2 W_\cT^2 ,\\
\tilde b  WW_\vp  \chi^2  &\leq  \frac 14  \frac{1+ Y_v^2}{D}  W_\vp^2  \chi^2+\frac {D}{1+ Y_v^2} \tilde b^2  W_\cT^2 ,
\end{align}
and
\begin{align} \label{eq-YvYp}
 \frac{2 Y_v Y_\vp}{D} W_v  W_\vp\leq \frac14  \frac{Y^2}{D} W_v^2 +  \frac{1}{100} \frac{Y_v^2}{D}W_\vp^2,
  \end{align}
where in the last step we used that $Y_\vp^2/Y^2\leq \eta$ thanks to Proposition \ref{lemma-mainY} (first tip derivatives) and Theorem \ref{strong_uniform0} (uniform sharp asymptotics).

\smallskip 
Combining the above, and discarding some lower order good terms, gives
 \begin{align}\label{eqn-We2}
\frac12 \frac{d}{d\tau}  \int  W_\cT^2 e^\mu  \leq& -\frac12 \,  \int \left( \frac{Y^2 }{D} \, W_v^2 + 
 \frac{ Y_v^2}D  \, W_\vp^2 \right) \chi^2 e^\mu+\int \mathcal{I}\, W_\cT^2 e^\mu\nonumber\\
  &-2\int\left(\frac{Y^2+Y_\vp^2}{D} WW_v -\frac{2Y_\vp Y_v}{D} W W_\vp \right) \chi\chi' e^{\mu},
\end{align}
where we abbreviated
\be\label{abbrev_I}
\mathcal{I}=   \frac {D}{Y^2} \tilde a^2+\frac {D}{1+ Y_v^2} \tilde b^2 +\tilde c .
\ee
Now, using the fact that $W_v\, \chi = (W_\cT)_v-W\chi'$ we can estimate
\be
- W_v^2 \chi^2 \leq -\frac{1}{2} (W_\cT)_v^2 + 2  W^2\chi'^2.
\ee
Moreover, we also have 
\be
-W W_v \chi \chi' = -(W_\cT)_v  W \chi' + W^2 \chi'^2 \le \frac{1}{12} (W_\cT)_v^2 + 4\chi'^2 W^2  ,
\ee
and
\be
 \frac{Y_v Y_{\vp}}{D} W W_{\vp} \chi\chi' \le \frac{Y^2 }{D} W^2 \chi'^2 + \frac{1}{100} \frac{Y_v^2}{D} W_{\vp}^2\chi^2,
\ee
where in the last step we used again that $Y_\vp^2/Y^2\leq \eta$. This yields
 \begin{multline}\label{eqn-We3}
\frac12 \frac{d}{d\tau}  \int  W_\cT^2 e^\mu  \leq -\frac{1}{12} \,  \int \frac{Y^2 }{D} \, (W_\cT)_v^2\,  e^\mu\\
+\int \mathcal{I}\, W_\cT^2\, e^\mu
+13\int \frac{Y^2}{D} W^2\chi'^2\, e^{\mu}.
\end{multline}
Finally, using again that $Y_\vp^2/Y^2\leq \eta$ we observe that we can replace $Y^2/D$ by $1/(1+Y_v^2)$ up to a multiplicative factor close to $1$. Remembering also that $|Y_v|\geq\tfrac{1}{4}\theta |\tau|^{1/2}$ on the support of $\chi'$ thanks to  Proposition \ref{lemma-mainY} (first tip derivatives), this concludes the proof of the proposition.
\end{proof}

To put the energy inequality into use it is crucial to control the coefficients:

\begin{lemma}[estimate for coefficients]
\label{lemma-coefficients}
For every $\eta >0$, there exist constants $\kappa > 0$, $\tau_* > -\infty$ and $\theta > 0$ with the following significance.
If $\mathcal{M}$ is $\kappa$-quadratic at time $\tau_0 \le \tau_*$,  then for $\tau \le \tau_0$ and $v\leq 2\theta$ we have
\be
    (1+Y_v^2)\tilde{a}^2+   |\tau| \tilde{b}^2 +|\tilde c | \leq \eta \, |\tau|.
\ee
In particular, the quantity $\mathcal{I}$ defined in \eqref{abbrev_I} satisfies $\mathcal{I}\leq 3\eta|\tau|$.
\end{lemma}

\begin{proof}
Throughout this proof, $o(1)$ denotes a quantity that can be made arbitrarily small in the tip region $v\leq 2\theta$, by choosing $\kappa$ small enough, $\tau_*$ negative enough and $\theta$ small enough.  Moreover, we abbreviate
\be
f\sim g \qquad :\Leftrightarrow \qquad \left|\frac{f}{g}-1\right|=o(1).
\ee
In particular, by definition we have
\be
v=o(1),
\ee
and thanks to Theorem \ref{strong_uniform0} (uniform sharp asymptotics) we get
\be
Y\sim \bar{Y}\sim \sqrt{2|\tau|}\, .
\ee
Recall that by Proposition \ref{lemma-mainY} (first tip derivatives) we have
\be
|Y_\vp|+|\bar Y_\varphi|=o(1)|\tau|^{1/2}.
\ee
Together with Corollary \ref{cor-YZv} (rescaled bowl profile) the above implies
\be
D\sim Y^2(1+Y_v^2)\sim \bar{Y}^2(1+\bar{Y}_v^2)\sim\bar{D}.
\ee
Also recall that by Proposition \ref{lemma-mainY} (first tip derivatives) we have
\be
\qquad \tfrac14 v|\tau|^{1/2}\leq |Y_v|\leq v|\tau|^{1/2}\, ,
\ee
and that by Corollary \ref{cor-cylindrical} (second tip derivatives) we have
\be\label{eqn-uphi_rest}
\left|\frac{Y_{vv}}{1+Y_v^2}\right|  \le C|\tau|^{1/2}, \quad 
\left|\frac{Y_v Y_{\varphi v} }{1+Y_v^2}\right|  =o(1) |\tau|, \quad  |Y_{\vp\vp}|  =o(1) |\tau|^{3/2},
\ee
and similarly for $\bar Y$, and that by Proposition \ref{lemma-imp} (weight estimates) we get
\be\label{eqn-imp_rest}
v \mu_v \sim 1+Y_{v}^2,\qquad
 |\mu_\vp  | + |\mu_\tau | =o(1) |\tau|\, .
\ee

\medskip

Now, let us begin by estimating
\be
\tilde c=\frac{1+\mu_\tau}{2} +\frac1 {Y\bar Y} + \frac{\bar Y_{vv} (Y+\bar Y ) }{D} + \frac{\bar Y_\varphi^2}{DY\bar Y}     + \frac{ (Y  +\bar Y )(1+Y_v^2)}{D\bar D} \left[ \frac{\bar Y_\varphi^2}{\bar Y} -\bar B\right].
\ee
The facts from above directly imply that the first four terms are of order at most $o(1)|\tau|$. That actually works for the last term as well, observing that
\be
\left| \frac{ 1}{\bar D} \left[ \frac{\bar Y_\varphi^2}{\bar Y} -\bar B\right]\right| \leq C|\tau|^{1/2}.
\ee
We have thus shown that
\be
|\tilde c| = o(1)|\tau|.
\ee

\medskip

Next, all the terms in 
\be
b= \frac{\bar Y_{vv} (Y_\varphi+\bar Y_\varphi )-2\bar Y_{\varphi v}\bar Y_v }{D} - \frac{Y_\varphi +\bar Y_\varphi}{YD}  + \frac{ (Y_\varphi  +\bar Y_\varphi )}{D\bar D} \left[ \frac{\bar Y_\varphi^2}{\bar Y} -\bar B\right]
\ee
can be dealt with similarly as above, yielding
\be
|b|=o(1).
\ee
Let us now consider
\be
\tilde b - b = \left(\frac{2 Y_v Y_\vp}{D }\right)_v + \frac{2 Y_v Y_\vp}{D} \, \mu_v  -  \left(\frac{1 + Y_v^2}{D}\right)_\vp  - \frac{(1+Y_v^2)}{D}\, \mu_{\vp}.
\ee
The above facts directly imply that the second and fourth term are of order at most $o(1)$. To deal with the first term, we observe that
\begin{align}
\left| \left(\frac{ Y_v Y_\vp}{D }\right)_v \right|&\leq \left| \frac{ Y_{vv} Y_\vp+Y_v Y_{\vp v}}{D }\right| +\left|\frac{ Y_v Y_\vp}{D^2 }D_v \right|\\
&= o(1)+\frac{o(1)}{|\tau|^{1/2}} \left| \frac{Y_v D_v}{(1+Y_v^2)D} \right| ,\nonumber
\end{align}
and
\begin{align}
\left| \frac{Y_v D_v}{(1+Y_v^2)D} \right|  &\leq  \left| \frac{2Y_v}{1+Y_v^2} \right| \left| \frac{Y_{\vp} Y_{\vp v} +  Y Y_v (1+ Y_v^2) + Y^2 Y_v Y_{vv}}{D}\right|\\
&\leq C |\tau|^{1/2}.\nonumber
\end{align}
Similarly, we can estimate the third term by
\be
\left|\left(\frac{1 + Y_v^2}{D}\right)_\vp\right| \leq o(1) +\frac{C}{|\tau|} \left| \frac{Y_{\vp} Y_{\vp\vp} +  Y Y_{\vp} ( 1+ Y_v^2) + Y^2 Y_v Y_{v\vp}}{D}\right| = o(1) .
\ee
Summing up, we have thus shown that
\be
|\tilde b|=o(1).
\ee

\medskip

It remains to show that $\sqrt{1+Y_v^2}|\tilde{a}|$ is of order at most $o(1)|\tau|^{1/2}$. To this end, we consider the quantities
\begin{align}
\tilde{a}_1 =  \frac{1}{v} - \mu_v \frac{Y^2+Y^2_{\varphi}}{D},
\end{align}
and
\begin{align}
\tilde{a}_2 =\frac{2Y^4Y_{v}Y_{vv}}{D^2}- \frac{\bar Y^4(Y_v+\bar Y_v) \bar Y_{vv}}{D\bar D}.
\end{align}
Then, a direct computation shows that
\begin{align}
\tilde{a}-\tilde{a}_1-\tilde{a}_2 =& \frac{\bar Y_{\varphi\varphi}(Y_v+\bar Y_v)-2(Y_{\vp v}+\bar Y_{\varphi v})Y_\varphi -2YY_v}{D} - \frac v 2\nonumber  \\
 &+ \frac{2Y^2Y_{\varphi}^2 Y_v Y_{vv}+2(Y^2 + Y_{\varphi}^2)\left(Y_{\vp} Y_{\vp v} +  Y Y_v (1+ Y_v^2)\right)}{D^2}\nonumber\\
 &- \frac{\bar Y^2 (Y_v +\bar Y_v)}{D\bar D} \left(\bar Y_\varphi^2\bar Y_{vv} - 2\bar Y_\varphi \bar Y_v \bar Y_{\varphi v} +(1+\bar Y_v^2)\bar Y_{\varphi\varphi} -\frac{\bar Y_\vp^2}{\bar Y}\right).
\end{align}

Most of these terms multiplied by $\sqrt{1+Y_v^2}$ can be estimated similarly as above, taking also into account the elementary fact that the function $x\mapsto x/\sqrt{1+x^2}$ is bounded. The only three terms for which one has to argue somewhat differently are $v$ and $(Y_{\vp v}+\bar{Y}_{\vp v}) Y_\vp/D$ and $(Y^2+Y_\vp^2)Y_\vp Y_{\vp v}/D^2$. Regarding the first term, since $v=o(1)$ and $\sqrt{1+Y_v^2}=o(1)|\tau|^{1/2}$ we easily get
\be
\sqrt{1+Y_v^2}v=o(1)|\tau|^{1/2}.
\ee
To deal with the second term, in the collar region we estimate
\begin{multline}
\left|\sqrt{1+Y_v^2}\frac{(Y_{\vp v}+\bar{Y}_{\vp v})Y_\vp}{D}\right| \\
\sim  \left|\frac{1}{2|\tau|}\frac{(Y_{\vp v}+\bar{Y}_{\vp v})Y_v}{(1+Y_v^2)} \frac{\sqrt{1+Y_v^2}}{Y_v}Y_\vp\right| = o(1)|\tau|^{1/2},
\end{multline}
where we used that $|Y_v|$ is large in the collar region. On the other hand, thanks to Corollary \ref{cor-cylindrical} (second tip derivatives)  in the soliton region $v\leq L/|\tau|^{1/2}$ we have have the sharper estimate $|Y_{\vp v}|+|\bar Y_{\vp v}|=o(1)|\tau|$, so we also get the desired bound in the soliton region. Finally, since $(Y^2+Y_\vp^2)/D\leq 1$ the same argument applies for the third term as well. This yields
\be
\sqrt{1+Y_v^2} |\tilde{a}-\tilde{a}_1-\tilde{a}_2|=o(1)|\tau|^{1/2}.
\ee

In contrast, to deal with $\tilde{a}_1$ and $\tilde{a}_2$ we have to use cancellations.

First, in the soliton region $v\leq L/|\tau|^{1/2}$, where $L<\infty$ is fixed, using that the formula $v\mu_v=1+Y_{B,v}^2$ holds there, we can estimate
\bea |\tilde{a}_1| &= \frac{1}{v}\left \vert \frac{Y^2 Y_v^2 - Y^2 Y_{B,v}^2 -Y_{\varphi}^2 Y_{B,v}^2 }{Y^2 (1 +Y_v^2)+Y_\varphi^2} \right \vert \\
&\le \frac{1}{v}\left \vert  1-\frac{1+Y_{B,v}^2}{1+Y_v^2} \right \vert +\frac{Y_{B,v}^2}{v} \left \vert \frac{Y_\varphi^2}{Y^2} \right \vert  =o(1)|\tau|^{1/2},
\eea
where in the last step we used Corollary \ref{cor-YZv} (rescaled bowl profile) and the above standard facts. On the other hand, in the collar region ${L}/{\sqrt{|\tau|}} \le v\le 2\theta$ using in particular Proposition \ref{lemma-imp} (weight estimates) we get
\be
\left| 1 -v\mu_v \frac{Y^2+Y_\vp^2}{D}\right| = o(1).
\ee
Observing also that $v^{-1}\sqrt{1+Y_v^2}\leq C|\tau|^{1/2}$ in the collar region, this yields
\be
\sqrt{1+Y_v^2} |\tilde{a}_1|=o(1)|\tau|^{1/2}.
\ee

To estimate $\sqrt{1+Y_v^2}|\tilde{a}_2|$, motivated by the product rule for differences, we rewrite $\tilde{a}_2$ in the form
\begin{multline}
\tilde{a}_2=\frac{2 (Y^4-\bar Y^4)Y_vY_{vv}}{D\bar D}
+\frac{\bar Y^4  (Y_{v}- \bar Y_{v})\bar Y_{vv} }{D\bar D}\\
+\frac{2 \bar Y^4 Y_{v}(Y_{vv}-\bar Y_{vv})}{D\bar D} 
-\frac{2Y^4Y_vY_{vv}}{D}\frac{D-\bar D}{D\bar D}.
\end{multline}
The contribution from the first term can be readily estimated observing that $|Y^4-\bar{Y}^4|=o(1)Y^4$. To deal with the second term we use that $-Y_v\sim {{vY}/{2}} \sim -\bar{Y}_v$ in the collar region thanks to Corollary \ref{prop-great} (almost Gaussian collar) and that $|Y_v-\bar Y_v|=o(1)$ in the soliton region thanks to \eqref{der_sol1}. To deal with the third term we use that $|Y_{vv}|=\eps (1+Y_v^2)|\tau|^{1/2}$ in the collar region thanks to Corollary \ref{cor-cylindrical} (second tip derivatives), where $\eps$ can be made arbitrarily small by adjusting the parameters, in particular choosing $L$ large enough, and that $|Y_{vv}-\bar Y_{vv}|=o(1)|\tau|^{1/2}$ in the soliton region thanks to \eqref{der_sol1}. Finally, to deal with the last term we use that $D-\bar{D}=o(1)D$ thanks to Corollary \ref{cor-YZv} (rescaled bowl profile). Summing up, this yields
\be
\sqrt{1+Y_v^2} |\tilde{a}_2|=o(1)|\tau|^{1/2},
\ee
and thus concludes the proof of the lemma.
\end{proof}

\smallskip

Recall that in the tip region we work with the norm
\begin{equation}\label{def_norm_tip_r}
\big\| F\big\|_{2,\infty}= \sup_{\tau\leq \tau_0} \frac{1}{|\tau|^{1/4}} \left( \int_{\tau-1}^\tau \int_0^{2\theta} \int_0^{2\pi} F(v,\vp,\sigma)^2 e^{\mu(v,\vp,\sigma)} \, d\vp dv d\sigma \right)^{1/2}.
\end{equation}
Having established the above results, our energy estimate in the tip region now becomes smooth sailing:

\begin{proposition}[energy estimate in tip region]\label{prop-tip}  For every $\eps>0$, there exist $\kappa > 0$, $\tau_* > -\infty$ and $\theta > 0$ with the following significance. If $\mathcal{M}^1$ and $\mathcal{M}^2$ are $\kappa$-quadratic at time $\tau_0 \le \tau_*$,  then for $\tau \le \tau_0$ we have
  \begin{equation}
    \label{eqn-tip}
   \big \| W_\cT \big\|_{2,\infty} \leq \eps \, \big\| W \, 1_{ \{ \theta \leq v \leq 2\theta\} } \big \|_{2,\infty}.
  \end{equation}
\end{proposition}

\begin{proof}
By Proposition \ref{lemma-energy} (energy inequality) we have
\begin{align}\label{evol_inequ_tip_rest}
\frac12\frac{d}{d\tau}\, \int W_\cT^2 \, e^{\mu} \le  - \frac{1}{20} \int \frac{(W_\cT)^2_v}{1+Y_v^2}\, e^{\mu} + \frac{C}{|\tau|}\, \int W^2 1_{\{ \theta\leq v\leq 2\theta\}}\, e^{\mu}\nonumber\\
+ \int \left( \frac{ D}{Y^2}\tilde{a}^2 + \frac{D}{1+Y_v^2}\tilde{b}^2+\tilde c \right)\,  W_\cT^2\, e^{\mu} ,
\end{align}
and applying Proposition \ref{prop-Poincare} (Poincare inequality) we get
\begin{equation}
|\tau| \int W_\cT^2 \, e^{\mu}  \le \, C_0\int \frac{(W_\cT)^2_v}{1+Y_{v}^2}\, e^{\mu}\, .
\end{equation}
Moreover, thanks to Lemma \ref{lemma-coefficients} (estimate for coefficients) we can estimate
\be
\frac{ D}{Y^2}\tilde{a}^2 + \frac{D}{1+Y_v^2}\tilde{b}^2+\tilde c\leq 3\eta |\tau|\, .
\ee
Combining the above facts and taking $\eta=\tfrac{1}{120C_0}$ we infer that 
\begin{equation}
\frac{d}{d\tau}\, \int W_\cT^2 \, e^{\mu} \le - \frac{1}{20C_0} |\tau|  \int W_\cT^2 \, e^{\mu} +\frac{C}{|\tau|}\int W^2 1_{\{ \theta\leq v\leq 2\theta\}}\, e^{\mu}.
\end{equation}
Setting $c=1/(20C_0)$ and considering
\begin{align}
&A(\tau)=\int^{\tau}_{\tau-1} \int W_\cT^2 \, e^{\mu} , && B(\tau)=\int^{\tau}_{\tau-1}  \int W^2 1_{\{ \theta\leq v\leq 2\theta\}}\, e^{\mu},
\end{align}
it follows that
\begin{equation}
\frac{d}{d\tau} \left[ e^{-\frac{c\tau^2}{2}} A(\tau) \right] \leq C|\tau| e^{-\frac{c\tau^2}{2}}\frac{B(\tau)}{|\tau|^2} .
\end{equation}
Integrating this from $-\infty$ to $\tau$ yields
\begin{align}
 A(\tau) &\leq C \sup_{\tau'\leq \tau}|\tau'|^{-2} B(\tau') \, ,
\end{align}
and hence in particular
\begin{equation}
|\tau|^{-\frac{1}{2}}A(\tau)\leq C|\tau|^{-2}\sup_{\tau'\leq \tau}|\tau'|^{-\frac{1}{2}}B(\tau')\, .
\end{equation}
This shows that
  \begin{equation}
    \label{eq-nice}
   \big \|W_\cT\big\|_{2,\infty} \le \frac{C}{|\tau_0|}\big \|W \, 1_{ \{ \theta \leq v \leq 2\theta\} } \big\|_{2,\infty}\, ,
  \end{equation}
and thus concludes the proof of the proposition.
\end{proof}

\medskip

\subsection{Proof of the spectral uniqueness theorem}
\label{sec-conclusion}
In this subsection, we prove our spectral uniqueness theorem by generalizing the arguments from \cite[Section 8]{ADS2} and \cite[Section 5.6]{CHH_translator} to the bubble-sheet setting.

\begin{proposition}[coercivity estimate]
  \label{prop-cor-main}
  For every $\varepsilon > 0$ there exist $\kappa > 0$ and $\tau_* > -\infty$,
such that if $\mathcal{M}^{1}$ and $\mathcal{M}^{2}$ are $\kappa$-quadratic at  time $\tau_0\leq \tau_{*}$, then for $\tau \le \tau_0$ we have
\be
\|w_\cC-\mathfrak{p}_0 w_\cC \|_{\mathcal{D},\infty}+\|W_\cT\|_{2,\infty} \leq  \eps \|\mathfrak{p}_0 w_\cC \|_{\mathcal{D},\infty}\, .
\ee
\end{proposition}

\begin{proof}
First of all, to compare our different norms in the transition region, let us abbreviate
\be
f\sim_\theta g\quad :\Leftrightarrow \quad \exists C=C(\theta)<\infty : C=C^{-1}f \leq g\leq Cf.
\ee
Then, by Theorem \ref{strong_uniform0} (uniform sharp asymptotics) and convexity we have
\be 
 \big|\partial_y v_{1}(y,\varphi,\tau)\big|\sim_\theta \frac{1}{|\tau|^{1/2}} \qquad \textrm{for}\,\, \theta \leq v_1(y,\varphi,\tau) \leq 2\theta.
 \ee
By the mean value theorem this implies
\be 
\left|\frac{w(Y_1(v,\varphi,\tau),\varphi,\tau)}{W(v,\varphi,\tau)}\right |\sim_\theta \frac{1}{|\tau|^{1/2}} \qquad \textrm{for}\,\, \theta \leq v \leq 2\theta.
 \ee
Hence, by the change of variables formula we infer that
\be
\frac{1}{|\tau|^{1/2}}\int_{\theta}^{2\theta} (W^2 e^{\mu})(v,\varphi,\tau) \, dv \sim_\theta \int_{Y_1(\theta,\varphi,\tau)}^{Y_1(2\theta,\varphi,\tau)} w^2(y,\varphi,\tau) e^{-y^2/4} dy .
\ee
Remembering \eqref{def_norm_tip_r}, this shows that
\begin{equation}\label{eq_equiv_of_norms}
C^{-1}\big\| W 1_{ \{ \theta \leq v \leq 2\theta\} } \big\|_{2,\infty}\leq \big\| w\, 1_{\{\theta\leq v_1(\tau)\leq 2\theta\}} \big\|_{\mathcal{H},\infty}\leq  C\big\| W 1_{ \{ \theta \leq v \leq 2\theta\} }\big\|_{2,\infty}\, .
\end{equation}

Now, Proposition \ref{prop-tip} (energy estimate in tip region) combined with \eqref{eq_equiv_of_norms} yields
  \begin{equation}
    \label{eqn-tip_rr}
    \big\| W_\cT \big\|_{2,\infty} \leq  C\eps \big\| w \, 1_{ \{ \theta \leq v_1 \leq 2\theta\} }  \big\|_{\mathcal{H},\infty} \leq C\eps  \big\| w_{\cC}  \big\|_{\mathcal{H},\infty},
  \end{equation}
where in the last step we used that $w(y,\varphi,\tau)=w_{\cC}(y,\varphi,\tau)$ for $v_1(y,\varphi,\tau)\geq \theta$, provided that $\kappa$ is small enough and $\tau_\ast$ is negative enough.

Similarly, by \eqref{eq_equiv_of_norms} applied with $\theta/2$ instead of $\theta$, we have
\be\label{w_trans_est}
 \big\| w \, 1_{ \{ \theta/2 \leq v_1 \leq \theta\} } \big\|_{\mathcal{H},\infty}  \leq C\big\| W_\cT \big\|_{2,\infty}\ ,
\ee
where we also used that $W(v,\varphi,\tau)=W_{\cT}(v,\varphi,\tau)$ for $v\leq \theta$, and so Proposition \ref{prop-cyl-est} (energy estimate in cylindrical region) gives
\begin{equation}
\label{eqn-cylindrical1_r}
 \big\Vert w_\cC - \mathfrak{p}_0w_\cC \big \Vert_{\hD ,\infty } \le \eps    \big\Vert  w_\cC \big\Vert_{\hD ,\infty} +C\eps  \big\Vert W_{\cT}  \big\Vert_{2,\infty}  .
\end{equation}

Finally, by the triangle inequality we clearly have
\begin{equation}
 \big\|w_\cC  \big\|_{\mathcal{D},\infty}\leq  \big\|w_\cC-\mathfrak{p}_0 w_\cC  \big\|_{\mathcal{D},\infty}+ \big\|\mathfrak{p}_0 w_\cC  \big\|_{\mathcal{D},\infty}\, .
\end{equation}

Combining the above inequalities, and adjusting $\varepsilon$, the assertion follows.
  \end{proof}

We also need the following standard derivative estimates:
\begin{lemma}[derivative estimates]\label{std_pde_der_est}
For any $L<\infty$ and $\theta>0$, there exist $\kappa>0$, $\tau_\ast>-\infty$ and $C<\infty$ with the following significance. If  $\mathcal{M}$ is $\kappa$-quadratic at time $\tau_0 \le \tau_*$, then for all $\tau\leq\tau_0$ we have
\be\label{very_good_grad}
\sup_{y\leq L} \left(|Dv(\tau)|+|D^2v(\tau)|\right)\leq \frac{C}{|\tau|},
\ee
and
\be\label{good_grad}
\sup_{v(\tau)\geq \theta/2} \left(|Dv(\tau)|+|D^2v(\tau)|\right)\leq \frac{C}{|\tau|^{1/2}}.
\ee
\end{lemma}
\begin{proof}
By Theorem \ref{strong_uniform0} (uniform sharp asymptotics) and convexity we get
\be\label{proof_first_der_bd}
|\tau|\sup_{y\leq 2L} |Dv(\tau)| + |\tau|^{1/2} \sup_{v(\tau)\geq \theta/4} |Dv(\tau)| \leq C.
\ee
In particular, by standard interior estimates we easily get \eqref{very_good_grad}. To prove the second derivative estimate in \eqref{good_grad}, given any $({\bf y}_{0}, \tau_{0})$ such that $v({\bf y}_{0}, \tau_{0})\geq \theta/2$, we consider the corresponding unrescaled point $({\bf x}_{0}, t_{0})$. Recall that the unrescaled profile function $V=V({\bf x},t)$, viewed as an extrinsic quantity, evolves by
\begin{equation}
\partial_t V = \left(\delta_{ij}-\frac{V_i V_j}{1+|DV|^2}\right)V_{ij} - \frac{1}{V}.
\end{equation}
Now, in the parabolic ball $P({\bf x}_{0}, t_{0},\sqrt{-t_0})$ we have
\be
C^{-1}\leq\frac{V}{\sqrt{-t_0}}\leq C,\qquad |DV| \leq \frac{C}{\sqrt{\log(-t_0)}}.
\ee
Hence, by standard interior estimates for perturbations of the heat equation, taking also into account that we have good control for the differentiated coefficients thanks to Corollary \ref{lemma-cylindrical} (cylindrical estimate), we obtain
\be
\sqrt{-t_0}|D^2V|({\bf x}_{0}, t_{0}) \leq \frac{C}{\sqrt{\log(-t_0)}}.
\ee
Transforming back to our initial variables, this concludes the proof.
\end{proof}

We can now conclude the proof of our spectral uniqueness theorem:

\begin{proof}[Proof of Theorem \ref{thm-spectral_uniqueness}]
In light of Proposition \ref{prop-cor-main} (coercivity estimate) our task boils down to controlling the spectral coefficients $a_i(\tau)$ defined by
\be
\mathfrak{p}_0w_\cC(\tau)=\sum_{i=0}^2 a_i(\tau) \psi_i\, ,
\ee
where
\be
\psi_0=y^2-4,\qquad \psi_1= y^2\cos(2\varphi),\qquad \psi_2=y^2\sin(2\varphi).
\ee
To this end, recall from Lemma \ref{lemma-evolution-wC} (evolution of $w_\cC$) that
\begin{align}
(\partial_\tau -\mathcal{L}) w_\cC=\mathcal{E}[w_\cC]+\overline{\mathcal{E}}[w,\chi_\cC(v_1)]+J+K,
\end{align}
where $J$ and $K$ are defined in \eqref{eqn-IJK}. Since $\mathcal{L}\psi_i=0$, this implies
\begin{equation}\label{ODE_a_i}
\frac{d}{d\tau}a_i(\tau)=\left\langle \mathcal{E}[w_\cC]+\overline{\mathcal{E}}[w,\chi_\cC(v_1)]+J+K,\frac{\psi_i}{\|\psi_i\|^2}\right\rangle\, .
\end{equation}
To proceed, note that
\be \frac{\langle \psi_0,\psi_0^2\rangle}{\| \psi_0\|^2} =\frac{\langle \psi_0,\psi_1^2\rangle}{\| \psi_1\|^2}=\frac{\langle \psi_0 , \psi_2^2 \rangle }{\Vert \psi_2\Vert^2 }= 8,
\ee
and
\be
\langle \psi_0, \psi_0\psi_1\rangle = \langle \psi_0, \psi_0\psi_2\rangle =\langle \psi_0, \psi_1\psi_2\rangle =0.
\ee
We can thus rewrite \eqref{ODE_a_i} in the form
\be
  \frac{d}{d\tau}a_i(\tau) = \frac{2a_i(\tau)}{|\tau|} + F_i(\tau),
\ee
where
  \begin{multline}\label{eq-F-tau}
      F_i =  \left\langle \cE[w_\mathcal{C}] -  \frac{\psi_0}{4|\tau |}\mathfrak{p}_0w_\cC,  \frac{\psi_i}{\|\psi_i\|^2}\right\rangle\\
      +\left\langle\bar{\cE}[w,\varphi_\mathcal{C}(v_1)] + J + K, \frac{\psi_i}{\|\psi_i\|^2}\right \rangle\, .
  \end{multline}
Solving these ODEs, taking into account the fact that $a_i(\tau_0)=0$ thanks to the spectral assumption \eqref{cent_0}, we get
\begin{equation}\label{eq_solution_a}
a_i(\tau)=-\frac{1}{\tau^2}\int_{\tau}^{\tau_0} F_i(\sigma)\sigma^2\, d\sigma.
\end{equation}
In the following, we use the notation
\begin{equation}
A(\tau):= \sup_{\tau'\leq \tau}\left( \int_{\tau'-1}^{\tau'} \sum_{i=0}^2 a_i(\sigma)^2d\sigma \right)^{1/2}\, .
\end{equation}

\begin{claim}[decay estimate]\label{claim_decay_est}
For every $\varepsilon>0$, there exist $\kappa>0$ and ${\tau}_{\ast}>-\infty$, such that assuming $\kappa$-quadraticity at $\tau_0\leq \tau_\ast$, for $\tau\leq\tau_0$ we have
\begin{equation}
\int_{\tau-1}^\tau |F_i(\sigma)|\, d\sigma \leq \frac{\varepsilon}{|\tau|} A(\tau_0).
\end{equation}
\end{claim}

\begin{proof}[Proof of the claim] By Proposition \ref{prop-cor-main} (coercivity estimate) we have
\be\label{cons_from_coerc}
\|w_\cC-\mathfrak{p}_0 w_\cC \|_{\mathcal{D},\infty}+\|W_\cT\|_{2,\infty} \leq  \eps CA(\tau_0).
\ee

Now, fix a smooth cutoff function $\bar{\chi}:\mathbb{R}^+\to [0,1]$ such that $\bar{\chi}(v)=1$ if $v\in [\tfrac{9}{16}\theta,\tfrac{15}{16}\theta]$ and $\bar{\chi}(v)=0$ if $v\notin [\tfrac{\theta}{2},\theta]$. Then, since $\textrm{spt}(\chi_\cC')\subseteq[\tfrac{5}{8}\theta,\tfrac{7}{8}\theta]$ using the estimates     
 \eqref{eq-I2}, \eqref{eq-J} and \eqref{eq-K} we infer that
    \begin{align}\label{eq-barE-est}
        \big|\langle \bar{\cE}[w(\tau),\chi_\mathcal{C}&(v_1(\tau))] + J(\tau) + K(\tau) , \psi_i\rangle\big| \nonumber\\
&\leq     \big\|\bar{\cE}[w(\tau),\chi_\mathcal{C}(v_1(\tau))] + J(\tau)  + K(\tau) \big\|_{\hD^*} \big\|\psi_i \, \bar {\chi}(v_1) \big\|_{\hD} \nonumber\\
        &\leq \eps \, \big\| w(\tau)  1_{\{\theta/2\leq v_1(\tau)\leq \theta\} }\big\|_{\mathcal{H}} \, e^{-|\tau|/4}.
\end{align}
Together with \eqref{w_trans_est} and \eqref{cons_from_coerc} this yields
\be
\int_{\tau-1}^\tau\left| \langle\bar{\cE}[w,\varphi_\mathcal{C}(v_1)] + J + K, \psi_i \rangle\right |\, \leq \frac{\eps}{|\tau|}A(\tau_0).
\ee

\medskip

Next, by \eqref{eqn-E15} we have
\be
 \cE[w_\cC]-  \frac{\psi_0}{4|\tau |}\mathfrak{p}_0w_\cC =  \mathcal{D}[w_\cC] + \frac{2-v_1v_2}{2v_1v_2}w_\cC-  \frac{\psi_0}{4|\tau |}\mathfrak{p}_0w_\cC,
\ee
where
\begin{multline}
\mathcal{D}[w_\cC] =  - \frac{Dv_1\otimes Dv_1\! :\! D^2 w_\cC}{1+|Dv_1|^2 } - \frac{D^2v_2\! :\! D(v_1+v_2)\!\otimes\! Dw_\cC}{(1+|Dv_1|^2)} \\
+\frac{D^2v_2\! :\! Dv_2\otimes Dv_2\, D(v_1+v_2)\!\cdot\! Dw_\cC}{(1+|Dv_1|^2)(1+|Dv_2|^2)}.
 \end{multline}
 
Now, using Lemma \ref{std_pde_der_est} (derivative estimates) we see that
\be
\big| \langle \mathcal{D}[w_\cC(\tau)] ,\psi_i\rangle\big| \leq \frac{\eps}{|\tau|} \| w_\cC(\tau) \|_{\mathcal{D}}.
\ee
Indeed, when computing  $\langle \mathcal{D}[w_\cC(\tau)], \psi_i \rangle$, the contribution to the integrals from the region $y\leq L$ decays quadratically in $|\tau|$ thanks to \eqref{very_good_grad}, and the contribution from the region $y\geq L$ can be bounded by $\eps(L)/|\tau|$ thanks to \eqref{good_grad} and the Gaussian weight. Hence,
\be
\int_{\tau-1}^\tau \big| \langle \mathcal{D}[w_\cC], \psi_i \rangle \big| \leq \frac{C\eps}{|\tau|}A(\tau_0).
\ee

\medskip

Next, we estimate
\begin{align}\label{est_third_term2_F}
&  \left|\left\langle \frac{2-v_1v_2}{2v_1v_2}w_\cC -\frac{\psi_0}{4|\tau|}\mathfrak{p}_0w_\cC , \psi_i \right\rangle \right|\leq 
 \left|\left\langle \frac{2-v_1v_2}{2v_1v_2}\chi_{\cC}(v_1) \Big(w_\cC-\sum_{j=0}^2a_j\psi_j\Big),\psi_i\right\rangle \right| \\
&+\sum_{j=0}^2 |a_j|\left|\left\langle\frac{2-v_1v_2}{2v_1v_2} \chi_{\cC}(v_1)\psi_j-\frac{\psi_0}{4|\tau|}\psi_j,\psi_i\right\rangle \right|
+\left|\left\langle \frac{2-v_1v_2}{2v_1v_2}(1-\chi_{\cC}(v_1))w_\cC,\psi_i\right\rangle \right|\, . \nonumber
\end{align}
To proceed, note that since the $v_i$ are $\kappa$-quadratic at time $\tau_0$, we have
\begin{equation}\label{diff_norm_small}
\big\|\sqrt{2}-v_i\big\|_{\mathcal{H}}\leq \frac{C}{|\tau|}
\end{equation}
Via the energy method as in the proof of Proposition \ref{ODE system} this implies
\begin{equation}
\big\| D v_{i} 1_{\{v_i\geq \theta/2\}} \big\|_{\mathcal{H}}\leq \frac{C}{|\tau|}\, .
\end{equation}
Remembering also \eqref{eq_wiggle_room} this yields
\begin{equation}
\big\|(\sqrt{2}-v_i)\chi_{\cC}(v_1)\big\|_{\cD} \leq \frac{C}{|\tau|}.
\end{equation}
Hence, arguing similarly as in \cite[Proof of Claim 5.26]{CHH_translator} we get
\be
\int_{\tau-1}^\tau \left|\left\langle \frac{2-v_1v_2}{2v_1v_2}\chi_{\cC}(v_1) \Big(w_\cC-\sum_{j=0}^2a_j\psi_j\Big),\psi_i\right\rangle \right| \leq \frac{\eps}{|\tau|} A(\tau_0).
\ee
and arguing similarly as as in \cite[Proof of Claim 8.3]{ADS2} we get
\be
\int_{\tau-1}^\tau \sum_{j=0}^2 |a_j|\left|\left\langle\frac{2-v_1v_2}{2v_1v_2} \chi_{\cC}(v_1)-\frac{\psi_0}{4|\tau|},\psi_i\psi_j\right\rangle \right| \leq \frac{\eps}{|\tau|} A(\tau_0).
\ee
Finally, since $(1-\chi_{\cC}(v_1))w_{\cC}$ is supported in the region where the Gaussian weight is exponentially small, we easily get
\begin{align}
\int_{\tau-1}^\tau \left|\left\langle \frac{2-v_1v_2}{2v_1v_2}(1-\chi_{\cC}(v_1))w_\cC,\psi_i\right\rangle \right| 
\leq \frac{\eps}{|\tau|} A(\tau_0).
\end{align}
This concludes the proof of the claim.
\end{proof}

Continuing the proof of the theorem, using Claim \ref{claim_decay_est} (decay estimate) we can now estimate 
\begin{multline}
\left| \int^{\tau_0}_\tau F_i(\sigma)\sigma^2d\sigma\right|\leq \sum^{\lceil\tau_0\rceil}_{k=\lfloor\tau\rfloor}\int^k_{k-1}|F_i(\sigma)|\sigma^2 d\sigma \\
 \leq \sum^{\lceil\tau_0\rceil}_{k=\lfloor\tau\rfloor} \frac{(|k|+1)^2}{|k|} \eps A(\tau_0)\leq |\tau|^2 \eps A(\tau_0) \, .
\end{multline}
Remembering \eqref{eq_solution_a}, this shows that for all $\tau\leq \tau_0$ we have
\begin{equation}\label{a_small_always}
|a_i(\tau)|\leq \eps A(\tau_0).
\end{equation}
Choosing $\varepsilon =1/10$ this implies
\begin{equation}
A(\tau_0)=0.
\end{equation}
Together with Proposition \ref{prop-cor-main} (coercivity estimate) we conclude that
\begin{equation}
\|w_\cC\|_{\cD,\infty} +\| W_\cT \|_{2,\infty} =0\, ,
\end{equation}
hence
\be
\mathcal{M}^1=\mathcal{M}^2.
\ee
This finishes the proof of the theorem.
\end{proof}

\bigskip

\section{From spectral uniqueness to classification}\label{classification}

In this section, we conclude the proof of our main classification theorem. 
We refer the reader to Subsection \ref{sec-outline} (Outline and intermediate results) for an overview of our strategy  in this section,  including some notation and terminology.

\subsection{Existence with prescribed spectral width ratio} 

In this subsection, we prove that the class $\mathcal{A}^\circ$ of $\mathbb{Z}_2^2\times\mathrm{O}(2)$-symmetric ovals from \cite{DH_ovals}, whose definition we recalled in \eqref{Ao}, realizes all spectral width ratios.

Recall that thanks to the $\mathbb{Z}_2^2$-symmetry the profile function of any $\mathcal{M}\in \mathcal{A}^{\circ}$ is orthogonal to the eigenfunctions $y\cos\vp,y\sin\vp$ and $y^2\sin(2\vp)$. Our first goal concerns orthogonality relations with respect to the eigenfunctions
\begin{equation}
\psi_1=1,\qquad \psi_2=y^2-4.
\end{equation}
To this end, given $\mathcal{M}=\{M_t\}$, and parameters $\beta$ and $\gamma$, we consider the transformed flow
\begin{equation}
\mathcal{M}^{\beta,\gamma}=\{ e^{\gamma/2} M_{e^{-\gamma}(t-\beta)}\}.
\end{equation}
For convenience we set
\begin{equation}\label{bgamma}
    b=\sqrt{1+\beta e^{\tau}}-1, \qquad \Gamma=\frac{\gamma-\ln (1+\beta e^{\tau})}{\tau}.
\end{equation}
Then, the renormalized profile functions of $\mathcal{M}^{\beta,\gamma}$ and $\mathcal{M}$ are related by
\begin{equation}\label{bgammatransform}
    v^{b \Gamma}({\bf y},\tau)=(1+b)\, v\left(\frac{{\bf y}}{1+b}, (1+\Gamma)\tau\right).
\end{equation}
Our first goal is to find a canonical zero of the map 
\begin{equation}\label{Psi=00}
\Psi_\tau(b, \Gamma)=\left(  \Big\langle  \psi_1 ,v^{b\Gamma}_\cC-\sqrt{2} \Big\rangle_\cH , 
\Big\langle  \psi_2, v^{b\Gamma}_\cC +\frac{\psi_2}{\sqrt{8}|\tau|}\Big\rangle_\cH \right).
\end{equation}
The key towards finding such a canonical zero is Proposition \ref{JPhiestimates_intro} (Jacobian estimate), which we restate here for convenience of the reader:

\begin{proposition}[Jacobian estimate]\label{JPhiestimates} There exist $\kappa>0$ and $\tau_\ast>-\infty$ with the following significance. If $\mathcal{M}$ is $\kappa$-quadratic at time $\tau_0\leq \tau_\ast$, then 
\begin{equation}
    \mathrm{det}(J\Psi_\tau(b, \Gamma))>0
\end{equation}
holds for all $\tau\leq\tau_0$ and all $(b,\Gamma)$ with $|\tau|^2b^2+\Gamma^2\leq 100 \kappa^2$.
\end{proposition}

\begin{proof} 
Throughout this proof, we write $f=O(g)$ if there exists some constant $C<\infty$ independent of $\tau\leq\tau_0$ and $(b,\Gamma)$, with $|\tau|^2b^2+\Gamma^2\leq 100 \kappa^2$, such that $|f|\leq Cg$.\footnote{The constraint $|\tau|^2b^2+\Gamma^2\leq 100 \kappa^2$ is motivated by the degree argument below.} 
Let us also recall that our cutoff function $\chi_\cC$ satisfies
\be
 \chi_\mathcal{C}(v)=0 \,\,\, \textrm{for}\, \,\, v\leq \tfrac58 \theta  \qquad \mbox{and} \qquad    \chi_\mathcal{C}(v)=1 \,\,\, \textrm{for}\, \, \, v\geq \tfrac78 \theta,
 \ee
and that thanks to Theorem \ref{strong_uniform0} (uniform sharp asymptotics) and convexity for $\theta$ small enough we have the Gaussian tail estimate
\begin{equation}\label{gaussian_tail}
e^{-\frac{y^2}{4}} 1_{\{ v^{b\Gamma}\leq \theta\}}\leq e^{-\frac{|\tau|}{3}},
\end{equation}
and the gradient estimate
\begin{equation}\label{grad_est_Sec5}
| D v^{b\Gamma}|  1_{\{v^{b\Gamma}\geq \theta/2\}}=O(|\tau|^{-1/2}).
\end{equation}
Now, to prove the proposition, we have to estimate the derivatives with respect to the parameters $b$ and $\Gamma$. We start with the former:

\begin{claim}[$b$-derivatives]\label{claim_b_der}
For the derivatives with respect to $b$ we have
\begin{equation}\label{1b}
\left\langle \psi_1,\partial_b v_{\cC}^{b\Gamma}\right\rangle_\cH = \sqrt{2}\|\psi_1\|^2+O(|\tau|^{-1})\,\,\,\textrm{and}\,\, \, \left\langle \psi_2,\partial_b v_{\cC}^{b\Gamma}\right\rangle_\cH = O(|\tau|^{-1}).
\end{equation}
\end{claim}

\begin{proof}[{Proof of Claim \ref{claim_b_der}}]
First, by the transformation formula  \eqref{bgammatransform} we have
\begin{equation}\label{partialb}
   \partial_b v^{b \Gamma}({\bf y},\tau)= v\left(\frac{{\bf y}}{1+b}, (1+\Gamma)\tau\right)-\frac{{\bf y}}{1+b}\cdot D v\left(\frac{{\bf y}}{1+b}, (1+\Gamma)\tau\right).
\end{equation}
Via integration by parts this implies 
\begin{multline}
   \left\langle \psi_i,\chi_{\cC}(v^{b\Gamma})\partial_b v^{b\Gamma}\right\rangle_\cH=\left\langle 
  \left(\psi_i+ \frac{2{\bf y}\cdot D \psi_i+(4-y^2)\psi_i}{2(1+b)}\right)\chi_{\cC}(v^{b\Gamma}),v\right\rangle_\cH\\
  +\left\langle \psi_i\frac{{\bf y}\cdot D v^{b\Gamma}}{1+b}\chi_{\cC}'(v^{b\Gamma}), v\right\rangle_\cH,
\end{multline}
where the function $v$ is evaluated at $\big(\frac{{\bf y}}{1+b}, (1+\Gamma)\tau\big)$.

Using the basic facts recalled above we see that
\begin{equation}
\left\langle 
  \left(\psi_i+ \frac{2{\bf y}\cdot D \psi_i+(4-y^2)\psi_i}{2(1+b)}\right)\left(1-\chi_{\cC}(v^{b\Gamma})\right), v\right\rangle_\cH= O(|\tau|^{-10}),
\end{equation}
and
\begin{equation}
\left\langle \psi_i \frac{{\bf y}\cdot D v^{b\Gamma}}{1+b}\chi_{\cC}'(v^{b\Gamma}), v\right\rangle_\cH=O(|\tau|^{-10}).
\end{equation} 

Next, by the $\kappa$-quadraticity assumption we have the $\mathcal{H}$-norm expansion
\begin{align}\label{ubg_expansion}
    v\left(\frac{{\bf y}}{1+b}, (1+\Gamma)\tau\right)
   =\sqrt{2}-\frac{y^2-4}{\sqrt{8}(1+\Gamma)|\tau|}+O\left(\frac{\kappa}{|\tau|}\right).
\end{align}
This yields
\begin{equation}
\left\langle 
  \left(\psi_1+ \frac{2{\bf y}\cdot D \psi_1+(4-y^2)\psi_1}{2(1+b)}\right),v\right\rangle_\cH= \sqrt{2}\|\psi_1\|^2+O(|\tau|^{-1}),
\end{equation}
and, taking into account the identity $\langle  \psi_2^2 -16,1\rangle_\cH = 0$, also yields
\begin{equation}
\left\langle 
  \left(\psi_2+ \frac{2{\bf y}\cdot D \psi_2+(4-y^2)\psi_2}{2(1+b)}\right),v\right\rangle_\cH=O(|\tau|^{-1}).
\end{equation}

Finally, observe that
\begin{equation}
\left\langle \psi_i,\partial_b v_{\cC}^{b\Gamma}\right\rangle_\cH=\left\langle \psi_i,\chi_{\cC}(v^{b\Gamma})\partial_b v^{b\Gamma}\right\rangle_\cH+\left\langle \psi_i,v^{b\Gamma}\chi_{\cC}'(v^{b\Gamma})\partial_b v^{b\Gamma}\right\rangle_\cH,
\end{equation}
and
\begin{equation}
\left\langle \psi_i,v^{b\Gamma}\chi_{\cC}'(v^{b\Gamma})\partial_b v^{b\Gamma}\right\rangle_\cH=O(|\tau|^{-10}).
\end{equation} 
Combining the above estimates the claim follows.
\end{proof}

\begin{claim}[$\Gamma$-derivatives]\label{claim_gamma_der} For the derivatives with respect to $\Gamma$ we have
\begin{equation}\label{2g}
     \left \langle \psi_1,  \partial_{\Gamma} v_{\cC}^{b \Gamma} \right\rangle_\cH =O(\kappa)\quad\textrm{and}\quad
     \left \langle \psi_2,  \partial_{\Gamma} v_{\cC}^{b \Gamma} \right\rangle_\cH =\frac{\|\psi_{2}\|^2}{\sqrt{8}|\tau|}+O\left(\frac{\kappa}{|\tau|}\right).
 \end{equation}
\end{claim}

\begin{proof}[{Proof of Claim \ref{claim_gamma_der}}]
By the transformation formula  \eqref{bgammatransform} we have
\begin{equation}\label{partialtau}
    \partial_\Gamma v^{b \Gamma}({\bf y},\tau)=(1+b)\tau v_\tau \left(\frac{{\bf y}}{1+b}, (1+\Gamma)\tau\right).
\end{equation}
Recall that the renormalized profile function evolves by
\begin{equation}
v_\tau=\mathcal{L}v - \frac{ D^2v\!:\! (D v\otimes D v)}{1+|D v|^2} -\frac{v}{2}-\frac{1}{v}.
\end{equation}
Together with integration by parts, this implies
\begin{multline}
\frac{\left\langle \psi_i,  \chi_{\cC}(v^{b \Gamma}) \partial_{\Gamma} v^{b \Gamma} \right\rangle_\cH}{(1+b)\tau}
     =  \left\langle \mathcal{L}\left(  \chi_{\cC}(v^{b \Gamma}) \psi_i\right), v\right\rangle_\cH\\
     -\left\langle \psi_i ,\chi_{\cC}(v^{b \Gamma})\left(\frac{v}{2}+\frac{1}{v}\right)\right\rangle_\cH 
      - \left\langle \psi_i,\chi_{\cC}(v^{b \Gamma}) \frac{  D^2v\!:\! (D v\otimes D v)}{1+| D v|^2}  \right\rangle_\cH ,
\end{multline}
where the function $v$ is evaluated at $\big(\frac{{\bf y}}{1+b}, (1+\Gamma)\tau\big)$ as usual.

We  can rewrite the first term using the product rule in the form
\begin{equation}
\mathcal{L}\left(  \chi_{\cC}(v^{b \Gamma}) \psi_i\right)=\chi_{\cC}(v^{b \Gamma}) \mathcal{L}   \psi_i+\psi_i(\mathcal{L}-1)\chi_{\cC}(v^{b \Gamma})+2 D \chi_{\cC}(v^{b \Gamma}) D \psi_i.
\end{equation}
Arguing as above, and using also that $\mathcal{L}\psi_i=\delta_{1i}$, this yields
\begin{equation}
 \left\langle \mathcal{L}\left( \chi_{\cC}(v^{b \Gamma}) \psi_i\right) , v\right\rangle_\cH= \left\langle  \chi_{\cC}(v^{b \Gamma})\delta_{1i} , v\right\rangle_\cH+O(|\tau|^{-10}).
\end{equation}

\bigskip
Next, we have to deal with the term
\begin{equation}
\left\langle \chi_{\cC}(v^{b \Gamma}), v\right\rangle_\cH-\left\langle \psi_1 ,\chi_{\cC}(v^{b \Gamma})\left(\frac{v}{2}+\frac{1}{v}\right)\right\rangle_\cH=\left\langle \chi_{\cC}(v^{b \Gamma}),\frac{v}{2}-\frac{1}{v} \right\rangle_\cH.
\end{equation}
To estimate this, we start with the algebraic identity
\begin{equation}
\frac{v}{2}-\frac{1}{v}=v-\sqrt{2}-\frac{(v-\sqrt{2})^2}{2v}.
\end{equation}
Using the tail bound \eqref{gaussian_tail} we see that
\begin{equation}
\left\langle \chi_{\cC}(v^{b \Gamma})-1,v-\sqrt{2}\right\rangle_\cH=O(|\tau|^{-10}),
\end{equation}
and using $\langle 1,\psi_2\rangle_\cH = 0$ and \eqref{ubg_expansion} we can estimate
\begin{equation}
\left\langle 1,v-\sqrt{2}\right\rangle_\cH=\left\langle 1,v-\sqrt{2}+\frac{\psi_2}{\sqrt{8}|\tau|}\right\rangle_\cH=O\left(\frac{\kappa}{|\tau|}\right).
\end{equation}
Moreover, remembering that $\chi_\cC(v)=0$ for $v\leq \tfrac58\theta$ we infer that
\begin{equation}
\left|\left\langle \chi_{\cC}(v^{b \Gamma}),\frac{(v-\sqrt{2})^2}{2v} \right\rangle_\cH\right|\leq \frac{2}{\theta}\left\| v-\sqrt{2}\right\|^2_\cH =O(|\tau|^{-2}).
\end{equation}
Combining the above observations shows that
\begin{equation}
\left\langle \chi_{\cC}(v^{b \Gamma}) , v\right\rangle_\cH-\left\langle \psi_1 ,\chi_{\cC}(v^{b \Gamma})\left(\frac{v}{2}+\frac{1}{v}\right)\right\rangle_\cH=O\left(\frac{\kappa}{|\tau|}\right).
\end{equation}

\medskip

Next, we have to deal with the term
\begin{equation}
-\Big\langle \psi_2,\chi_{\cC}(v^{b \Gamma})\left(\frac{v}{2}+\frac{1}{v}\right)\Big\rangle_\cH.
\end{equation}
To estimate this, we start with the algebraic identity
\begin{equation}
\frac{v}{2}+\frac{1}{v}=\sqrt{2}+ \frac{u^2}{\sqrt{8}}-\frac{u^3}{\sqrt{8}v},\quad \textrm{where} \,\, u=v-\sqrt{2}.
\end{equation}
Using in particular the tail bound \eqref{gaussian_tail} we see that
\begin{equation}
\Big\langle \psi_2 \chi_{\cC}(v^{b \Gamma}),\sqrt{2}\Big\rangle_\cH=O(|\tau|^{-10}),
\end{equation}
and
\begin{equation}
\Big\langle \psi_2 \left(\chi_{\cC}(v^{b \Gamma})-1\right),u^2\Big\rangle_\cH=O(|\tau|^{-10}).
\end{equation}
Also observe that 
\begin{equation}\label{blabla1}
\left\langle  1_{\{  y\leq 4\textrm{ or } y\geq |\tau|^{1/100} \}} ,\psi_2\left( u^2-\frac{\psi_2^2}{8|\tau|^2}\right)\right\rangle_\cH=O\left(\frac{\kappa}{|\tau|^2}\right).
\end{equation}
On the other hand, abbreviating $\chi:=1_{\{ 4\leq y\leq |\tau|^{1/100} \}}$, we can estimate
\begin{multline}\label{blabla2}
\left|\left\langle \chi  ,\psi_2\left( u^2-\frac{\psi_2^2}{8|\tau|^2}\right)\right\rangle_\cH\right|\\
\leq \left\|\chi \psi_2^{1/2}\left(u-\frac{\psi_2}{\sqrt{8}|\tau|}\right)\right\|_\cH \cdot  \left\|\chi \psi_2^{1/2}\left(u+\frac{\psi_2}{\sqrt{8}|\tau|}\right)\right\|_\cH .
\end{multline}
To proceed, we set
\begin{equation}
w=\hat{u}+\frac{\psi_2}{\sqrt{8}|\tau|},
\end{equation}
where $\hat{u}$ is truncated via the graphical radius $\rho(\tau)=|\tau|^{1/100}$.
Note that
\begin{equation}
\|  w \|_\cH =O\left(\frac{\kappa}{|\tau|}\right) \quad \textrm{and}\quad
\| (\partial_\tau-\mathcal{L}) w \|_\cH =O(|\tau|^{-101/100}).
\end{equation}
Hence, arguing similarly as in the proof of Proposition \ref{ODE system} we infer that
\begin{equation}\label{der_bound_eq}
\|Dw \|_\cH =O\left(\frac{\kappa}{|\tau|}\right).
\end{equation}
Using also the weighted Poincar\'e inequality \eqref{easy_Poincare} this implies
\begin{equation}
\left\| \chi\psi_2^{1/2}\left(u+\frac{\psi_2}{\sqrt{8}|\tau|}\right)\right\|_\cH=O\left(\frac{\kappa}{|\tau|}\right),
\end{equation}
and  in particular via the triangle inequality this also implies
\begin{equation}
\left\| \chi\psi_2^{1/2}\left(u-\frac{\psi_2}{\sqrt{8}|\tau|}\right)\right\|_\cH=O(|\tau|^{-1}).
\end{equation}
Together with \eqref{blabla1} and \eqref{blabla2}, and with $\langle \psi_2,\psi_2^2\rangle_\cH = 8 \|\psi_2\|^2$, this yields
\begin{equation}
\Big\langle \psi_2 , u^2\Big\rangle_\cH=\frac{\|\psi_{2}\|^2}{|\tau|^2}+O\left(\frac{\kappa}{|\tau|^2}\right).
\end{equation}
Moreover, using the graphical radius condition from \eqref{condition2} we see that
\begin{equation}
\left| \Big\langle \psi_21_{\{y\leq |\tau|^{1/100}\}},\frac{u^3}{v} \Big\rangle_\cH \right|\leq C \langle |\psi_2|,u^2\rangle_\cH \max_{y\leq |\tau|^{1/100}} |u|=O\left(|\tau|^{-\frac{101}{50}}\right).
\end{equation}
Note also that we have the tail estimate
\begin{equation}
\Big\langle \psi_2\chi_{\cC}(v^{b \Gamma})1_{\{y\geq |\tau|^{1/100}\}},\frac{u^3}{v} \Big\rangle_\cH =O(|\tau|^{-10}).
\end{equation}
Combining the above observations shows that
\be
-\left\langle \psi_2 ,\chi_{\cC}(v^{b \Gamma})\left(\frac{v}{2}+\frac{1}{v}\right)\right\rangle_\cH
=-\frac{\|\psi_{2}\|^2}{\sqrt{8}|\tau|^2}+O\left(\frac{\kappa}{|\tau|^2}\right).
\ee

\medskip

Moreover, using by Lemma \ref{std_pde_der_est} (derivative estimates) we can estimate
\begin{equation}
 \left\langle 1_{\{y\leq |\tau|^{1/100}\}}, |\psi_i| |D^2v| |Dv|^2 \right\rangle_\cH =O\left(|\tau|^{\frac{1}{50}-\frac{3}{2}}\right).
\end{equation}
and
\begin{equation}
 \left\langle 1_{\{y\geq |\tau|^{1/100}\}},\psi_i \chi_{\cC}(v^{b \Gamma}) \frac{  D^2v\!:\! (D v\otimes D v)}{1+| D v|^2}  \right\rangle_\cH  = O(|\tau|^{-10}).
\end{equation}
This shows that
\be
\left\langle \psi_i,\chi_{\cC}(v^{b \Gamma}) \frac{  D^2v\!:\! (D v\otimes D v)}{1+| D v|^2}  \right\rangle_\cH=O\left(|\tau|^{-\frac{149}{100}}\right).
\ee

Furthermore, observe that
\begin{equation}
\left\langle \psi_i,\partial_\Gamma v_{\cC}^{b\Gamma}\right\rangle_\cH=\left\langle \psi_i,\chi_{\cC}(v^{b\Gamma})\partial_\Gamma v^{b\Gamma}\right\rangle_\cH+\left\langle \psi_i,v^{b\Gamma}\chi_{\cC}'(v^{b\Gamma})\partial_\Gamma v^{b\Gamma}\right\rangle_\cH,
\end{equation}
and
\begin{equation}
\left\langle \psi_i,v^{b\Gamma}\chi_{\cC}'(v^{b\Gamma})\partial_\Gamma v^{b\Gamma}\right\rangle_\cH=O(|\tau|^{-10}).
\end{equation} 
Combining the above estimates the claim follows.
\end{proof}
 
Finally, combining Claim \ref{claim_b_der} ($b$-derivatives) and Claim \ref{claim_gamma_der} ($\Gamma$-derivatives) we conclude that
\begin{equation}
    \mathrm{det}(J\Psi_\tau)=\frac{\|\psi_{1}\|^2\|\psi_{2}\|^2}{2|\tau|}+O\left(\frac{\kappa}{|\tau|}\right).
\end{equation}
This proves the proposition.
\end{proof}

In the following continuity argument we need a family for which we know a priori that it depends continuously on the parameters. Hence, 
instead of with the ancient ovals themselves we will actually work with sequences of ellipsoidal flows that approximate the elements of the class $\mathcal{A}^\circ$. Specifically, we call a mean curvature flow $\mathcal{M}=\{M_t\}_{t\geq T}$ an \emph{ellipsoidal flow} if its initial condition at some given time $T>-\infty$ is given by an ellipsoid of the form
\begin{equation}
E(a,\ell,R):=\left\{x\in\mathbb{R}^4: \frac{a^2}{\ell^2}x_1^2 + \frac{(1-a)^2}{\ell^2}x_2^2 + x_3^3 + x_4^2 = R^2\right\},
\end{equation}
for some parameters $a\in (0,1)$, $\ell<\infty$ and $R<\infty$. To relate to the notation from the introduction, note that the flow $M^{\ell,a}_t$ from  \eqref{m_ell_a} is an ellipsoidal flow with initial condition $E(a,\ell,\sqrt{2\lambda_{\ell,a}})$ at time $-\lambda^{2}_{\ell,a} t_{\ell,a}$.

\begin{definition}[$\kappa$-quadratic between $\tau_0$ and $2\tau_0$]
An ellipsoidal flow $\mathcal{M}=\{M_t\}_{t\geq T}$  is called \emph{$\kappa$-quadratic between $\tau_0$ and $2\tau_0$}, if $\log(-T)\geq 2|\tau_0|$, and
\begin{equation}\label{condition1_ell}
   \max_{2\tau_0\leq\tau\leq \tau_0}|\tau| \left\|v_{\cC}(y,\varphi, \tau)-\sqrt{2}+\frac{y^2-4}{\sqrt{8}|\tau|}  \right\|_\cH \leq \kappa,
\end{equation}
and 
\begin{equation}\label{condition2_ell}
 \max_{2\tau_0\leq\tau\leq \tau_0} |\tau|^{\frac{1}{50}}   \|v(\cdot ,\tau)-\sqrt{2}\|_{C^{4}(B(0, 2|\tau|^{{1}/{100}}))}\leq 1.
\end{equation}
\end{definition}
The Jacobian estimate at $\tau=\tau_0$ also holds for ellipsoidal flows:

\begin{corollary}[Jacobian estimate for ellipsoidal flows]\label{cor_jac}
There exist $\kappa>0$ and $\tau_\ast>-\infty$ with the following significance. 
If an ellipsoidal flow $\mathcal{M}$ is $\kappa$-quadratic between $\tau_0$ and $2\tau_0$, for some $\tau_0\leq\tau_\ast$, then we have
\begin{equation}
    \mathrm{det}(J\Psi_{\tau_0}(b, \Gamma))>0
\end{equation}
for all $(b,\Gamma)$ with $|\tau_0|^2b^2+\Gamma^2\leq 100 \kappa^2$.
\end{corollary}

\begin{proof} Indeed, this follows by inspecting the above proof.
\end{proof}

Now, given $\kappa>0$ and $T,\tau_0>-\infty$ with  $\log(-T)\geq 2|\tau_0|$, set
\begin{multline}
\mathcal{E}_{\kappa}^T(\tau_0):=\big\{ \mathcal{M} : \textrm{$\mathcal{M}=\{M_t\}_{t\geq T}$ is an ellipsoidal flow }\\
\textrm{ that is $2\kappa$-quadratic between $\tau_0$ and $2\tau_0$}\big\}.
\end{multline}

We equip $\mathcal{E}_{\kappa}^T(\tau_0)$ with the topology induced by the parameterization in terms of $(a,\ell,R)\in(0,1)\times\mathbb{R}_+\times\mathbb{R}_+$. Of course, since the flow is well-posed, if the initial time slices are close, then later time slices are close as well.

\begin{proposition}[transformation map]\label{S_property} There exist $\kappa>0$ and $\tau_\ast>-\infty$ with the following significance.
Given any $\tau_0\leq\tau_\ast$ and any $T\leq - e^{-2\tau_0}$, there exist $\beta=\beta(\mathcal{M})$ and $\gamma=\gamma(\mathcal{M})$ depending continuously on $\mathcal{M}\in \mathcal{E}_{\kappa}^T(\tau_0)$ such that the truncated renormalized profile function $v^{\beta,\gamma}_\cC$ of the transformed flow $\mathcal{M}^{\beta,\gamma}$ satisfies
\begin{equation}\label{orth_cond_prop}
 \Big\langle  1 ,v^{\beta,\gamma}_\cC(\tau_0)-\sqrt{2} \Big\rangle_\cH=0, \quad
\Big\langle  y^2-4, v^{\beta,\gamma}_\cC(\tau_0) +\frac{ y^2-4}{\sqrt{8}|\tau_0|}\Big\rangle_\cH=0.
\end{equation}
\end{proposition}

\begin{proof}
We will combine the above Jacobian estimate with a mapping degree argument. Since $\mathcal{M}$ is  $2\kappa$-quadratic between $\tau_0$ and $2\tau_0$, we have
\begin{equation}
    \left\| v_{\cC}(y,\varphi,  \tau_0)-\sqrt{2}+\frac{y^2-4}{\sqrt{8}|\tau_0|}  \right\|_\cH \leq \frac{2\kappa}{|\tau_0|} .
\end{equation}
Thus, using the transformation formula \eqref{bgammatransform} and standard Gaussian tail estimates, for all $(b, \Gamma)\in [-1/|\tau|,1/|\tau|]\times [-1/2,1/2]$ we get
\begin{equation}
 \left\| v^{b\Gamma}_{\cC}(y,\varphi,\tau_0)-\sqrt{2} - \sqrt{2}b
+\frac{y^2-4}{\sqrt{8}|\tau_0|(1+\Gamma)}\right\|_\cH \leq \frac{10\kappa}{|\tau_0|}.
\end{equation}
Hence, for $\kappa>0$ small enough and $\tau_0\leq\tau_\ast$ negative enough, the map $\Psi$ is homotopic to 
\begin{equation}
(b,\Gamma)\mapsto \left(
     \sqrt{2}\|\psi_1\|^2\,  b,
     \frac{\|\psi_2\|^2}{\sqrt{8}|\tau_0|}\frac{\Gamma}{(1+\Gamma)}
    \right),
\end{equation}
when restricted to the boundary of the disc
\begin{equation}\label{Dk2}
D:=\left\{(b, \Gamma)\in\mathbb{R}^2\;:\; |\tau_0|^2b^2+\Gamma^2\leq 100 \kappa^2  \right\},
\end{equation}
where the homotopy can be chosen through maps avoiding the origin. Hence,
\begin{equation}\label{degree=1}
    \mathrm{deg}(\Psi|_D)=1.
\end{equation}
On the other hand, by Corollary \ref{cor_jac} (Jacobian estimate for ellipsoidal flows), as long as we choose $\kappa>0$ small enough and $\tau_0\leq\tau_\ast$ negative enough, we have
\begin{equation}
\mathrm{det}(J\Psi|_D)>0.
\end{equation}
In particular, $(0, 0)$ is a regular value. Recalling also that degree is given by counting the inverse images according to the sign of their Jacobian, we thus infer that there exists a 
 unique $(b, \Gamma)\in D$ such that $\Psi(b, \Gamma)=0$. By uniqueness, $(b,\Gamma)$, and thus the corresponding $(\beta,\gamma)$, depends continuously on $\mathcal{M}$. This finishes the proof of the proposition.
\end{proof}

Given $\kappa>0$ and $\tau_0>-\infty$, we set
 \begin{multline}
 \mathcal{A}'_{\kappa}(\tau_0):=\Big\{ \mathcal{M} \, :\,  \mathcal{M} \textrm{ is $\kappa$-quadratic at time $\tau_0$, and satisfies}\\
 \big\langle v_{\cC}^{\mathcal{M}}{( \tau_0)}+ \tfrac{y^2-4}{\sqrt{8}|\tau_0|},  y^2-4\big\rangle_\cH=0 , 
 \textrm{ and $\exists \mathcal{N}\in\mathcal{A}^\circ$, $\exists\beta,\gamma$: $\mathcal{M}=\mathcal{N}^{\beta,\gamma}$}\Big\}.
\end{multline}
We remind the reader that by the $\mathbb{Z}_2^2$-symmetry the function $v^{\mathcal{M}}_\cC(\tau_0)$, for $\mathcal{M} \in  \mathcal{A}'_{\kappa}(\tau_0)$  is automatically orthogonal to the eigenfunctions $y\cos\vp,y\sin\vp$ and $y^2\sin(2\vp)$, and that by definition of $\kappa$-quadratic at time $\tau_0$ we in particular have
\be
 \big\langle  v^{\mathcal{M}}_\cC(\tau_0)-\sqrt{2} , 1 \big\rangle_\cH=0.
 \ee
Considering the spectral width ratio map
\begin{equation}
\mathcal{R}:\mathcal{A}'_\kappa(\tau_0) \to \mathbb{R},\qquad
 \mathcal{M}\mapsto\frac{\langle v^{\mathcal{M}}_{\cC}(\tau_{0}), y^2\cos^2\varphi-2 \rangle_\cH}{\langle v^{\mathcal{M}}_{\cC}(\tau_{0}), y^2\sin^2\varphi-2 \rangle_\cH}\, ,
\end{equation}
we can now prove Theorem \ref{prescribed_eccentricity_intro}  (existence with prescribed spectral width ratio), which we restate here in the following technically sharper form:

\begin{theorem}[existence with prescribed spectral width ratio]\label{prescribed_eccentricity_restated} 
There exist constants $\delta>0$, $\kappa>0$ and $\tau_\ast>-\infty$ with the following significance.  For every $\tau_{0}\leq \tau_{*}$ and every $r\in [(1+ \delta |\tau_0|^{-1})^{-1},1+ \delta |\tau_0|^{-1}]$, there exists an $\mathcal{M}\in\mathcal{A}'_\kappa(\tau_0)$ that is a bubble-sheet oval and satisfies
\begin{equation}
    \mathcal{R}(\mathcal{M})=r.
\end{equation}
\end{theorem}

\begin{proof}
Our argument is related to \cite[Proof of Theorem 4.11]{CHH_translator}, but with some modifications, since there is no Rado-type argument available in our setting.
Fix constants $\tau_\ast>-\infty$ and $\kappa>\kappa'>0$ such that Proposition \ref{S_property} (transformation map) and Theorem \ref{point_strong} (strong $\kappa$-quadraticity) apply. 
Possibly after decreasing $\tau_\ast$, given any $\tau_0\leq\tau_\ast$, by the uniqueness result of symmetric ancient ovals from \cite{DH_ovals}, we can assume that the $\mathrm{O}(2)\times\mathrm{O}(2)$-symmetric oval $\mathcal{M}^{\textrm{sym}}$ is $\frac{\kappa'}{100}$-quadratic at time $\tau_0$ and satisfies
\begin{equation}
\Big\langle v_{\cC}^{\mathcal{M}^\textrm{sym}}(\tau_0)+ \tfrac{y^2-4}{\sqrt{8}|\tau_0|},  y^2-4\Big\rangle_\cH=0.
\end{equation}
 Note also that  $\mathcal{R}(\mathcal{M}^{\textrm{sym}})=1$ thanks to the symmetry.
 Let $\delta:=\kappa'/100$. Our goal is to show given any $0<\delta'\leq\delta$ that
 the image of $\mathcal{R}:\mathcal{A}'_\kappa(\tau_0) \to \mathbb{R}$ contains the points $r_\pm=(1+\delta' |\tau_0|^{-1})^{\pm 1}$.

Fixing any sequence $T_i\to -\infty$, denote by $\mathcal{M}_{a,i}$ the ellipsoidal flow whose initial condition at (unrescaled) time $T_i$ is given by the ellipsoid
\begin{equation}
E_{a,i}=E(a,\ell_{a,i},R_{a,i}),
\end{equation}
where $\ell_{a,i},R_{a,i}$ are such that $\mathcal{M}_{a,i}$ becomes extinct at (unrescaled) time $0$ and satisfies
\begin{equation}\label{Huisken_time_minus1}
\int_{(M_{a,i})_{-1}}\frac{1}{(4\pi)^{3/2}}e^{-\frac{|p|^2}{4}}\, dA(p)=\frac{1}{2}\Theta_{S^1\times \mathbb{R}^2}+\frac{1}{2}\Theta_{S^2\times \mathbb{R}}\, ,
\end{equation}
where $\Theta_{S^2\times \mathbb{R}}=4/e< \sqrt{2\pi/e}=\Theta_{S^1\times \mathbb{R}^2}$.
Now, given any $i\gg 1$, consider all $a\in(0,1)$ such that there exist some $\beta_{a,i},\gamma_{a,i}$ such that the transformed flow $\mathcal{M}^{a}_i:=\mathcal{M}_{a,i}^{\beta_{a,i},\gamma_{a,i}}$  is $\kappa'$-quadratic at time $\tau_0$, and satisfies the orthogonality condition
\begin{equation}\label{orth_again}
 \Big\langle   v^{\mathcal{M}^{a}_i}_\cC(\tau_0) +\frac{ y^2-4}{\sqrt{8}|\tau_0|},y^2-4\Big\rangle_\cH=0.
 \end{equation}

\begin{claim}[compactness]\label{claim_boundedness} We have $\limsup_{i\to \infty}\sup_a \left(|\beta_{a,i}|+|\gamma_{a,i}|\right)<\infty$.
\end{claim}
\begin{proof}[Proof of the claim]To begin with, since the transformed flow $\mathcal{M}^{a}_i$ is $\kappa'$-quadratic at time $\tau_0$, by comparison with the round shrinking sphere and the round shrinking bubble-sheet we see that the absolute value of its (unrescaled) extinction time is bounded by some $c=c(\tau_0)$. Together with the fact that the untransformed flow $\mathcal{M}_{a,i}$ becomes extinct at (unrescaled) time $0$, this yields
\begin{equation}\label{bound_beta}
|\beta_{a,i}|\leq c.
\end{equation}
To establish the bound for $\gamma_{a,i}$, for any given $i\gg 1$ and $a$, we consider the flow $\widetilde{\mathcal{M}}=\mathcal{M}_{a,i}^{\beta_{a,i},0}$ and analyze Huisken's monotone quantity \cite{Huisken_monotonicity},
\begin{equation}
\Theta(r) = \int_{\widetilde{M}_{-r^2}} \frac{1}{4\pi r^2} e^{-\frac{|p|^2}{4r^2}} dA(p),
\end{equation}
at dyadic annuli of scales $r_j=2^{j}$, where $j\in\mathbb{Z}$.
By the equality case of the monotonicity formula, if $\Theta(r_{j+1})-\Theta(r_{j-1})=0$ then $\widetilde{\mathcal{M}}$ is selfsimilarly shrinking  at scale $r_j$. Upgrading this to a quantitative rigidity statement, via a standard contradiction argument similarly as in \cite[Proof of Lemma 3.2]{CHN_stratification}, for any $\eps>0$ we can find a $\delta>0$ such that if $\Theta(r_{j+1})-\Theta(r_{j-1})<\delta$ then $\widetilde{\mathcal{M}}$ is $\eps$-close at scale $r_j$ to a selfsimilarly shrinking noncollapsed hypersurface. Recall also that the only selfsimilarly shrinking noncollapsed hypersurfaces in $\mathbb{R}^4$ are the flat plane, the round shrinking sphere, and the round shrinking neck and the round shrinking bubble-sheet, and that the value of Huisken's quantity for these four solutions are ordered by size.
Moreover, note that by \eqref{Huisken_time_minus1} and \eqref{bound_beta} there is some $R_0<\infty$, independent of $i\gg 1$ and $a$, such that
\be
\Theta(R_0) \geq \frac{1}{3}\Theta_{S^1\times \mathbb{R}^2}+\frac{2}{3}\Theta_{S^2\times \mathbb{R}}.
\ee 
Hence, by quantitative differentiation, similarly as in \cite[Proof of Lemma 3.2]{CHN_stratification}, given any $\eps>0$ we can find an integer $J<\infty$, independent of $i\gg 1$ and $a$, such that for all $j\geq J$ the flow $\widetilde{\mathcal{M}}$ is $\eps$-close at scale $r_j$ to the round shrinking bubble-sheet. Now, choosing $\eps=\eps(\tau_0)$ small enough,
if $\gamma_{a,i}$ was very negative, then we would obtain a contradiction with the orthogonality condition \eqref{orth_again}. This proves that $\gamma_{a,i}$ is bounded below.\\
Finally, if $\gamma_{a,i}$ was very large, then by a similar argument the flow $\mathcal{M}^a_i$ at (renormalized) time $\tau_0$ would be very close to a blowup limit of $\mathcal{M}_{a,i}$, which would again violate \eqref{orth_again}. This shows that $\gamma_{a,i}$ is also bounded above, and thus concludes the proof of the claim. 
\end{proof}

Now, for each fixed $i\gg 1$, consider the largest interval $[a_i,b_i]$ containing $1/2$ such that for every $a\in [a_i,b_i]$  there exist some $\beta_{a,i},\gamma_{a,i}$ such that the transformed flow $\mathcal{M}^a_i:=\mathcal{M}_{a,i}^{\beta_{a,i},\gamma_{a,i}}$ 
\begin{enumerate}[(i)]
\item  is $\kappa'$-quadratic at time $\tau_0$,\label{cont_cond1}
\item satisfies the orthogonality condition \eqref{orth_again},
\item and we have that $\mathcal{R}(\mathcal{M}^a_i)\in[r_{-},r_{+}]$.\footnote{To be clear, the definitions of $\kappa'$-quadraticity and the width ratio map $\mathcal{R}$ for ellipsoidal flows are verbatim the same as for ancient bubble-sheet ovals.}
\end{enumerate}
Note that such a largest interval indeed exists thanks to Claim \ref{claim_boundedness} (compactness), and is nonempty since it contains $a=1/2$ thanks to \cite{DH_ovals}. 
Also, note that by the $\mathbb{Z}_2$-symmetry from swapping $y_1$ and $y_2$, we have
\begin{equation}
b_i=1-a_i.
\end{equation}
Moreover, observe that
\begin{equation}
a_i>0,
\end{equation}
since $T_i$ and $\tau_0$ are fixed and thus, remembering also Claim \ref{claim_boundedness} (compactness),  for $a$ very close to $0$ the orthogonally condition  \eqref{orth_again} cannot hold.

In general it is not obvious whether or not $\mathcal{M}^a_i$ depends continuously on $a$, since the parameters  $\beta_{a,i},\gamma_{a,i}$ might be nonunique. However, fortunately we can locally construct a continuous family satisfying the orthogonality condition \eqref{orth_again} as follows. Given any $\bar{a}_i\in[a_i,b_i]$, note that $\mathcal{M}^{\bar{a}_i}_i$ is an ellipsoidal flow with initial condition 
\begin{equation}
\widetilde{E}_{i}=E(\bar{a}_i,\ell_i,e^{\overline{\gamma}_i/2}R_i)\,\,\textrm{at time}\,\, \widetilde{T}_i=e^{\overline{\gamma}_i}T_i +\overline{\beta}_i,
\end{equation}
where we abbreviated
\begin{equation}
\overline{\beta}_i=\beta_{\bar{a}_i,i}, \quad\overline{\gamma}_i=\gamma_{\bar{a}_i,i}.
\end{equation}
Now, by Theorem \ref{point_strong} (strong $\kappa$-quadraticity) and Corollary \ref{cor_full_rank} (full rank), for $i$ large enough $\mathcal{M}^{\bar{a}_i}_i$ is $\frac{3}{2}\kappa$-quadratic between $\tau_0$ and $2\tau_0$.
We consider 
\begin{equation}\label{def_tilde}
\widetilde{\mathcal{M}}_{a,i}:=\mathcal{M}_{a,i}^{\bar{\beta}_i,\bar{\gamma}_i}.
\end{equation}
Note that $\widetilde{\mathcal{M}}_{a,i}$ is an ellipsoidal flow with initial condition 
\begin{equation}
\widetilde{E}_{a,i}=E(a,\ell_i,e^{\gamma_{\bar{a}_i,i}/2}R_i)\,\,\textrm{at time}\,\, \widetilde{T}_i,
\end{equation}
and that it is $2\kappa$-quadratic between $\tau_0$ and $2\tau_0$, provided $a$ is sufficiently close to $\bar{a}_i$.
Hence, applying Proposition \ref{S_property} (transformation map) in a neighborhood of $\bar{a}_i$  we get a continuous family 
\begin{equation}
a\mapsto \widetilde{\mathcal{M}}_{a,i}^{\widetilde{\beta}_{a,i},\widetilde{\gamma}_{a,i}},
\end{equation}
that satisfies the the orthogonality condition \eqref{orth_again}, where
 we abbreviated
\begin{equation}
\widetilde{\beta}_{a,i}:=\beta(\widetilde{\mathcal{M}}_{a,i}),\quad \widetilde{\gamma}_{a,i}:= \gamma(\widetilde{\mathcal{M}}_{a,i}).
\end{equation}
Finally, remembering \eqref{def_tilde} we can rewrite this continuous family as
\begin{equation}
a\mapsto \mathcal{M}_{a,i}^{\overline{\beta}_i+e^{\overline{\gamma}_i}\widetilde{\beta}_{a,i},\overline{\gamma}_i+\widetilde{\gamma}_{a,i}}.
\end{equation}
In particular, considering this continuous family around $\bar{a}_i=1/2$ it follows that
\begin{equation}
a_i<1/2.
\end{equation}

\begin{claim}[saturation]\label{claim_endpoints} For all large $i$ we have
\begin{equation}\label{cont_claim_eq}
\mathcal{R}(\mathcal{M}^{a_i}_i)\in \{r_{-},r_{+}\}\, .
\end{equation}
\end{claim}

\begin{proof}[Proof of the claim]
By definition of $a_i$ and the above construction of a continuous family near $\bar{a}_i=a_i$, either condition (i) or condition (iii) must be saturated. Suppose towards a contradiction condition (i) is saturated for increasingly high values of $i$, i.e. that at least one of the weak inequalities
\begin{equation}\label{mu_quad_back_not}
\left\|v^{\mathcal{M}_i^{a_i}}_{\cC}(y,\varphi,\tau_{0})-\sqrt{2}+\frac{y^2-4}{\sqrt{8}|\tau_{0}|}\right\|_\cH \leq \frac{\kappa'}{|\tau_{0}|},
\end{equation}
and
\begin{equation}\label{mu_quad_rad_not}
\sup_{\tau\in [2\tau_0,\tau_0]} |\tau|^{1/50}\big\|v^{\mathcal{M}_i^{a_i}}(\cdot,\tau)-\sqrt{2}]\big\|_{C^4(B(0,2|\tau|^{1/100})}  \leq 1,
\end{equation}
is an equality. After passing to a subsequence the $\mathcal{M}^{a_i}_i$ converge to an ancient noncollapsed flow $\mathcal{M}$,  whose tangent flow at $-\infty$ is given by \eqref{bubble-sheet_tangent_intro}, which is $\mathrm{SO}(2)$-symmetric in the $x_3x_4$-plane centered at the origin, and which satisfies the inequalities \eqref{mu_quad_back_not} and \eqref{mu_quad_rad_not} as well as the centering conditions
\begin{equation}
\fp_{+}(v_{\cC}^{\mathcal{M}}(\tau_0)-\sqrt{2})=0,
\end{equation}
and
\begin{equation}\label{cent_ysq}
\Big\langle  v^{\mathcal{M}}_\cC +\frac{ y^2-4}{\sqrt{8}|\tau_0|}, y^2-4\Big\rangle_\cH=0.
\end{equation}
Thus, by Corollary \ref{cor_full_rank} (full rank) our limit $\mathcal{M}$ is a bubble-sheet oval, and by Theorem \ref{point_strong} (strong $\kappa$-quadraticity), it is strongly $\kappa$-quadratic from time $\tau_{0}$. In particular, $\rho(\tau)=|\tau|^{1/10}$ is an admissible graphical radius function for $\tau\leq \tau_0$, so inequality \eqref{mu_quad_rad_not} is a strict inequality for $i$ large enough. Thus, it must be the case that
\begin{equation}\label{mu_quad_back__realy_not}
\left\|v^{\mathcal{M}}_{\cC}(y,\varphi,\tau_{0})-\sqrt{2}+\frac{y^2-4}{\sqrt{8}|\tau_{0}|}\right\|_\cH = \frac{\kappa'}{|\tau_{0}|}.
\end{equation}
On the other hand, provided $\tau_0\leq\tau_\ast$ is sufficiently negative, by Lemma \ref{quant_MZ} (quantitative Merle-Zaag type estimate) we have
\begin{equation}
\big\|\fp_{-}(v_{\cC}^{\mathcal{M}}(\tau_{0}))\big\|_\cH \\ \leq \frac{\kappa'}{100|\tau_0|},
\end{equation}
and by the orthogonality condition from equation \eqref{cent_ysq}, together with the facts that $\langle v_\cC^{\mathcal{M}},y^2\sin(2\varphi)\rangle_\cH=0$ and $\mathcal{R}(\mathcal{M})\in[r_{-},r_{+}]$,  we have
\begin{equation}\label{still_in_int}
\left\|\fp_{0}(v_{\cC}^{\mathcal{M}}(\tau_{0}))-\sqrt{2}+\frac{y^2-4}{\sqrt{8}|\tau_{0}|}\right\|_\cH \leq \frac{10\delta'}{|\tau_0|}.
\end{equation}
Since $\delta'\leq\kappa'/100$, this contradicts \eqref{mu_quad_back__realy_not}, and thus proves the claim.
\end{proof}

To conclude, by Claim \ref{claim_endpoints} (saturation) and the $\mathbb{Z}_2$-symmetry from swapping $y_1$ and $y_2$ it must be the case that
\begin{equation}
\left\{\mathcal{R}(\mathcal{M}^{a_i}_i),   \mathcal{R}(\mathcal{M}^{b_i}_i)\right\}=\left\{ r_{-},r_{+}\right\}.
\end{equation}
Hence, passing to subsequential limits, we get ancient noncollapsed mean curvature flows $\mathcal{M}_\pm$,  whose tangent flow at $-\infty$ is given by \eqref{bubble-sheet_tangent_intro}, that are $\mathrm{SO}(2)$-symmetric in the $x_3x_4$-plane centered at the origin, and satisfy
\be
\Big\langle  v^{\mathcal{M}_\pm}_\cC +\frac{ y^2-4}{\sqrt{8}|\tau_0|}, y^2-4\Big\rangle_\cH=0\quad\mathrm{and}\quad \mathcal{R}(\mathcal{M}_\pm)=r_\pm,
\ee
and that are $\kappa'$-quadratic at time $\tau_0$. In particular, by Corollary \ref{cor_full_rank} (full rank) the flows $\mathcal{M}_\pm$ are bubble-sheet ovals, and hence, remembering also that they are obtained as limits of transformed ellipsoidal flows, and that $\kappa'<\kappa$, they belong to the class $\mathcal{A}'_{\kappa}(\tau_0)$.  We have thus showed that the image of $\mathcal{R}:\mathcal{A}'_\kappa(\tau_0) \to \mathbb{R}$ contains the points $r_\pm$. This proves the theorem.
\end{proof}

\medskip

\subsection{Conclusion of the proof}\label{sec_conclusion}
In this subsection, we conclude the proof of the main theorem.

Given any bubble-sheet oval $\mathcal{M}=\{M_t\}$ in $\mathbb{R}^4$ (with coordinates chosen as usual) and  parameters $\alpha\in \mathbb{R}^2, \beta\in \mathbb{R}, \gamma\in \mathbb{R}$ and $\phi\in [0, 2\pi)$, we set
\begin{equation}
\mathcal{M}^{\alpha, \beta, \gamma, \phi}:=\{e^{\gamma/2} R_{\phi}(M_{e^{-\gamma}(t-\beta)}-\alpha)\},
\end{equation}
where the translation by $\alpha$ is understood to be in the $x_1x_2$-plane, and $R_\phi$ denotes rotation by $\phi$ in the $x_1x_2$-plane centered at the origin.

\begin{proposition}[orthogonality]\label{prop_orthogonality}
For any bubble-sheet oval $\mathcal{M}$ in $\mathbb{R}^4$ and any $\kappa>0$, there exists a constant $\tau_{\ast}={\tau}_\ast(\mathcal{M},\kappa)>-\infty$ with the following significance. For every $\tau_0\leq {\tau}_{\ast}$, there exist $\alpha\in \mathbb{R}^2, \beta\in \mathbb{R}, \gamma\in \mathbb{R}$ and $\phi\in [0, 2\pi)$, such that the truncated renormalized profile function $v^{\alpha, \beta, \gamma, \phi}_{\cC}$ of the transformed flow $\mathcal{M}^{\alpha, \beta, \gamma, \phi}$ satisfies
\begin{equation}
\qquad \Big\langle v^{\alpha, \beta, \gamma, \phi}_{\cC}(\tau_0), y^2\sin(2\varphi)\Big\rangle_\cH=0,\quad \Big\langle v^{\alpha, \beta, \gamma, \phi}_{\cC}(\tau_0) +\frac{y^2-4}{\sqrt{8}|\tau_0|},y^2-4\Big\rangle_\cH=0,
\end{equation}
and such that $\mathcal{M}^{\alpha, \beta, \gamma, \phi}$ is $\kappa$-quadratic at time $\tau_0$, in particular
\begin{equation}
\fp_+ \big(v^{\alpha, \beta, \gamma, \phi}_\cC(\tau_0)-\sqrt{2}\big)=0.
\end{equation}
\end{proposition}

\begin{proof}
We will solve four equations using a degree argument similarly as in \cite[Section 7]{ADS2}, and solve the remaining fifth equation using basic linear algebra. For convenience, we set
\begin{equation}\label{new_vars_scale}
a=e^{\tau/2}\alpha,\quad  b= \sqrt{1+\beta e^{\tau}}-1,\quad \Gamma=\frac{\gamma-\ln (1+\beta e^{\tau})}{\tau}.
\end{equation}
Then
\begin{align}\label{transofrm_scale}
     v^{\alpha, \beta, \gamma, \phi}(y,\tau)=(1+b)v\left(\frac{R_{-\phi}y-a}{1+b},(1+\Gamma)\tau\right).
\end{align}
Our first goal is, given any $\phi\in [0,2\pi)$, to find  a suitable zero of the map
\begin{equation}\label{Psi=0}
\Psi^\phi(a, b, \Gamma)=\left(\begin{array}{c}
     \Big\langle  1,v^{ab\Gamma \phi}_\cC-\sqrt{2} \Big\rangle_\cH \\
      \Big\langle y\cos\varphi, v^{ab\Gamma \phi}_\cC-\sqrt{2} \Big\rangle_\cH\\
       \Big\langle y\sin\varphi, v^{ab\Gamma \phi}_\cC-\sqrt{2}\Big\rangle_\cH\\
        \Big\langle  y^2-4, v^{ab\Gamma \phi}_\cC +\frac{y^2-4}{\sqrt{8}|\tau|}\Big\rangle_\cH
            \end{array}
    \right).
\end{equation}
To this end, observe that the vector space spanned by the eigenfunctions
\begin{equation}
\psi_0=1,\quad \psi_1= y\cos\varphi,\quad \psi_2=y\sin\varphi,
\end{equation}
as well as the eigenfunction
\begin{equation}
 \psi_3=y^2-4,
\end{equation}
are $\mathrm{SO}(2)$-invariant. Hence, if $(a, b, \Gamma)$ is a solution of $\Psi^0(a, b, \Gamma)=0$, then it actually solves $\Psi^\phi(a, b, \Gamma)=0$ for all $\phi\in [0,2\pi)$. To proceed, we need:
\begin{claim}[transformation estimate]\label{tranform_est}
For every $\kappa>0$ there exists $\tau_{\ast}=\tau_{\ast}(\mathcal{M},\kappa)>-\infty$, such that for all $\tau\leq\tau_{\ast}$ and all $(a,b, \Gamma)\in[-1,1]^2\times  [-1/|\tau|,1/|\tau|]\times [-1/2,1/2]$, we have
\begin{align}
&\left| \Big\langle\frac{\psi_0}{ \|\psi_0\|^{2}}, v^{ab\Gamma 0}_\cC-\sqrt{2} \Big\rangle_\cH-\left(\sqrt{2}b-\frac{|a|^2}{\sqrt{8}|\tau|(1+\Gamma)}\right)\right| \leq \frac{\kappa}{500|\tau|},\\ 
&\max_{i=1,2}\left| \Big\langle \frac{\psi_i}{ \|\psi_i\|^{2}}, v^{ab\Gamma 0}_\cC-\sqrt{2} \Big\rangle_\cH-\frac{2a_{i}}{\sqrt{8}|\tau|(1+\Gamma)}\right|\leq \frac{\kappa}{500|\tau|},\\
&\left| \Big\langle \frac{\psi_3}{ \|\psi_3\|^{2}}, v^{ab\Gamma 0}_\cC +\frac{y^2-4}{\sqrt{8}|\tau|}\Big\rangle_\cH-\frac{\Gamma}{\sqrt{8}|\tau|(1+\Gamma)}\right|\leq \frac{\kappa}{500|\tau|}.
\end{align}
\end{claim}

\begin{proof}  By \cite[Proposition 6.2 and Section 2.2]{DH_hearing_shape}, for any $\kappa'>0$, there exists $\tau_{\ast}>-\infty$, such that $\mathcal{M}$ is strongly $\kappa'$-quadratic from time $\tau_\ast$. In particular, for all $\tau\leq\tau_\ast$ we have
\begin{equation}
    \left\| v_{\cC}(y_{1}, y_{2},  \tau)-\sqrt{2}+\frac{y^2-4}{\sqrt{8}|\tau|}  \right\|_{\mathcal{H}}\leq \frac{\kappa'}{|\tau|} .
\end{equation}
Since $(a,b, \Gamma)$ are in the given rectangle, using \eqref{transofrm_scale} this implies
\begin{multline}
 v^{ab\Gamma0}=\sqrt{2} + \sqrt{2}b-\frac{|a|^2}{\sqrt{8}|\tau|(1+\Gamma)} 
-\frac{1}{\sqrt{8}|\tau|(1+\Gamma)}(y^2-4)\\
 +\frac{2a_{1}}{\sqrt{8}|\tau|(1+\Gamma)}y\cos\varphi+\frac{2a_{2}}{\sqrt{8}|\tau|(1+\Gamma)}y\sin\varphi+O\left(\frac{\kappa'}{|\tau|}\right)
\end{multline}
in $\mathcal{H}$-norm. Choosing $\kappa'\ll \kappa$, together with standard Gaussian tail estimates, the claim follows.
\end{proof}

Now, by Claim \ref{tranform_est} (transformation estimate), for $\kappa>0$ small enough, for every $\tau_0\leq \tau_\ast$, the map $\Psi^0$ restricted to the boundary of
\begin{equation}\label{Dk}
D_{\kappa}:=\left\{(a, b, \Gamma)\in\mathbb{R}^4\;:\; |a|^2+|\tau_0|^2b^2+\Gamma^2\leq 100 \kappa^2  \right\}
\end{equation}
is homotopic to the injective map
\begin{equation}
(a,b,\Gamma)\mapsto \left(
     \sqrt{2}\|\psi_0\|^2\, b-\frac{\|\psi_0\|^2\, |a|^2}{\sqrt{8}|\tau_0|(1+\Gamma)},
     \frac{2\|\psi_1\|^2\, a}{\sqrt{8}|\tau_0|(1+\Gamma)},
     \frac{\|\psi_3\|^2\,\Gamma}{\sqrt{8}|\tau_0|(1+\Gamma)}
    \right),
\end{equation}
through maps from $\partial D_\kappa$ to $\mathbb{R}^4\setminus\{0\}$. The map $\Psi^0$ from the full ball to $\mathbb{R}^4$ has therefore degree one. Hence, there exists $(a,b,\Gamma)\in  D_\kappa$ solving $\Psi^0(a,b,\Gamma)=0$. Remembering the above discussion, this actually solves
\begin{equation}
\Psi^\phi(a, b, \Gamma)=0\,\,\, \textrm{ for all  } \phi.
\end{equation}
Finally, observe that
\be
\big\langle y^2\sin(2\varphi), v^{ab\Gamma \phi}_\cC(\tau_0) \big\rangle_\cH=\big\langle y^2\sin(2(\varphi+\phi)), v^{ab\Gamma 0}_\cC(\tau_0) \big\rangle_\cH.
\ee
Hence, choosing $\phi\in [0, 2\pi)$ such that
\be
\left(\begin{array}{c}
\cos(2\phi)\\
\sin(2\phi)\end{array}\right) \cdot
\left(\begin{array}{c}
 \big\langle y^2\sin(2\varphi), v^{ab\Gamma 0}_\cC(\tau_0) \big\rangle_\cH\\
 \big\langle y^2\cos(2\varphi), v^{ab\Gamma 0}_\cC(\tau_0) \big\rangle_\cH\end{array}\right)=0
\ee
 we can obtain a solution of the remaining fifth equation
\begin{equation}
\big\langle y^2\sin(2\varphi), v^{ab\Gamma \phi}_\cC(\tau_0) \big\rangle_\cH =0.
\end{equation}
Noting also that by the above estimates the transformed flow $\mathcal{M}^{\alpha, \beta, \gamma, \phi}$ is $\kappa$-quadratic at time $\tau_0$, this finishes the proof of the proposition.
\end{proof}

We can now prove Theorem \ref{classification_theorem} (classification of bubble-sheet ovals):

\begin{proof}[Proof of Theorem \ref{classification_theorem}]
Let $\mathcal{M}^1$ be a bubble-sheet oval in $\mathbb{R}^4$. As always, we work in coordinates such that the tangent flow at $-\infty$ is given by \eqref{bubble-sheet_tangent_intro} and such that the $\mathrm{SO}(2)$-symmetry is in the $x_3x_4$-plane centered at the origin. By  Proposition \ref{prop_orthogonality} (orthogonality), given any $\kappa'>0$ and $\tau_0\leq\tau_\ast(\mathcal{M}^1,\kappa')$ after a suitable space-time transformation we can assume that the truncated renormalized profile function $v_{\cC}^{\mathcal{M}^1}$ of $\mathcal{M}^1$ satisfies
\begin{equation}
\Big\langle v_{\cC}^{\mathcal{M}^1}(\tau_0),  y^2\sin(2\varphi)\Big\rangle_\cH=0,\quad \Big\langle v_{\cC}^{\mathcal{M}^1}(\tau_0) +\frac{y^2-4}{\sqrt{8}|\tau_0|},y^2-4\Big\rangle_\cH=0,
\end{equation}
and such that $\mathcal{M}^1$ is $\kappa'$-quadratic at time $\tau_0$, in particular
\begin{equation}
\fp_+ \big(v_{\cC}^{\mathcal{M}^1}(\tau_0)-\sqrt{2}\big)=0.
\end{equation}
Let $\delta>0$, $\kappa>0$, $\tau_\ast>-\infty$ be constants such that Theorem \ref{prescribed_eccentricity_intro} (existence with prescribed spectral width ratio) and Theorem \ref{thm-spectral_uniqueness} (spectral uniqueness) apply. Possibly after decreasing $\tau_\ast$ and choosing $\kappa'\ll \delta$ we can arrange that
\begin{equation}
|\mathcal{R}(\mathcal{M}^1)-1|\leq \frac{\delta }{|\tau_0|}.
\end{equation}
Thus, by Theorem \ref{prescribed_eccentricity_intro} (existence with prescribed spectral width ratio) there exists a bubble-sheet 
oval $\mathcal{M}^2$, that up to transformation belongs to the class $\mathcal{A}^\circ$, that satisfies
\begin{equation}
\Big\langle v_{\cC}^{\mathcal{M}^2}(\tau_0) +\frac{y^2-4}{\sqrt{8}|\tau_0|},y^2-4\Big\rangle_\cH=0,
\end{equation}
and
\begin{equation}
 \mathcal{R}(\mathcal{M}^2)=\mathcal{R}(\mathcal{M}^1),
 \end{equation}
 and that is $\kappa$-quadratic at time $\tau_0$, in particular
 \begin{equation}
\fp_+ \big(v_{\cC}^{\mathcal{M}^2}(\tau_0)-\sqrt{2}\big)=0.
\end{equation}
Recall also that by construction we have
\begin{equation}
\Big\langle v_{\cC}^{\mathcal{M}^2}(\tau_0),  y^2\sin(2\varphi)\Big\rangle_\cH=0.
\end{equation}
Hence, we can apply Theorem \ref{thm-spectral_uniqueness} (spectral uniqueness) to conclude that the bubble-sheet ovals $\mathcal{M}^1$ and $\mathcal{M}^2$ coincide. We have thus shown that any bubble-sheet oval in $\mathbb{R}^{4}$ belongs, up to parabolic rescaling and space-time rigid motion, to the oval class  $\mathcal{A}^{\circ}$. This proves the theorem.
\end{proof}

 \medskip

Finally, let us explain how the corollaries follow:

\begin{proof}[{Proof of Corollary \ref{cor_blowup_limits}}]
By general theory \cite{White_nature,HaslhoferKleiner_meanconvex} all blowup limits of mean-convex mean curvature flow of 3-dimensional hypersurfaces are ancient noncollapsed flows in $\mathbb{R}^4$, in particular smooth and convex until they become extinct. Moreover, by the quoted references for any ancient noncollapsed flow $\mathcal{M}$ in $\mathbb{R}^4$ the tangent flow at $-\infty$ is either (i) a round shrinking sphere, or (ii) a round shrinking neck, or (iii) a round shrinking bubble-sheet, or (iv) a static plane.

In case (i), it follows from Huisken's classical roundness estimate \cite{Huisken_convex}, that the flow $\mathcal{M}$ itself is a round shrinking $S^3$. In case (iv), by the equality case of Huisken's monotonicity formula \cite{Huisken_monotonicity}, the flow $\mathcal{M}$ itself must be a static $\mathbb{R}^3$. In case (ii), by the work of Brendle-Choi \cite{BC1,BC2} and Angenent and the second and fifth author \cite{ADS2}, the flow $\mathcal{M}$ is, up to scaling and rigid motion, either the round shrinking $\mathbb{R}\times S^2$, or the rotationally symmetric 3d-bowl or the rotationally symmetric 3d-oval from \cite{White_nature}.

We can thus assume from now on that we are in case (iii). We consider three subcases according to the rank of the bubble-sheet matrix $Q$:

If $\mathrm{rk}(Q)=0$, then by \cite[Theorem 1.2]{DH_hearing_shape}, which has been obtained as a consequence of the no-ancient-wings theorem from \cite{CHH_wing}, the flow $\mathcal{M}$ must be either a round shrinking $\mathbb{R}^2\times S^1$ or a translating $\mathbb{R}\times$2d-bowl.

If $\mathrm{rk}(Q)=1$, and if $\mathcal{M}$ does not split off a line then it is strictly convex by Hamilton's tensor maximum principle \cite{Ham_pco}, and if 
 $\mathcal{M}$ does split off a line then up to scaling and rigid motion it is $\mathbb{R}\times$2d-oval by \cite{ADS2}.
 
If $\mathrm{rk}(Q)=2$, then by Theorem \ref{classification_theorem} (classification of bubble-sheet ovals), the flow $\mathcal{M}$ is, up to scaling and rigid motion, either the $\mathrm{O}(2)\times \mathrm{O}(2)$-symmetric 3d-oval from \cite{HaslhoferHershkovits_ancient},
or belongs to the one-parameter family of $\mathbb{Z}_2^2\times \mathrm{O}(2)$-symmetric 3d-ovals from \cite{DH_ovals}.
 \end{proof}
 
 \medskip
 
 \begin{proof}[{Proof of Corollary \ref{cor_moduli}}]
By \cite{CHH_wing,DH_hearing_shape} there are no compact solutions with $\mathrm{rk}(Q)=0$, and by a forthcoming result of Choi, Hershkovits and the fourth author \cite{CHH_ancient} there are no compact solutions with $\mathrm{rk}(Q)=1$, either.

Consider the canonical map $q: \mathcal{A}^\circ\to \mathcal{X}$ that sends any $\mathcal{M}\in\mathcal{A}^\circ$ to its equivalence class $[\mathcal{M}]\in\mathcal{X}$.  By Theorem \ref{classification_theorem} (classification of bubble-sheet ovals) and the above, the map $q$ is well defined and surjective. Suppose now that $q(\mathcal{M}_1)=q(\mathcal{M}_2)$. Then, by definition of our equivalence relation, $\mathcal{M}_2$ is obtained from $\mathcal{M}_1$ by a space-time rigid motion and parabolic dilation. Since all elements at of the class $\mathcal{A}^\circ$ become extinct at the origin at time zero, there cannot be any nontrivial space-time translation, and thanks to the condition that the Huisken density at time $-1$ equals $(4/e+\sqrt{2\pi/e})/2$ there cannot be any nontrivial parabolic dilation either. So $\mathcal{M}_2$ is obtained from $\mathcal{M}_1$ by a rotation $R\in \mathrm{SO}(4)$. Since the rotation fixes the tangent flow at $-\infty$, the rotation must be of the form $R=R_{12} R_{34}$, where $R_{ij}$ is a rotation in the $x_ix_j$-plane. Moreover, since $R_{34}$ acts trivially thanks to the $\mathrm{SO}(2)$-symmetry, we can assume $R=R_{12}$. To proceed, let us consider the spectral width matrix
\be
\mathcal{W}(\mathcal{M})=\left(\begin{array}{cc}
\big\langle v^{\mathcal{M}^{\beta,\gamma}}_\cC(\tau_0),  y^2\cos^2\varphi -2 \big\rangle_\cH & \big\langle  v^{\mathcal{M}^{\beta,\gamma}}_\cC(\tau_0),  y^2\sin(2\varphi) \big\rangle_\cH   \\
\big\langle  v^{\mathcal{M}^{\beta,\gamma}}_\cC(\tau_0),  y^2\sin(2\varphi) \big\rangle_\cH  & \big\langle  v^{\mathcal{M}^{\beta,\gamma}}_\cC(\tau_0),  y^2\sin^2\varphi -2 \big\rangle_\cH \end{array}\right),
\ee
where $\beta,\gamma$ are chosen such that the orthogonality conditions \eqref{orth_cond_prop} hold.
Note that $\mathcal{W}(\mathcal{M})$ is diagonal for any $\mathcal{M}\in \mathcal{A}^\circ$ thanks to the $\mathbb{Z}_2^2$-symmetry. If $\mathcal{W}(\mathcal{M}_1)$ is a multiple of the identity matrix, then applying Theorem \ref{thm-spectral_uniqueness} (spectral uniqueness) we see that $\mathcal{M}_1=\mathcal{M}_2$ is the unique $\mathrm{SO}(2)\times \mathrm{SO}(2)$-symmetric bubble-sheet oval, and if $\mathcal{W}(\mathcal{M}_1)$ has two distinct eigenvalues then, taking also into account again the $\mathbb{Z}_2^2$-symmetry, we see that either $\mathcal{M}_2=\mathcal{M}_1$ or $\mathcal{M}_2$ is obtained from $\mathcal{M}_1$ by a rotation by $\pi/2$. This shows that the induced map $\bar{q}: \mathcal{A}^\circ/\mathbb{Z}_2\to \mathcal{X}$ is bijective. Hence, by definition of the quotient topology, the map $\bar{q}$ is a homeomorphism.

To conclude, given any $\mathcal{M}\in \mathcal{A}^\circ$, we can choose suitable $\kappa>0$, $\tau_0>-\infty$ and $\beta,\gamma$ such that setting $\mathcal{M}':=\mathcal{M}^{\beta,\gamma}$ and  $r:=\mathcal{R}(\mathcal{M}')$ we have $\mathcal{M}'\in \mathcal{A}'_{\kappa}(\tau_0)$ and $(1+ \delta |\tau_0|^{-1})^{-1}< r<1+ \delta |\tau_0|^{-1}$. We equip $\mathcal{A}'_{\kappa}(\tau_0)$ with the smooth topology. Then, by Theorem \ref{thm-spectral_uniqueness} (spectral uniqueness) and Theorem \ref{prescribed_eccentricity_intro} (existence with prescribed spectral width ratio) there exists an open neighborhood $\mathcal{I}'\subset\mathcal{A}'_{\kappa}(\tau_0)$ of $\mathcal{M}'$ such that the restricted width ratio map $\mathcal{R}|_{\mathcal{I}'}$ is a homeomorphism from $\mathcal{I}'$ to an open interval containing $r$. Furthermore, possibly after decreasing the intervals, arguing as in the proof of Theorem \ref{prescribed_eccentricity_intro} (existence with prescribed spectral width ratio) we can find a homeomorphism from $\mathcal{I}'$ to an open neighborhood $\mathcal{I}\subset\mathcal{A}^\circ$ of $\mathcal{M}$. This shows that every $\mathcal{M}\in  \mathcal{A}^\circ$ has a neighborhood homeomoprhic to an open interval. Similarly, given any two elements $\mathcal{M}_1,\mathcal{M}_2\in \mathcal{A}^\circ$, we can choose suitable $\kappa>0$, $\tau_0>-\infty$ and $\beta_i,\gamma_i$ for $i=1,2$ such that setting $\mathcal{M}_i':=\mathcal{M}_i^{\beta_i,\gamma_i}$ and  $r_i:=\mathcal{R}(\mathcal{M}_i')$ we have $\mathcal{M}_i'\in \mathcal{A}'_{\kappa}(\tau_0)$ and $(1+ \delta |\tau_0|^{-1})^{-1}< r_i <1+ \delta |\tau_0|^{-1}$. We can assume without loss of generality that $r_1\leq r_2$. Then, by Theorem \ref{thm-spectral_uniqueness} (spectral uniqueness) and Theorem \ref{prescribed_eccentricity_intro} (existence with prescribed eccentricity) the map $[0,1]\ni s\mapsto [\mathcal{R}^{-1}(sr_1+(1-s)r_2)]\in \mathcal{X}$ is a continuous path in $\mathcal{X}$ from $[\mathcal{M}_1]$ to $[\mathcal{M}_2]$. This shows that $\mathcal{X}$ is connected. Observing also that $\mathcal{X}$ is Hausdorff and second countable, we thus conclude that ${\mathcal{X}}\approx \mathcal{A}^\circ/\mathbb{Z}_2$ is homeomorphic to a half-open interval.
\end{proof}
 
\bigskip

\bibliography{oval}

\bibliographystyle{alpha}

\vspace{10mm}

{\sc Beomjun Choi, Department of Mathematics, POSTECH, Gyengbuk, Korea}

{\sc Panagiota Daskalopoulos, Department of Mathematics, Columbia University, New York, USA}

{\sc Wenkui Du, Department of Mathematics, University of Toronto, Ontario, Canada}

{\sc Robert Haslhofer, Department of Mathematics, University of Toronto, Ontario, Canada}

{\sc Natasa Sesum, Department of Mathematics, Rutgers University, New Jersey, USA}

\vspace{5mm}

\emph{E-mail:} bchoi@postech.ac.kr, pdaskalo@math.columbia.edu,\\
wenkui.du@mail.utoronto.ca, roberth@math.toronto.edu, natasas@rutgers.edu

\end{document}